\documentclass[11pt]{article}

\usepackage{amssymb,amsmath,amscd,amsthm,xy,enumitem,pstricks,pst-node,mathrsfs}
\xyoption{all}

\newtheorem{thm}{Theorem}[section]
\newtheorem{prp}[thm]{Proposition}
\newtheorem{lmm}[thm]{Lemma}
\newtheorem{lmmdfn}[thm]{Lemma-Definition}
\newtheorem{crl}[thm]{Corollary}

\theoremstyle{definition}
\newtheorem{dfn}[thm]{Definition}
\newtheorem{eg}[thm]{Example}

\theoremstyle{remark}
\newtheorem{rmk}[thm]{Remark}

\numberwithin{equation}{section}

\def\lra{\longrightarrow}
\def\Lra{\Longrightarrow}

\def\BE#1{\begin{equation}\label{#1}}
\def\EE{\end{equation}}
\def\lr#1{\langle#1\rangle}
\def\flr#1{\lfloor{#1}\rfloor}
\def\blr#1{\big\langle#1\big\rangle}
\def\ti#1{\tilde{#1}}
\def\wt#1{\widetilde{#1}}
\def\ov#1{\overline{#1}}
\def\eref#1{(\ref{#1})}
\def\tn#1{\textnormal{#1}}
\def\sf#1{\textsf{#1}}
\def\br#1{\breve{#1}}

\def\De{\Delta}
\def\Ga{\Gamma}
\def\La{\Lambda}
\def\Om{\Omega}
\def\Si{\Sigma}

\def\al{\alpha}
\def\be{\beta}

\def\eps{\epsilon}
\def\gm{\gamma}

\def\ka{\kappa}
\def\la{\lambda}
\def\na{\nabla}
\def\om{\omega}
\def\si{\sigma}
\def\th{\theta}

\def\vph{\varphi}
\def\vp{\varpi}
\def\vt{\vartheta}

\def\fB{\mathfrak B}

\def\bB{\mathbb B}
\def\C{\mathbb C}
\def\cC{\mathcal C}  
\def\cD{\mathcal{D}}
\def\bE{\mathbb E}

\def\bH{\mathbb H}
\def\cH{\mathcal H}

\def\bI{\mathbb I}
\def\fI{\mathfrak i}
\def\cJ{\mathcal J}
\def\fJ{\mathfrak j}

\def\bK{\mathbb K}

\def\cM{\mathcal M}
\def\fM{\mathfrak M}
\def\mf{\mathfrak M}

\def\cO{\mathcal O}
\def\cP{\mathcal P}
\def\P{\mathbb P}
\def\R{\mathbb{R}}
\def\bT{\mathbb{T}}

\def\cT{\mathcal T}

\def\Z{\mathbb{Z}}

\def\a{\mathbf a}
\def\fa{\mathfrak a}
\def\b{\mathbf b}
\def\k{\mathbf k}
\def\x{\mathbf x}
\def\fc{\mathfrak c}
\def\fd{\mathfrak d}
\def\ft{\mathfrak t}

\def\cl{\tn{cl}}
\def\cok{\textnormal{cok}}

\def\tnd{\textnormal{d}}
\def\DM{\tn{DM}}
\def\Ext{\tn{Ext}}
\def\ev{\tn{ev}}
\def\GL{\tn{GL}}
\def\Hom{\tn{Hom}}
\def\id{\textnormal{id}}

\def\ind{\textnormal{ind}}

\def\pt{\tn{pt}}
\def\PGL{\tn{PGL}}
\def\rk{\textnormal{rk}}
\def\sgn{\textnormal{sgn}}
\def\SO{\tn{SO}}

\def\top{\textnormal{top}}

\def\0{\mathbf 0}
\def\1{\mathbf 1}

\def\oI{\mathring{\mathbb{I}}}

\def\dbar{\bar\partial}
\def\prt{\partial}
\def\eset{\emptyset}
\def\i{\infty}
\def\bp{\bar\partial}

\begin{document}

\title{The moduli space of maps with crosscaps:\\
the relative signs of the natural automorphisms}
\author{Penka Georgieva and 
Aleksey Zinger\thanks{$^*$Partially supported by the IAS Fund for Math and NSF grants 
DMS 0635607 and 0846978}}
\date{\today}
\maketitle

\begin{abstract}
\noindent
Just as a symmetric surface with separating fixed locus halves into two 
oriented bordered surfaces, 
an arbitrary symmetric surface halves into two oriented symmetric half-surfaces,
i.e.~surfaces with crosscaps.
Motivated in part by the string theory view of real Gromov-Witten invariants,
we previously introduced moduli spaces of maps from surfaces with crosscaps, 
developed the relevant Fredholm theory, and resolved the orientability problem
in this setting.
In this paper, we determine the relative signs of the automorphisms
of these moduli spaces induced by interchanges of boundary components of the domain 
and by the anti-symplectic involution on the target manifold, 
without any global assumptions on the latter.
As immediate applications, we describe sufficient conditions for 
the moduli spaces of real genus~1 maps and for real maps 
with separating fixed
locus to be orientable; we treat the general genus~2+ case in a separate paper.
Our sign computations also lead to an extension of recent Floer-theoretic applications 
of anti-symplectic involutions 
and to a related reformulation of these results in a more natural~way.
\end{abstract}

\tableofcontents

\section{Introduction}
\label{intro_sec}

\noindent
The theory of $J$-holomorphic maps plays a prominent role in symplectic topology,
algebraic geometry, and string theory.
The foundational work of~\cite{Gr,Witten,McSa94,RT,FO,LT} has 
established the theory of (closed) Gromov-Witten invariants,
i.e.~counts of $J$-holomorphic maps from closed Riemann surfaces to symplectic manifolds.
In contrast, the theory of open and real Gromov-Witten invariants, i.e.~counts of
$J$-holomorphic maps from bordered Riemann surfaces with boundary mapping 
to a Lagrangian submanifold and of $J$-holomorphic maps from symmetric
Riemann surfaces commuting with the involutions on the domain and the target, 
has been under development over the past 10-15 years and still is today.\\

\noindent
The two main obstacles to defining the open invariants are the potential non-orientability 
of the moduli space and the existence of real codimension-one boundary strata.
The orientability problem in open Gromov-Witten theory is studied in
\cite{FOOOold, Sol} and is fully addressed 
in~\cite{Ge}, by specifying a procedure for locally orienting moduli spaces 
of open maps and conditions sufficient for the existence of a global orientation.
Some approaches \cite{Melissa, PSW, Teh} to dealing with the codimension-one boundary
have raised the issue of orientability in real Gromov-Witten theory.
Symmetric Riemann surfaces, however, have convoluted degenerations, making
the orientability of their moduli spaces difficult to study.
Physical considerations \cite{SV,AAHV,Wal}
suggest that oriented surfaces with crosscaps provide 
a suitable replacement for symmetric Riemann surfaces in real Gromov-Witten theory.
In~\cite{GZ}, we introduced moduli spaces of $J$-holomorphic maps from 
oriented surfaces with crosscaps, developed the necessary Fredholm theory,
and  fully addressed the orientability problem for these moduli spaces,
by specifying a procedure for locally orienting them
and conditions sufficient for the existence of a global orientation.
The last problem is related to the orientability problem in real Gromov-Witten theory
via the automorphisms of these moduli spaces induced by interchanges of boundary components 
of the domain and by the anti-symplectic involution on the target manifold.
The effect of the only such automorphism on the local system of orientations of 
the disk moduli space is computed in~\cite{Ge2}, without assuming the existence of 
any global orienting structure on~$X$,
as a step in defining a count of disk maps in a much wider collection of cases
than in \cite{Cho, Sol};
these counts correspond to real invariants of the target manifold 
for the standard anti-holomorphic involution on the sphere as defined in \cite{Wel1,Wel2}.\\

\noindent
In this paper, we apply the relative sign idea of~\cite{Ge2} to compute the change in 
local orientations of the moduli space of real $J$-holomorphic maps from 
an oriented symmetric half-surface~$(\Si,c)$  
under an interchange of boundary components of~$\Si$ and 
the action of an anti-symplectic involution~$\phi$ on the target manifold~$X$,
without any global assumptions on the latter; see Propositions~\ref{relsign_prp} 
and~\ref{transfsign_prp}.\footnote{Proposition~\ref{transfsign_prp}
also applies to the action of orientation-preserving diffeomorphisms of~$\Si$
on moduli spaces of 
open maps to a pair $(X,L)$ consisting of a symplectic manifold and a Lagrangian submanifold;
an anti-symplectic involution~$\phi$ on $X$ with $L\!=\!X^{\phi}$ 
is needed only for the action of orientation-reversing diffeomorphisms of~$\Si$.}
The special case of these propositions with  
the diffeomorphism of~$\Si$ fixing the components of~$\prt\Si$ 
includes orientability results of \cite{FOOO9, Melissa, Sol, Teh};
see Remark~\ref{Sol_rmk}.
Applying Propositions~\ref{relsign_prp} and~\ref{transfsign_prp},
we obtain sufficient conditions for the moduli spaces of 
real $J$-holomorphic maps from a symmetric surface $(\Si,\si)$ of genus~0 and~1, 
with any given involution on the domain, 
or from a symmetric surface of arbitrary genus with an involution with
a separating fixed locus to be orientable; 
see Theorems~\ref{g0orient_thm}, \ref{g1orient_thm}, and~\ref{seorient_thm}.
In \cite{GZ3}, we treat the orientability question for arbitrary symmetric surfaces.
In a future paper, we will study compactifications of the moduli spaces of maps with crosscaps
and use them to define real Gromov-Witten invariants in the style of~\cite{Wal}.
The disk case of Propositions~\ref{relsign_prp} and~\ref{transfsign_prp} immediately extends
and reformulates, in a  more natural way,  
the arguments of~\cite{FOOO9} utilizing anti-symplectic involutions on 
the target for Floer-theoretic applications; see Section~\ref{FOOO_subs}.
In addition to leading to orientability results in more general settings, 
the relative sign approach of~\cite{Ge2} does not involve a global orienting structure
coming from the ambient manifold and thus reduces the likelihood of computational mistakes
of the kind that did appear in earlier approaches to the orientability problem in real settings;
see the beginning of Section~\ref{main_sec} for details.\\

\noindent
An \textsf{involution} on a topological space (resp.~smooth manifold) $M$ 
is a homeomorphism (resp.~diffeomorphism) 
$c\!:M\!\lra\!M$ such that $c\!\circ\!c\!=\!\id_M$;
in particular, the identity map on~$M$ is an involution.
Let
$$M^c=\big\{x\!\in\!M\!:~c(x)\!=\!x\big\}$$
denote the fixed locus.
An involution~$c$ determines an action of~$\Z_2$ on~$M$;
we denote~by $H^*_c(M)$ the $\Z_2$-equivariant cohomology 
with $\Z_2$-coefficients; see Section~\ref{equivdfn_e}.
A \sf{conjugation} on a complex vector bundle $V\!\lra\!M$ 
\sf{lifting} an involution~$c$ is a vector bundle homomorphism 
$\ti{c}\!:V\!\lra\!V$ covering~$c$ (or equivalently 
a vector bundle homomorphism  $\ti{c}\!:V\!\lra\!c^*V$ covering~$\id_M$)
such that the restriction of~$\ti{c}$ to each fiber is anti-complex linear
and $\ti{c}\!\circ\!\ti{c}\!=\!\id_V$.
We denote~by
$$\La_{\C}^{\top}(V,\ti{c})=(\La_{\C}^{\top}V,\La_{\C}^{\top}\ti{c})$$
the top exterior power of $V$ over $\C$ with the induced conjugation and~by
$$w_i^{\ti{c}}(V)\in H^i_c(M)$$
the $i$-th $\Z_2$-equivariant Stiefel-Whitney class of~$V$.\\

\noindent
A \sf{symmetric surface} $(\hat\Si,\si)$ consists of a closed connected oriented smooth 
surface~$\hat\Si$ 
(manifold of real dimension~2) and an orientation-reversing involution 
$\si\!:\hat\Si\!\lra\!\hat\Si$.
There are two equivalence classes of orientation-reversing involutions on $S^2\!=\!\P^1$,
\BE{tauetadfn_e}\tau,\eta\!:\P^1\lra\P^1, \qquad 
\tau\big([u,v]\big)=[\bar{v},\bar{u}],\quad \eta\big([u,v]\big)=[-\bar{v},\bar{u}],\EE
and three equivalence classes of such involutions on 
$\bT\!=\!S^1\!\times\!S^1\subset\C\!\times\!\C$,
\BE{T2symm_e}\ft_0,\ft_1,\ft_2\!:\bT\lra \bT, \qquad
\ft_0(u,v)=(u,\bar{v}),\quad \ft_1(u,v)=(v,u),\quad \ft_2(u,v)=(-u,\bar{v});\EE
see \cite[Section~9]{AG}, for example.
The fixed loci of $\tau$ and $\ft_0$ are a circle separating $\P^1$ into two disks
interchanged by~$\tau$ and a pair of disjoint circles separating $\bT$ into two
annuli interchanged by~$\ft_0$, respectively.
The fixed loci of~$\eta$ and~$\ft_2$ are empty, with the quotients $\hat\Si/\si$ being $\R\P^2$
and the Klein bottle~$\bK$, respectively, while the fixed locus of~$\ft_1$ is one circle.\\

\noindent
Let $(X,\phi)$ be a smooth manifold with an involution.
If $(\hat\Si,\si)$ is a symmetric surface, 
a \sf{real map} 
$$F\!:(\hat\Si,\si)\lra(X,\phi)$$ 
is a map $F\!:\hat\Si\!\lra\!X$ such that $F\!\circ\!\si=\phi\!\circ\!F$.
We denote the space of smooth real maps from~$(\hat\Si,\si)$ to~$(X,\phi)$  by~$\fB(X)^{\phi,\si}$.
Since $\si$ is orientation-reversing,
$$\phi_*\big(F_*[\hat\Si]_{\Z}\big)=-F_*[\hat\Si]_{\Z}\in H_2(X;\Z)
\qquad\forall~F\in \fB(X)^{\phi,\si}\,,$$
where $[\hat\Si]_{\Z}\!\in\!H_2(\hat\Si;\Z)$ is the fundamental homology class of~$\hat\Si$.
For each $B\!\in\!H_2(X;\Z)$, let 
$$\fB(X,B)^{\phi,\si}=\big\{F\!\in\!\fB(X)^{\phi,\si}\!:\,F_*[\Si]_{\Z}\!=\!B\big\};$$
this space is empty unless $\phi_*[\hat\Si]_{\Z}\!=\!-[\hat\Si]_{\Z}$.\\

\noindent
For a symplectic manifold $(X,\om)$, we denote~by $\cJ_{\om}$
the space of $\om$-compatible almost complex structures on~$X$.
If $\phi$ is  an anti-symplectic involution on~$(X,\om)$,
let 
$$\cJ_{\phi}=\big\{J\!\in\!\cJ_{\om}\!:\,\phi^*J\!=\!-J\big\}.$$
For a genus~$\hat{g}$ symmetric surface~$(\hat\Si,\si)$, we similarly denote by $\cJ_{\si}$
the space of complex structures on~$\hat\Si$ compatible with the orientation such that 
$\si^*\fJ\!=\!-\fJ$.
For $J\!\in\!\cJ_{\phi}$, $\fJ\!\in\!\cJ_{\si}$, and
$F\!\in\!\fB(X)^{\phi,\si}$, let 
$$\dbar_{J,\fJ}F=\frac{1}{2}\big(\tnd F+J\circ\tnd F\!\circ\!\fJ\big).$$
If $J\!\in\!\cJ_{\phi}$ and $B\!\in\!H_2(X;\Z)$, let
\begin{equation*}\begin{split}
\fM(X,B;J)^{\phi,\si}= \big\{(F,\fJ)\!\in\!\fB(X,B)^{\phi,\si}\!\times\!\cJ_{\si}\,:
~\dbar_{J,\fJ}F\!=\!0\big\}\big/\!\!\sim
\end{split}\end{equation*}
be the moduli space of equivalence classes of degree~$B$ real $J$-holomorphic maps
from $(\hat\Si,\si)$ to~$(X,\phi)$; two $J$-holomorphic maps are equivalent in this space 
if they differ by a diffeomorphism of~$\hat\Si$ commuting with~$\si$.
By \cite{FO,LT} and \cite[Section~7]{Sol},
this moduli space comes with a natural Kuranishi structure;
we call the former \sf{orientable} if the latter~is.\\

\noindent
Mathematical considerations, such as \cite[Section~4]{Klein}, 
\cite[Section~5]{Se91}, \cite[Sections~3,4]{Melissa}, \cite[Section~1.5]{PSW},
and \cite[Section~3]{Teh}, suggest that the map counts arising from involutions~$\si$
 of different topological types should be combined to get well-defined invariants,
as there is a path through one-nodal degenerations between any two 
involutions of different topological types on the same closed oriented surface~$\hat\Si$.
For this reason, for each $\hat{g}\!\in\!\Z^{\ge0}$
we define
\BE{totalspace_e}\fM_{\hat{g}}(X,B;J)^{\phi}
=\bigsqcup_{\si}\fM(X,B;J)^{\phi,\si}\,,\EE
where the disjoint union is taken over representatives for the equivalence classes
of orientation-reversing involutions on a genus~$\hat{g}$ closed connected oriented 
smooth surface~$\hat\Si$.
There are $\left\lfloor\frac{3\hat{g}+4}{2}\right\rfloor$ such classes;
see \cite[Corollary~1.4]{Nat}.
In order to define real invariants, it is essential to study the orientability
of these moduli spaces.
The next two theorems, which are specializations of 
Corollaries~\ref{seorient_crl} and~\ref{etaorient_crl}
of Propositions~\ref{relsign_prp} and~\ref{transfsign_prp},  
describe topological conditions on $(X,\om,\phi)$ insuring that 
the moduli spaces~\eref{totalspace_e} with $\hat{g}\!=\!0,1$ are orientable. 
The $\hat{g}\!\ge\!2$ case is treated in~\cite{GZ3}.

\begin{thm}\label{g0orient_thm}
Let $(X,\om)$ be a symplectic manifold with an anti-symplectic involution~$\phi$,
$J\!\in\!\cJ_{\phi}$, and $B\!\in\!H_2(X;\Z)$.
If $X^{\phi}\!\subset\!X$ is orientable, there exists $\vp\!\in\!H^2(X;\Z_2)$ such~that 
\BE{flipeven0_e} w_2(TX^{\phi})=\vp|_{X^{\phi}} 
\qquad\hbox{and}\qquad 
\frac12\lr{c_1(TX),B}+\lr{\vp,B}\in 2\Z\,,\EE
and $w_2^{\La_{\C}^{\top}\tnd\phi}(\La_{\C}^{\top}TX)\!=\!\ka_{\phi}^2$ for some $\ka_{\phi}\!\in\!H_{\phi}^1(X)$,
then the moduli space $\fM_0(X,B;J)^{\phi}$ is orientable.
\end{thm}

\begin{thm}\label{g1orient_thm}
Let $(X,\om)$ be a symplectic $2n$-manifold with an anti-symplectic involution~$\phi$,
$J\!\in\!\cJ_{\phi}$, and $B\!\in\!H_2(X;\Z)$.
If $n$ is odd,  $X^{\phi}\!\subset\!X$ is orientable,
there exists a real bundle pair 
$(E,\ti\phi)\!\lra\!(X,\phi)$ such~that 
\BE{flipeven1_e} w_2(TX^{\phi})=w_1(E^{\ti\phi})^2
\qquad\hbox{and}\qquad 
\frac12\lr{c_1(TX),B}+\lr{c_1(E),B}\in 2\Z\,,\EE
and $w_2^{\La_{\C}^{\top}\tnd\phi}(\La_{\C}^{\top}TX)\!=\!0$,  
then the moduli space $\fM_1(X,B;J)^{\phi}$ is orientable.
\end{thm}

\noindent
The first two conditions in Theorem~\ref{g0orient_thm} insure that 
the moduli space $\fM(X,B;J)^{\phi,\tau}$ is orientable.
A special case of this result, which is also captured by \cite[Proposition~5.1]{Sol}, 
can be seen as corresponding to \cite[Theorem~1.3]{FOOO9}, 
which involves a less readily checkable condition.
In fact, this part of Theorem~\ref{g0orient_thm} immediately establishes 
\cite[Proposition~3.14]{FOOO9}; see Corollary~\ref{FOOO9prp314_crl}.
The $\Si\!=\!D^2$ case of Corollary~\ref{seorient_crl}, which implies the $\tau$-orientability
part of Theorem~\ref{g0orient_thm}, relaxes the second condition 
beyond what the assumptions of \cite[Proposition~5.1]{Sol} allow.
The last condition in Theorem~\ref{g0orient_thm} insures that 
the moduli space $\fM(X,B;J)^{\phi,\eta}$ is orientable.
By Corollary~\ref{SQrt_crl}, it is satisfied if either $\pi_1(X)\!=\!0$
and $w_2(TX)\!=\!0$ or $\La_{\C}^{\top}(TX,\tnd\phi)$ admits a \sf{real square root}, 
i.e.~there exist a rank~1 real bundle pair $(L,\ti\phi)\!\lra\!(X,\phi)$
and  an isomorphism of real bundle pairs
$$ \La_{\C}^{\top}(TX,\tnd\phi)\approx (L,\ti\phi)^{\otimes2}\,.$$
As explained in \cite[Section~1]{GZ}, this result on the orientability 
of the $\eta$-moduli space is readily implied by \cite[Theorem~1.1]{GZ}
and contains \cite[Theorem~1.3]{Teh}.\\

\noindent
The first three conditions in Theorem~\ref{g1orient_thm},
which are relaxed in Corollary~\ref{seorient_crl},
insure that the moduli space $\fM(X,B;J)^{\phi,\ft_0}$ is orientable.
The last three conditions insure that the moduli space $\fM(X,B;J)^{\phi,\ft_1}$ is orientable;
they are relaxed in Corollary~\ref{etaorient_crl}.
The first and last conditions in Theorem~\ref{g1orient_thm},
which are also relaxed in Corollary~\ref{etaorient_crl}, imply that 
the moduli space $\fM(X,B;J)^{\phi,\ft_2}$ is orientable.
By Corollary~\ref{SQrt_crl}, the condition on the equivariant
Stiefel-Whitney class in Theorem~\ref{g1orient_thm} is satisfied
if $\La_{\C}^{\top}(TX,\tnd\phi)$ admits a real square root.
Special cases of Theorem~\ref{g1orient_thm} have been obtained independently 
in \cite{Remi,Teh3} using different methods.\\

\noindent
The paradigmatic example of a symplectic manifold with an anti-symplectic involution
is the complex projective space~$\P^{n-1}$ with the standard 
Fubini-Study symplectic forms~$\om_n$ and the standard conjugation
$$\tau_n\!:\P^{n-1}\lra\P^{n-1}, \qquad [Z_1,\ldots,Z_n]\lra[\bar{Z}_1,\ldots,\bar{Z}_n];$$
the fixed point locus of this involution is~$\R\P^{n-1}$.
An odd-dimensional projective space also admits an involution without fixed points,
$$\eta_{2m}\!: \P^{2m-1}\lra\P^{2m-1},\qquad
[Z_1,Z_2,\ldots,Z_{2m-1},Z_{2m}]\lra 
\big[-\bar{Z}_2,\bar{Z}_1,\ldots,-\bar{Z}_{2m},\bar{Z}_{2m-1}\big].$$
If $k\!\ge\!0$, $\a\equiv(a_1,\ldots,a_k)\in(\Z^+)^k$,
and $X_{n;\a}\!\subset\!\P^{n-1}$ is a complete intersection of multi-degree~$\a$
preserved by~$\tau_n$,  $\tau_{n;\a}\!\equiv\!\tau_n|_{X_{n;\a}}$
is an anti-symplectic involution on $X_{n;\a}$ with respect to the symplectic form
$\om_{n;\a}\!\equiv\!\om_n|_{X_{n;\a}}$. 
Similarly, if $X_{2m;\a}\!\subset\!\P^{2m-1}$ is preserved by~$\eta_{2m}$,
$\eta_{2m;\a}\!\equiv\!\eta_{2m}|_{X_{2m;\a}}$
is an anti-symplectic involution on $X_{2m;\a}$ with respect to the symplectic form
$\om_{2m;\a}\!\equiv\!\om_{2m}|_{X_{2m;\a}}$.

\begin{crl}\label{CIorient_crl}
Let $n\!\in\!\Z^+$, $k\!\in\!\Z^{\ge0}$, $\a\!\equiv\!(a_1,\ldots,a_k)\!\in\!(\Z^+)^k$, 
and $B\!\in\!H_2(X_{n;\a};\Z)$. 
\begin{enumerate}[label=(\arabic*),leftmargin=*]
\item If $X_{n;\a}\!\subset\!\P^{n-1}$ is a complete intersection of multi-degree~$\a$
preserved by~$\tau_n$,
$$\sum_{i=1}^ka_i\equiv n \mod2, \qquad\hbox{and}\qquad
\sum_{i=1}^ka_i^2\equiv \sum_{i=1}^ka_i \mod4,$$
then the moduli space  $\fM_0(X_{n;\a},B;J)^{\tau_{n;\a}}$
is orientable for every $J\!\in\cJ_{\tau_{n;\a}}$.
If in addition $n\!-\!k$ is even, then the moduli space $\fM_1(X_{n;\a},B;J)^{\tau_{n;\a}}$
is also orientable.
\item If $n\!=\!2m$, $X_{n;\a}$ is a complete intersection of multi-degree~$\a$
preserved by $\eta_{2m}$, and 
$$a_1\!+\!\ldots\!+\!a_k\equiv n \mod2,$$
then the moduli space  $\fM_0(X_{n;\a},B;J)^{\eta_{n;\a}}$
is orientable for every $J\!\in\cJ_{\eta_{n;\a}}$.
If in addition $k$ is even and $a_1\!+\!\ldots\!+\!a_k\equiv n \mod4$,
the moduli space $\fM_1(X_{n;\a},B;J)^{\eta_{n;\a}}$ is also orientable.
\end{enumerate}
\end{crl}

\noindent
This corollary follows immediately from Theorems~\ref{g0orient_thm} and~\ref{g1orient_thm}.
If $a_1\!+\!\ldots\!+\!a_k\!\equiv\!n\mod4$, we take $\vp\!=\!0$ and
$E$ to be the trivial rank~0 vector bundle for the purposes of applying these
two theorems.
Otherwise, we take 
$$\vp=c_1(\cO_{\P^{n-1}}(1))|_{X_{n;\a}}
\qquad\hbox{and}\qquad
(E,\ti\phi)=\big(\cO_{\P^{n-1}}(1),\ti\tau_{n-1}\big)\big|_{X_{n;\a}}\,,$$
where $\cO_{\P^{n-1}}(1)\!\lra\!\P^{n-1}$ is the hyperplane line bundle
with the involution~$\ti\tau_n$ canonically induced by~$\tau_n$.
The idea of introducing an additional bundle~$(E,\ti\phi)$ in Theorem~\ref{g1orient_thm}
is directly motivated by~\cite{Ge2} and provides a way of bypassing the requirements
that $X^{\phi}$ be spin and $\lr{c_1(TX),B}$ be divisible by~4, 
which appear in the genus~0 orientability results in~\cite{FOOO9} and 
the genus~1 orientability results of \cite{Remi,Teh3}.
It can also be applied in Theorem~\ref{g0orient_thm}, instead of using~$\vp$,
but the latter is the more customary relative spin condition on~$TX^{\phi}$.
In \cite{GZ3}, we extend Corollary~\ref{CIorient_crl} to arbitrary genus.\\

\noindent
Theorems~\ref{g0orient_thm} and~\ref{g1orient_thm} are obtained by reducing 
the orientability problem for the moduli spaces 
$\fM(X,B;J)^{\phi,\si}$ of maps from $(\hat\Si,\si)$ to orientability questions 
about the moduli spaces  $\fM(X,B;J)^{\phi,c}$ of maps from a corresponding 
\sf{oriented sh-surface} $(\Si,c)$, where $\Si$ is a bordered oriented smooth
surface and $c\!:\prt\Si\!\lra\!\prt\Si$ is an orientation-preserving involution;
see Section~\ref{analys_subs}.
The former moduli space is the  quotient of 
the latter moduli space by automorphisms, which include the involution
\BE{conjaut_e} [\fJ,f]\lra [-\fc_{\Si}^*\fJ,\phi\circ f\circ \fc_{\Si}],\EE
where $\fc_{\Si}\!:\Si\!\lra\!\Si$ is an orientation-reversing involution.
The remaining automorphisms involve diffeomorphisms of $\Si$ that interchange 
its boundary components. 
The second restrictions in~\eref{flipeven0_e} and in~\eref{flipeven1_e}
and the dimensional condition in Theorem~\ref{g1orient_thm} imply that the relevant 
automorphisms on the moduli spaces $\fM(X,B;J)^{\phi,c}$ are orientation-preserving;
the remaining conditions insure that these moduli spaces of maps with crosscaps
are orientable.\\

\noindent
The orientability of the moduli spaces $\fM(X,B;J)^{\phi,c}$
of maps with crosscaps is studied in~\cite{GZ}.
In this paper, we compute the effect of the actions of the natural
automorphisms of these spaces on local orientations;
see Propositions~\ref{relsign_prp} and~\ref{transfsign_prp}. 
If $\Si\!=\!\P^1$, \eref{conjaut_e} is the only relevant automorphism and 
is equivalent to the automorphism 
\BE{conjaut_e2} [\fJ_0,f]\lra [\fJ_0,\phi\circ f\circ \fc_{D^2}],\EE
where $\fJ_0$ is the standard complex structure on $D^2$ and 
$\fc_{D^2}$ is the standard anti-holomorphic involution (conjugation) on~$D^2$.
By \cite[Section~2.1]{Teh} and the proof of Corollary~\ref{etaorient_crl}, 
this automorphism is orientation-preserving 
in the case $\si\!=\!\eta$ if the appropriate moduli space $\fM(X,B;J)^{\phi,c}$ is orientable
and has no effect on the relative sign of the automorphism~\eref{conjaut_e2}
in general.
In the case $\si\!=\!\tau$, the purely cohomological condition~\eref{flipeven0_e}
in Theorem~\ref{g0orient_thm} replaces the rather artificial 
case-by-case conditions in~\cite{FOOO9}; see Section~\ref{FOOO_subs} for more details.\\

\noindent
Theorems~\ref{g0orient_thm} and~\ref{g1orient_thm} can be extended to 
the corresponding moduli spaces $\fM_{k,l}(X,B;J)^{\phi,\si}$ of 
real maps with $k$~boundary and $l$~interior marked points,
as the effect of adding marked points on the sign of the relevant automorphisms
can be easily determined.
The introduction of decorated marked points on oriented sh-surfaces in~\cite{Ge2} 
in fact removes 
the need to consider the effect of conjugate pairs of marked points,
while making the resulting invariants precisely agree with counts of real curves
in sufficiently positive symplectic manifolds.\footnote{In contrast,
the invariants of~\cite{Sol}, for example, differ by~$\pm2^{l-1}$.}\\

\noindent
The moduli spaces $\fM_{k,l}(X,B;J)^{\phi,\si}$ typically have codimension-one boundary
and often of more than one type.
The codimension-one boundary stratum consisting of maps from $\hat\Si$ with
a bubble attached at a real point of the domain can be eliminated by the gluing 
of procedure of \cite{Cho,Sol}, which is adapted to maps with decorated 
marked points in \cite[Section~3]{Ge2}.
By \cite[Theorems~1.3]{Ge2}, the proof of \cite[Corollary~6.1]{Ge2}, 
Propositions~\ref{relsign_prp} and~\ref{transfsign_prp}, and Corollary~\ref{transfsign_crl},
the glued  moduli spaces $\wt\fM_{\hat{g},0,l}(X,B;J)^{\phi}$,
with $\hat{g}\!=\!0,1$,
are still orientable under the conditions of Theorems~\ref{g0orient_thm} and~\ref{g1orient_thm}.
For $k\!>\!0$,  $\wt\fM_{\hat{g},k,l}(X,B;J)^{\phi}$ is typically not orientable,
but it may still be possible to define invariants in some cases, 
as done in \cite[Section~7.1]{Ge2}.\\

\noindent
The remaining types of codimension-one boundary strata of $\fM_{\hat{g},0,l}(X,B;J)^{\phi}$
correspond to one-nodal degenerations of~$\hat\Si$ passing between
involutions on~$\hat\Si$ of different topological types, 
as described in detail in  \cite[Section~4]{Klein}, 
\cite[Section~5]{Se91}, and~\cite[Sections~3,4]{Melissa}.
As suggested in \cite[Section~1.5]{PSW} and carried out in \cite[Section~3]{Teh} 
in the case $\hat\Si\!=\!\P^1$, the moduli spaces $\fM_{0,l}(X,B;J)^{\phi,\si}$
with different types of involutions~$\si$ on~$\hat\Si$ should in general be combined
to get well-defined invariants by gluing along codimension-one boundaries.
However, in positive genus, it is simpler to consider such transitions
in the setting of oriented sh-surfaces.
We intend to pursue this approach to real Gromov-Witten invariants in a future paper,
making use of the preliminaries developed in~\cite{GZ} and
of the signs of the natural automorphisms provided by 
Propositions~\ref{relsign_prp} and~\ref{transfsign_prp}.\\

\noindent
We conclude this section by describing topological conditions on $(X,\om,\phi)$
insuring that
the actions of the natural automorphisms on topological components of the moduli space
$$\fM_{\Si}(X,X^{\phi},B;J)\equiv\fM(X,B;J)^{\phi,\id_{\Si}}$$
of $J$-holomorphic maps $(\Si,\prt\Si)\!\lra\!(X,X^{\phi})$ are  
orientation-preserving.
On one hand, these automorphisms being orientation-preserving is equivalent
to the moduli space $\fM(X,B;J)^{\phi,\si}$,
where $\si$ is the involution on~$\hat\Si$ with $\hat\Si^{\si}\!=\!\prt\Si$,
being orientable.
On the other hand, notable Floer-theoretic applications of this property 
in the $\Si\!=\!D^2$ case are demonstrated in~\cite{FOOO9};
in light of recent work on higher-genus generalizations of Floer homology,
it is reasonable to expect similar applications of this property
in other cases as well.\\

\noindent
Let $(\prt\Si)_1,\ldots,(\prt\Si)_m$ be the boundary components of~$\Si$.
For each 
\BE{btupldfn_e}\b\equiv (\be,b_1,\ldots,b_m)\in 
H_2(X,X^{\phi};\Z)\oplus H_1(X^{\phi};\Z)^{\oplus m}\,,\EE
we denote by $\fM(X,X^{\phi},\b;J)$ 
the space of equivalence classes of pairs $(f,\fJ)$, where $\fJ$ is 
a complex structure on~$\Si$ compatible with the orientation and 
$$f\!:(\Si,\prt\Si)\lra(X,X^{\phi}) \qquad\hbox{s.t.}\quad
\dbar_{J,\fJ}f=0,\quad
f_*[\Si,\prt\Si]_{\Z}=\be, \quad f_*[(\prt\Si)_i]=b_i~~\forall\,i.$$
Every diffeomorphism of $\Si$, orientation-preserving or orientation-reversing,
induces an automorphism of this moduli space.
For $\be$ as in~\eref{btupldfn_e}, let $\fd(\be)\!\in\!H_2(X;\Z)$ denote
the natural $\phi$-double of~$\be$; see \cite[Section~3]{Ge2}.

\begin{thm}\label{seorient_thm}
Let $(X,\om)$ be a symplectic $2n$-manifold with an anti-symplectic involution~$\phi$,
$J\!\in\!\cJ_{\phi}$, $\Si$ be an oriented bordered surface,
and $\b$ be as in~\eref{btupldfn_e}.
If $b_i\!=\!b_j$ for some $i\!\neq\!j$, assume also that $n$ is odd.
If $X^{\phi}\!\subset\!X$ is orientable and there exists $\vp\!\in\!H^2(X;\Z_2)$ 
satisfying~\eref{flipeven0_e} with $B\!=\!\fd(\be)$,  then
$\fM(X,X^{\phi},\b;J)$ can be oriented by a relative spin structure 
so that the natural automorphisms of this moduli space
induced by diffeomorphisms of~$\Si$ are orientation-preserving.
\end{thm}

\noindent
This theorem is implied by Corollary~\ref{seorient_crl} of 
Propositions~\ref{relsign_prp} and~\ref{transfsign_prp}, which establishes the same
conclusion under weaker assumptions.
In the case of the natural automorphisms induced by diffeomorphisms preserving
the boundary components, Theorem~\ref{seorient_thm} recovers 
the orientable case of \cite[Proposition~5.1]{Sol};
the full statement of the latter is recovered in Remark~\ref{Sol_rmk}.
In the case of the natural automorphisms induced by diffeomorphisms preserving
the orientation, Theorem~\ref{seorient_thm}  applies to arbitrary Lagrangians,
as the involution~$\phi$ plays no role then.\\

\noindent
This paper is organized as follows.
Section~\ref{analys_subs} reviews the analytic setup and Fredholm theory 
for oriented sh-surfaces.
In Section~\ref{DMsign_subs}, we determine the signs of natural automorphisms 
on the Deligne-Mumford moduli spaces of bordered Riemann surfaces,
establishing one of the three key statements in this paper.
In Section~\ref{topol_subs}, we make a number of simple, but useful, topological observations
that  naturally fit  with the orientability problem in Gromov-Witten theory.
Section~\ref{equivcoh_subs} reviews basic notions in equivariant cohomology
and investigates in detail properties of the equivariant~$w_2$ of real bundle pairs
over surfaces.
Section~\ref{main_sec} establishes the two remaining key statements of this paper,
Propositions~\ref{relsign_prp} and~\ref{transfsign_prp}, 
which determine the effects of the conjugation and interchanges of boundary
components on local orientations of index bundles of real Cauchy-Riemann operators.
We introduce a weaker version of spin structures in Section~\ref{SpinSubStr_subs}
and show in Section~\ref{MSopen_subs} that relative spin sub-structures induce
orientations on moduli spaces of open maps. 
In Section~\ref{MSreal_subs}, we combine Proposition~\ref{DMsgn_prp} with
Propositions~\ref{relsign_prp} and~\ref{transfsign_prp} 
to show that relative spin sub-structures compatible with an involution~$\phi$
induce orientations preserved by the natural automorphisms of the moduli spaces
of maps from surfaces with crosscaps
and establish more general versions of Theorems~\ref{g0orient_thm}, 
\ref{g1orient_thm}, and~\ref{seorient_thm}.
In Section~\ref{FOOO_subs}, we apply Propositions~\ref{relsign_prp} and~\ref{transfsign_prp} 
to Floer homology, following the principles laid out in~\cite{FOOO9}.\\

\noindent
We would like to thank A.~Bahri, E.~Brugall\'e, R.~Cr\'etois, 
K.~Fukaya, M.~Goresky, N.~Hingston, 
M.~Liu, J.~Solomon, M.~Tehrani, G.~Tian, and J.~Welschinger for related discussions
and the referee for a very thorough reading of our paper and pointing out mistakes
in some of the original arguments.
The second author is also grateful to the IAS School of Mathematics for its hospitality 
during the period when the results in this paper were obtained.

\section{Preliminaries}
\label{prelim_sec}

\noindent
We begin with some background material.
We first review the analytic setup for maps from oriented sh-surfaces
introduced in~\cite{GZ} and recall some of the related results obtained in~\cite{GZ}.
We then determine the signs of natural automorphisms on the Deligne-Mumford moduli spaces
of bordered marked Riemann surfaces; this is one of the key ingredients 
in the proof of Theorem~\ref{seorient_thm}.
We conclude with a number of purely topological observations that naturally fit 
with the orientability conditions discovered in~\cite{Ge} in open Gromov-Witten
theory and adapted in~\cite{GZ} to real Gromov-Witten theory.

\subsection{Review of analytic setup}
\label{analys_subs}

\noindent
An \sf{oriented symmetric half-surface} (or simply \sf{oriented sh-surface}) 
is a pair $(\Si,c)$ consisting of an oriented bordered smooth surface~$\Si$ 
and an involution $c\!:\prt\Si\!\lra\!\prt\Si$ preserving each component
and the orientation of~$\prt\Si$.
The restriction of such an involution~$c$ to a boundary component~$(\prt\Si)_i$
is either the identity or the antipodal map
\BE{antip_e}\fa\!:S^1\lra S^1, \qquad z\lra -z,\EE
for a suitable identification of $(\prt\Si)_i$ with $S^1\!\subset\!\C$;
the latter type of boundary structure is called \sf{crosscap} 
in the string theory literature.
We define
$$c_i=c|_{(\prt\Si)_i}, \qquad 
|c_i|= \begin{cases} 0,&\hbox{if}~c_i=\id;\\ 1,&\hbox{otherwise};\end{cases}
\qquad
|c|_k=\big|\{(\prt\Si)_i\!\subset\!\Si\!:\,|c_i|\!=\!k\}\big|\quad k=0,1.$$
Thus, $|c|_0$ is the number of standard boundary components of $(\Si,\prt\Si)$  
and $|c|_1$ is the number of crosscaps.
We order the boundary components~$(\prt\Si)_i$ of~$\Si$ so that $|c_i|\!=\!0$
for $i\!=\!1,\ldots,|c|_0$.\\

\noindent
An oriented symmetric half-surface $(\Si,c)$
\sf{doubles} to the topological symmetric surface
\begin{alignat*}{2}
&\hat\Si\equiv\big(\Si^+\sqcup\Si^-\big)\big/\!\sim~
\!\!\equiv\{+,-\}\!\times\!\Si\big/\!\sim,&\qquad
&(+,z)\sim\big(-,c(z)\big)~~~\forall\,z\!\in\!\prt\Si,\\
&\hat{c}\!:\hat\Si\lra\hat\Si,&\qquad &\hat{c}\big([\pm,z]\big)=[\mp,z]~~~\forall\,z\!\in\!\Si,
\end{alignat*}
with $\Si^+$ having the same orientation as~$\Si$
(and $\Si^-$ having the opposite orientation to~$\Si$).
By \cite[Theorems~1.1,~1.2]{Nat}, every symmetric surface~$(\hat\Si,\si)$ can be obtained in this way.
If $(X,\phi)$ is a manifold with an involution,  
a \sf{real map} $f\!:(\Si,c)\!\lra\!(X,\phi)$ 
is a map $f\!:\Si\!\lra\!X$ such that $f\!\circ\!c=\phi\!\circ\!f$ on~$\prt\Si$.
Such a map \sf{doubles} to a real map $\hat{u}\!:(\hat\Si,\hat{c})\!\lra\!(X,\phi)$ such that
$\hat{u}|_{\Si^+}\!=\!u$.\\

\noindent
Given an oriented sh-surface $(\Si,c)$, let $\cD_{\Si}$ be the group of diffeomorphisms
of $\Si$ preserving the orientation and the boundary components and 
$\cD_c\!\subset\!\cD_{\Si}$ be the subgroup of diffeomorphisms that commute 
with the involution~$c$ on~$\prt\Si$. 
Let $\cJ_{\Si}$ be the space of all complex structures on~$\Si$ compatible
with the orientation.
For each $\fJ\!\in\!\cJ_{\Si}$, $\Si$ can be covered by coordinate charts  
$$\psi\!:(U,U\!\cap\!\prt\Si)\lra (\bH,\R)
\qquad\hbox{s.t.}\quad \fJ=\psi^*\fJ_0,$$
where $\bH$ is the closed upper-half plane and $\fJ_0$ is the standard complex
structure on~$\C$,
and the overlap maps between such charts are holomorphic;
see \cite[Corollary~A.2]{GZ}.
We call such charts $\fJ$-holomorphic.
Let $\cJ_c\!\subset\!\cJ_{\Si}$ denote the subspace of complex structures~$\fJ$
such that  $c$ is real-analytic with respect to~$\fJ$, i.e.~for every $z\!\in\!\prt\Si$
there exist $\fJ$-holomorphic charts
$$\psi_z\!:U_z\lra U_z' \qquad\hbox{and}\qquad \psi_{c(z)}\!:U_{c(z)}\lra U_{c(z)}',$$
where $U_z$ and $U_{c(z)}$ are open subsets of $\Si$ containing $z$ and $c(z)$,
respectively, and  $U_z'$ and $U_{c(z)}'$ are open subsets of~$\bH$,
such~that 
$$\psi_{c(z)}\circ c\circ\psi_z^{-1}\!:~
  \psi_z\big(U_z\cap c(U_{c(z)}\!\cap\!\prt\Si)\big)\lra\R$$
is a real-analytic function on an open subset of $\R\!\subset\!\C$.
By \cite[Lemma~3.1]{GZ}, the subspace~$\cJ_c$ of~$\cJ_{\Si}$ is preserved by~$\cD_c$. 
By \cite[Corollary~3.3]{GZ}, each $\fJ\!\in\!\cJ_c$ doubles to a complex
structure~$\hat\fJ$ on $\hat\Si$ so that $\hat\fJ|_{\Si}\!=\!\fJ$, 
$\hat{c}^*\hat\fJ\!=\!-\hat\fJ$, and $\prt\Si$ is a real-analytic curve in~$\hat\Si$;
in particular, $\hat\fJ$ determines a smooth structure on~$\hat\Si$.
If $(X,\phi)$ is a manifold with an involution and  $u\!:(\Si,c)\!\lra\!(X,\phi)$
is a real $(J,\fJ)$-holomorphic map,
then $\hat{u}$ is a $(J,\hat\fJ)$-holomorphic map.\\

\noindent
For applications to moduli problems, it is convenient to introduce subspaces $\cJ_c^*$ and
$\cD_c^*$ of  $\cJ_c$ and $\cD_c$ so that the natural map
\BE{DMisom_e}\cJ_c^*/\cD_c^*\lra \cJ_{\Si}/\cD_{\Si}\EE
induced by the inclusions $\cJ_c^*\!\lra\!\cJ_{\Si}$ and $\cD_c^*\!\lra\!\cD_{\Si}$
is an isomorphism (as Artin stacks)
whenever $(\Si,c)$ is not a disk with an involution different from the identity.
We take $\cJ_c^*\!=\!\{\fJ_0\}$ if $\Si\!=\!D^2$,
$\cD_c^*\!=\!\PGL_2^0\R$ to be the group of holomorphic automorphisms of~$D^2$
if $(\Si,c)\!=\!(D^2,\id_{S^1})$, and 
$\cD_c^*\!=\!S^1$ to be the group of standard rotations of~$D^2$
if $(\Si,c)\!=\!(D^2,\fa)$.\\

\noindent
Suppose next that $\Si$ is a cylinder with ordered boundary components $(\prt\Si)_1$
and~$(\prt\Si)_2$.
Let $\oI\!=\!(0,1)$ and 
$$h_{\C^*}\!:\C^*\lra\C^*, \qquad z\lra 1/z\,.$$
For each $r\!\in\!\oI$, we define 
$$A_r=\big\{z\!\in\!\C\!:\,(|z|\!-\!r)(r|z|\!-\!1)\le0\big\},\quad
(\prt A_r)_1=\big\{z\!\in\!\C\!:\,|z|\!=\!r\big\}, \quad
(\prt A_r)_2=\big\{z\!\in\!\C\!:\,r|z|\!=\!1\big\}.$$
Choose a smooth map
$$\Psi\!: \oI\!\times\!\Si\lra\C^*, \qquad
\Psi(r,z)\lra \Psi_r(z),$$
such that each map
$$\Psi_r\!: \big(\Si,(\prt\Si)_1,(\prt\Si)_2\big)
\lra  \big(A_r,(\prt A_r)_1,(\prt A_r)_2\big), \qquad r\in\oI,$$
is a diffeomorphism so that $\fa\!\circ\!\Psi_r\!=\!\Psi_r\!\circ\!c_i$
on $(\prt\Si)_i$ if $|c_i|\!=\!1$, $i\!=\!1,2$, and the diffeomorphisms
$$\Psi_r\!\circ\!\Psi_{r'}^{-1}\!: A_{r'}\lra A_r, \qquad r,r'\in\oI,$$
commute with $h_{\C^*}$, the standard conjugation~$\fc_{\C}$ on~$\C$, and
the standard action of $S^1\!\subset\!\C^*$ on~$\C$.
The last condition implies~that the diffeomorphisms
\begin{alignat*}{2}
h_{\Si}\!: \Si&\lra\Si, &\qquad 
h_{\Si}(z)&=\Psi_r^{-1}\big(h_{\C^*}(\Psi_r(z))\big), \qquad\hbox{and}\\
\fc_{\Si}\!: \Si&\lra\Si, &\qquad 
\fc_{\Si}(z)&\lra\Psi_r^{-1}\big(\fc_{\C}(\Psi_r(z))\big),
\end{alignat*}
and the $S^1$-action on~$\Si$ given~by
\BE{S1act_e}S^1\times\Si\lra\Si, \qquad \th\cdot z=\Psi_r^{-1}\big(\th\Psi_r(z)\big)
\quad\forall\,z\!\in\!\Si,\,\th\!\in\!S^1\!\subset\!\C,\EE
are independent of $r\!\in\!\oI$.
In this case, we take 
$$\cJ_c^*=\big\{\Psi_r^*\fJ_0\!:\,r\!\in\!\oI\big\}$$
and $\cD_c^*\!\subset\!\cD_c$ to be the subgroup corresponding to the action~\eref{S1act_e}.
The latter is the group of automorphisms of each complex structure in $\cJ_c^*$
that preserve each boundary component of~$\Si$.
By the classification of complex structures on the cylinder \cite[Section~9]{AG}, 
for every $\fJ\!\in\!\cJ_{\Si}$, there exist a unique 
$r\!\in\!\oI$ and a diffeomorphism~$h$ of~$\Si$ preserving the orientation
and the boundary components such that $\fJ\!=\!h^*\Psi_r^*\fJ_0$.
It follows that  the map~\eref{DMisom_e} is an isomorphism.\\

\noindent
If $\Si$ is not a disk or a cylinder, i.e.~the genus of its double is at least~2,
we identify each boundary component $(\prt\Si)_i$ of $\prt\Si$ with $S^1$ in such 
a way that $c_i\!\equiv\!c|_{(\prt\Si)_i}$ corresponds to either the identity or
the antipodal map on~$S^1$
and denote by $\cD_i$ the subgroup of diffeomorphisms of $(\prt\Si)_i$ corresponding
to the rotations of~$S^1$ under this identification.
For each $\fJ\!\in\!\cJ_{\Si}$, there exists a unique metric $\hat{g}_{\fJ}$
on the double $(\hat\Si',\hat\fJ')$ of $(\Si,\fJ)$ with respect to the involution
$\id_{\prt\Si}$ so that $\hat{g}_{\fJ}$ has constant scalar curvature~-1
and is compatible with~$\hat\fJ'$.
Each boundary component $(\prt\Si)_i$ is a geodesic with respect to~$\hat{g}_{\fJ}$,
and each isometry of $(\prt\Si)_i$ with respect to $\hat{g}_{\fJ}$ is real-analytic
with respect to~$\fJ$.
We denote by $\cJ_c^*\!\subset\!\cJ_{\Si}$ the subspace of complex structures~$\fJ$
so that each $\cD_i$ is the group of isometries of~$(\prt\Si)_i$ with respect to~$\hat{g}_{\fJ}$ 
and by $\cD_c^*$ the subgroup of diffeomorphisms of $\Si$ that preserve the orientation
and the boundary components and restrict to elements of~$\cD_i$ on each boundary component
$(\prt\Si)_i$ of~$\Si$. 
Since $c_i\!\in\!\cD_i$ for each~$i$, 
$\cJ_c^*\!\subset\!\cJ_c$ and $\cD_c^*\!\subset\!\cD_c$.
By \cite[Lemma~6.1]{GZ}, the map~\eref{DMisom_e} is an isomorphism
in this case as well.\\

\noindent
Let $(\Si,c)$ be an oriented sh-surface with orderings 
$$(\prt\Si)_1,\ldots,(\prt\Si)_{|c|_0} \qquad\hbox{and}\qquad
(\prt\Si)_{|c|_0+1},\ldots,(\prt\Si)_{|c|_0+|c|_1}$$
of the boundary components with $|c_i|\!=\!0$ and with $|c_i|\!=\!1$
and  $(X,\phi)$ be a smooth manifold with an involution.
Given
\BE{btuple_eq}
\b=(B,b_1,\ldots,b_{|c|_0+|c|_1})
\in H_2(X;\Z)\oplus H_1(X^{\phi};\Z)^{\oplus |c|_0}
\oplus H_1^{\phi}(X;\Z)^{\oplus |c|_1},
\EE 
let $\fB(X,\b)^{\phi,c}$ denote the space of smooth maps 
$u\!:\Si\!\lra\!X$ such that 
\begin{enumerate}[label=$\bullet$,leftmargin=*]
\item $u\!\circ\!c\!=\!\phi\!\circ\!u$ on $\prt\Si$, 
\item $\hat{u}_*[\hat\Si]=B$, $u_*[(\prt\Si)_i]=b_i$ for $i=1,\ldots,|c|_0$, and
\item $[u|_{(\prt\Si)_i}]_{\Z_2}^{c_i}=b_i$ for $i=|c|_0\!+\!1,\ldots,|c|_0\!+\!|c|_1$,
where $[u|_{(\prt\Si)_i}]_{\Z_2}^{c_i}$ is the equivariant pushforward 
of $[(\prt\Si)_i]_{\Z_2}^{c_i}$  by~$u|_{(\prt\Si)_i}$,  as in Section~\ref{equivdfn_e}.
\end{enumerate}
We denote by $\fB(X)^{\phi,c}$ the disjoint union of the spaces $\fB(X,\b)^{\phi,c}$
over all tuples~$\b$ as in~\eref{btuple_eq}.
If in addition $\k\!=\!(k_1,\ldots,k_{|c|_0+|c|_1})$ is a tuple of nonnegative integers, 
let
$$\fB_{\k}(X,\b)^{\phi,c}=\fB(X,\b)^{\phi,c}\times
\prod_{i=1}^{|c|_0+|c|_1}\!\!\!\!\!\big((\prt\Si)_i^{k_i}-\De_{i,k_i}\big),$$
where 
$$\De_{i,k_i}=\big\{(x_{i,1},\ldots,x_{i,k_i})\!\in\!(\prt\Si)_i^{k_i}\!:~
x_{i,j'}\!\in\!\{x_{i,j},c(x_{i,j})\}~\tn{for some}~j,j'\!=\!1\ldots,k_i,~j\!\neq\!j'\big\}$$
is the big $c$-symmetrized diagonal.
Let
\BE{cHcMdfn_e}\begin{split}
\cH_{\k}^*(X,\b)^{\phi,c}&=
\big(\fB_{\k}(X,\b)^{\phi,c}\!\times\!\cJ_c^*\big)\big/\cD_c^*,\\
\fM_{\k}(X,J,\b)^{\phi,c}&=
\big\{[u,\x_1,\ldots,\x_{|c|_0+|c|_1},\fJ]\!\in\!\cH_{\k}^*(X,\b)^{\phi,c}\!:~
\dbar_{J,\fJ}u\!=\!0\big\},
\end{split}\EE
where $\dbar_{J,\fJ}$ is the usual Cauchy-Riemann operator with respect to
the complex structures~$J$ on~$X$ and $\fJ$ on~$\Si$.
If~$X$ is a point and $\b$ is the zero tuple, we denote
$\cH_{{\k}}^*(X,\b)^{\phi,c}$ by $\cM_{\Si,\k}^c$;
by \cite[Lemma~6.1]{GZ}, 
this is the usual Deligne-Mumford moduli space  
$$\cM_{\Si,\k}=\cM_{\Si,\k}^{\id_{\Si}}$$
of stable bordered Riemann surfaces  with ordered boundary components 
with boundary 
if $\Si$ is not a disk with $k_1\!<\!3$ or a cylinder with $k_1,k_2\!=\!0$ 
(for stability reasons).\\

\noindent
If $M$ is a manifold, possibly with boundary, or a (possibly nodal) surface, 
and $c$ is an involution on a submanifold $M'\!\subset\!M$,
a \sf{real bundle pair} $(V,\ti{c})\!\lra\!(M,c)$   
consists of a complex vector bundle $V\!\lra\!M$ and a conjugation~$\ti{c}$ on $V|_{M'}$
lifting~$c$.
A real bundle pair $(V,\ti{c})\!\lra\!(\Si,c)$, where $(\Si,c)$ is an oriented sh-surface,
\sf{doubles} to a real  bundle pair over $(\hat\Si,\hat{c})$,
\begin{alignat*}{2}
&\hat{V}\equiv\big(\{+\}\!\times\!V\sqcup \{-\}\!\times\!\bar{V}\big)\big/\!\sim,
&\qquad &(+,v)\sim\big(-,\ti{c}(v)\big)~~~\forall\,v\!\in\!V|_{\prt\Si},\\
&\br{c}\!:\hat{V}\lra\hat{V},
&\qquad & \br{c}\big([\pm,v]\big)=[\mp,v]~~~\forall\,v\!\in\!V,
\end{alignat*} 
where $\bar{V}$ denotes the same real vector bundle over~$\Si$ as~$V$, but
with the opposite complex structure on the fibers.
We define the \sf{Maslov index of~$(V,\ti{c})$}~by
\BE{Maslovdfn_e}\mu(V,\ti{c})=\lr{c_1(\hat{V}),[\hat\Si]}.\EE
By \cite[Theorem~C.3.5 and~(C.3.4)]{MS}, this agrees with the usual definition of
the Maslov index of $(V,V^{\ti{c}})$ if $c\!=\!\id_{\prt\Si}$.
By \cite[Propositions~4.1,~4.2]{BHH}, real bundle pairs $(V,\ti{c})\!\lra\!(\Si,c)$
are classified by their rank, the Maslov index, and the orientability of $V^{\ti{c}}$
over each boundary component $(\prt\Si)_i$ with $|c_i|\!=\!0$.\\

\noindent
A \sf{real Cauchy-Riemann operator on a real bundle pair $(V,\ti{c})\!\lra\!(\Si,c)$},  
where $(\Si,c)$ is an oriented sh-surface, is a linear map of the~form
\BE{CRdfn_e}\begin{split}
D=\bp\!+\!A\!: \Ga(\Si;V)^{\ti{c}}
\equiv&\big\{\xi\!\in\!\Ga(\Si;V)\!:\,\xi\!\circ\!c\!=\!\ti{c}\!\circ\!\xi|_{\prt\Si}\big\}\\
&\hspace{.5in}\lra
\Ga_{\fJ}^{0,1}(\Si;V)\equiv\Ga\big(\Si;(T^*\Si,\fJ)^{0,1}\!\otimes_{\C}\!V\big),
\end{split}\EE
where $\bp$ is the holomorphic $\bp$-operator for some $\fJ\!\in\!\cJ_{\Si}$
and a holomorphic structure in~$V$ and  
$$A\in\Ga\big(\Si;\Hom_{\R}(V,(T^*\Si,\fJ)^{0,1}\!\otimes_{\C}\!V) \big)$$ 
is a zeroth-order deformation term. 
By \cite[Proposition~3.6]{GZ},
a real Cauchy-Riemann operator on a real bundle pair
is Fredholm in the appropriate completions if $\fJ\!\in\!\cJ_c$;
by \cite[Remark~3.7]{GZ}, it need not be Fredholm if $\fJ\!\not\in\!\cJ_c$.
A continuous family of such Fredholm operators $D_t$ over a topological space~$\cH$  
determines a line bundle over~$\cH$, called \sf{the determinant line bundle of~$\{D_t\}$}
and denoted $\det D$;
see \cite[Section~A.2]{MS} and \cite{detLB} for a construction.\\

\noindent
Families of real Cauchy-Riemann operators often arise by pulling back data from
a target manifold by smooth maps as follows. 
Suppose $(X,J)$ is an almost complex manifold with an anti-complex  involution
$\phi\!:X\!\lra\!X$ and $(V,\ti\phi)\!\lra\!(X,\phi)$ is a real bundle pair.
Let $\na$ be a connection on~$V$ and 
$$A\in\Ga\big(X;\Hom_{\R}(V,(T^*X,J)^{0,1}\otimes_{\C}\!V)\big).$$ 
For any map $u:\Si\!\lra\!X$ and $\fJ\!\in\!\cJ_{\Si}$, 
let $\na^u$ denote the induced connection in $u^*V$ and
$$ A_{\fJ;u}=A\circ \prt_{\fJ} u\in\Ga(\Si;
\Hom_{\R}(u^*V,(T^*\Si,\fJ)^{0,1}\otimes_{\C}u^*V)\big).$$
If $c$ is a boundary involution on $\Si$ and $u\!\circ\!c\!=\!\phi\!\circ\!u$ on $\prt\Si$, 
the homomorphisms
$$\bp_u^\na =\frac{1}{2}(\na^u+\fI\circ\na^u\circ\fJ), \,\,
D_u\equiv \bp_u^\na\!+\!A_{\fJ;u}\!: \Ga(\Si;u^*V)^{u^*\ti\phi}\lra
\Ga^{0,1}_{\fJ}(\Si;u^*V)$$
are real Cauchy-Riemann operators on $(u^*V,u^*\ti\phi)\!\lra\!(\Si,c)$
that form families of real Cauchy-Riemann operators over families of maps.
We denote the determinant line bundle of such a family by $\det(D_{V,\ti\phi})$.

\subsection{The signs of automorphisms of the Deligne-Mumford spaces}
\label{DMsign_subs}

\noindent
Let $\Si$ be an oriented genus~$g$ surface with $m\!>\!0$ boundary components
so that $2g\!+\!m\!\ge\!3$. 
For any diffeomorphism $h\!:\Si\!\lra\!\Si$, let 
\BE{habsdfn_e}
|h|=\begin{cases}0,&\hbox{if}~h~\hbox{is orientation-preserving};\\
1,&\hbox{if}~h~\hbox{is orientation-reversing};\end{cases}\EE
and denote by $\sgn_h\!\in\!\{0,1\}$ the sign of the permutation 
induced by~$h$ on the~set of boundary components of~$\Si$. 
Such a diffeomorphism induces an automorphism
\BE{DMhdfn_e}\DM_h\!: \cM_{\Si}\equiv\cM_{\Si,\0}\lra \cM_{\Si},\qquad 
[\fJ]\lra \big[(-1)^{|h|}h^*\fJ\big],\EE
with the notation as in~\eref{cHcMdfn_e}.
The Deligne-Mumford moduli space $\cM_{\Si}$ of bordered Riemann surfaces 
is orientable; see Lemma~\ref{DMorient_lmm} below.
Proposition~\ref{DMsgn_prp} determines  the sign of the automorphism~$\DM_h$;
the sign of its analogue on $\cM_{\Si,\k}$ can then be easily obtained 
by proceeding as in the proof of \cite[Corollary~1.8]{Ge}.

\begin{lmm}\label{DMorient_lmm}
Let $\Si$ be an oriented genus~$g$ surface with $m$ boundary components
so that $2g\!+\!m\!\ge\!3$.
An ordering of the boundary components of~$\Si$ canonically determines an
orientation of the Deligne-Mumford moduli space~$\cM_{\Si}$.
\end{lmm}

\begin{proof}
For each complex structure $\fJ$ on $\Si$, there is a unique metric~$g_{\fJ}$ 
on~$\Si$ of constant scalar curvature~$-1$ such that each boundary component $(\prt\Si)_i$
of~$\Si$ is a $g_{\fJ}$-geodesic.
We denote by~$L_i(\fJ)$ the length of $(\prt\Si)_i$ with respect to~$g_{\fJ}$.
The boundary length~map 
\BE{DMfibr_e}\cM_{\Si}\lra (\R^+)^m\,, \qquad [\fJ]\lra \big(L_1(\fJ),\ldots,L_m(\fJ)\big),\EE
is a fibration; the fiber over a point $(L_1,\dots, L_m)$ is the moduli space 
$\cM_{\Si}(L_1,\ldots,L_m)$ of bordered Riemann surfaces with fixed lengths of 
the boundary components equal to~$(L_1,\dots, L_m)$. 
The latter moduli space carries a Weil-Petersson volume form; see \cite[Section~2]{Mir}.
Thus, it is canonically oriented. 
By the homotopy exact sequence for this fibration,
the fundamental groups of the fiber and the total space are isomorphic 
by the inclusion homomorphism and so every loop homotopes to a fiber.
Since each fiber and its normal bundle are canonically oriented, 
the total space is also canonically oriented. 
\end{proof}

\begin{lmm}\label{DMorient_lmm2}
Let $\Si$ be an oriented genus~$g$ surface with $m$ boundary components
so that $2g\!+\!m\!\ge\!3$ and $\ov\Si$ be the same surface with the opposite orientation.
The sign of the diffeomorphism 
$$\cC_{\Si}\!:\cM_{\Si}\lra \cM_{\bar\Si}, \qquad  [\fJ]\lra[-\fJ],$$
with respect to the orientations of Lemma~\ref{DMorient_lmm} is $(-1)^{3g-3+m}$.
\end{lmm}

\begin{proof}
The diffeomorphism~$\cC_{\Si}$ induces a diffeomorphism~$\wt\cC_{\Si;\1}$ 
between the Teichm\"uller spaces~$\cT_{\Si}(\1)$ and~$\cT_{\ov\Si}(\1)$
of bordered surfaces with unit length boundary components. 
It is sufficient to determine the sign of~$\wt\cC_{\Si;\1}$
in the Fenchel-Nielsen coordinates 
$$(\ell_1,\vt_1,\dots,\ell_{3g-3+m},\vt_{3g-3+m})
\qquad\hbox{and}\qquad 
(\bar\ell_1,\bar\vt_1,\dots,\bar\ell_{3g-3+m},\bar\vt_{3g-3+m})$$
on these spaces obtained from the same pair-of-pants decomposition of $\Si\!=\!\bar\Si$;
see \cite[(II.3.2)]{Abikoff}, for example.
Since $g_{-\fJ}\!=\!g_{\fJ}$, but the orientations of the cutting circles are reversed,
$$\wt\cC_{\Si;\1}(\ell_1,\vt_1,\dots,\ell_{3g-3+m},\vt_{3g-3+m})
=(\ell_1,-\vt_1,\dots,\ell_{3g-3+m},-\vt_{3g-3+m})\,.$$
The claim now follows from \cite[Theorem~2.1]{Mir}.
\end{proof}

\begin{crl}\label{DMorient_crl}
Let $\Si$ be an oriented genus~$g$ surface with $m$ boundary components
so that $2g\!+\!m\!\ge\!3$.
If $h\!:\Si\!\lra\!\Si$ is an orientation-reversing diffeomorphism  
preserving the boundary components of~$\Si$,
the sign of $\DM_h$ is $(-1)^{g+m-1}$.
\end{crl}

\begin{proof}
Since $h$ preserves the boundary components and $g_{-h^*\fJ}\!=\!h^*g_{\fJ}$,
$\DM_h$ acts on the fibers of~\eref{DMfibr_e}.
By Lemma~\ref{DMorient_lmm}, it is thus sufficient to show that the diffeomorphism
$$\cT_{\Si}(\1)\lra \cT_{\ov\Si}(\1), \qquad [\fJ]\lra [h^*\fJ],$$
is orientation-preserving with respect to the Weil-Petersson orientations.
This diffeomorphism takes the Fenchel-Nielsen coordinates 
with respect to the pair of pants decomposition~$\cP$ to 
the Fenchel-Nielsen coordinates with respect to the pair of pants decomposition~$h^{-1}(\cP)$.
The claim now follows from \cite[Theorem~2.1]{Mir}.
\end{proof}

\begin{lmm}\label{DMsgn_lmm}
Let $\Si$ be an oriented genus~$g$ surface with $m$ boundary components
so that $2g\!+\!m\!\ge\!3$.
For every pair of distinct boundary components $(\prt\Si)_i$ and $(\prt\Si)_j$
of~$\Si$, there exists an orientation-preserving diffeomorphism~$h_{ij}$ of~$\Si$
interchanging $(\prt\Si)_i$ and $(\prt\Si)_j$ and 
preserving the remaining boundary components of~$\Si$ such that 
$\DM_{h_{ij}}$ is an orientation-reversing diffeomorphism.
\end{lmm}

\begin{proof}
Choose a pair of pants decomposition of~$\Si$ such that $(\prt\Si)_i$ and $(\prt\Si)_j$
are contained in the same pair of paints; we denote the latter by~$\Si_{ij}$.
Let $h_{ij}\!:\Si\!\lra\!\Si$ be an orientation-preserving diffeomorphism 
which restricts to the identity on an open neighborhood of $\ov{\Si\!-\!\Si_{ij}}$
in~$\Si$ and interchanges $(\prt\Si)_i$ and $(\prt\Si)_j$.
We show below that the induced diffeomorphism $\DM_{h_{ij}}$ is orientation-reversing.\\

\noindent
The diffeomorphism $\DM_{h_{ij}}$ naturally lifts to a diffeomorphism $\wt\DM_{h_{ij}}$
on the Teichm\"uller space~$\cT_{\Si}$.
It is sufficient to show that the action of~$\wt\DM_{h_{ij}}$ on 
the Fenchel-Nielsen coordinates
\BE{FNcoord_e}(\ell_1,\vt_1,\dots,\ell_{3g-3+m},\vt_{3g-3+m},L_1,\dots,L_m)\EE
on $\cT_{\Si}$  is orientation-reversing.
Since $g_{h^*\fJ}\!=\!h^*g_{\fJ}$ in the notation of the proof of Lemma~\ref{DMorient_lmm},
$\wt\DM_{h_{ij}}$ preserves the lengths~$\ell_k$ of the cutting circles for 
the pair of pants decomposition
and the lengths~$L_k$ of the boundary circles with $k\!\neq\!i,j$;
the lengths~$L_i$ and~$L_j$ get interchanged by~$\wt\DM_{h_{ij}}$.
This establishes the claim in the case $\Si\!=\!\Si_{ij}$, i.e.~$g\!=\!0$ and~$m\!=\!3$.
In the other cases, $\wt\DM_{h_{ij}}$ also preserves the twisting parameters~$\vt_k$
associated with the cutting circles other than $\prt\Si_{ij}\!-\!\prt\Si$. 
In the description of \cite[(II.3.2)]{Abikoff}, $\wt\DM_{h_{ij}}$ interchanges 
the origins of the unique shortest geodesics in~$\Si_{ij}$ from this boundary component 
to $(\prt\Si)_i$ and~$(\prt\Si)_j$ and thus changes the associated twisting 
parameter~$\vt_k$ by~$\pi$.
Thus, the action of~$\wt\DM_{h_{ij}}$ on the coordinates~\eref{FNcoord_e} 
is orientation-reversing.
\end{proof}

\begin{prp}\label{DMsgn_prp} 
Let $\Si$ be an oriented genus~$g$ surface with $m$ boundary components
so that $2g\!+\!m\!\ge\!3$.
\begin{enumerate}[label=(\arabic*),leftmargin=*]
\item\label{DMsignP_it2} 
If $h\!:\Si\!\lra\!\Si$ is an orientation-preserving diffeomorphism, 
the sign of the automorphism~$\DM_h$ on $\cM_{\Si}$ is $(-1)^{\sgn_h}$.
\item\label{DMsignR_it2} 
If $h\!:\Si\!\lra\!\Si$ is an orientation-reversing diffeomorphism, 
the sign of the automorphism~$\DM_h$ on $\cM_{\Si}$ 
is $(-1)^{g+m-1+\sgn_h}$.
\end{enumerate}
\end{prp}

\begin{proof}
By Lemma~\ref{DMsgn_lmm}, there exists an orientation-preserving diffeomorphism 
$\hat{h}\!:\!\Si\!\lra\!\Si$ such that the diffeomorphism $\hat{h}^{-1}\!\circ\!h$
preserves the boundary components of~$\Si$ and 
$$\sgn\,\DM_{\hat{h}}=(-1)^{\sgn_h}\,.$$
If $h$ is orientation-preserving, $\hat{h}^{-1}\!\circ\!h$ is then an element of $\cD_{\Si}$
and so
$$\DM_h\circ\DM_{\hat{h}}^{-1}= \DM_h\circ\DM_{\hat{h}^{-1}}
= \DM_{\hat{h}^{-1}\circ h}=\id_{\cM_{\Si}}\,.$$
This establishes the first claim.\\

\noindent
By Corollary~\ref{DMorient_crl} and Lemma~\ref{DMsgn_lmm}, 
 there exists an orientation-reversing diffeomorphism 
$\hat{h}\!:\!\Si\!\lra\!\Si$ such that the diffeomorphism $\hat{h}^{-1}\!\circ\!h$
preserves the boundary components of~$\Si$ and 
$$\sgn\,\DM_{\hat{h}}=(-1)^{g+m-1+\sgn_h}\,.$$
If $h$ is orientation-reversing, $\hat{h}^{-1}\!\circ\!h$ is then an element of $\cD_{\Si}$
and so $\sgn\,\DM_h\!=\!\sgn\,\DM_{\hat{h}}$ as in the previous case.
This establishes the second claim.
\end{proof}

\noindent
By \cite[Lemma~6.1]{GZ}, the conclusion of Proposition~\ref{DMsgn_prp}
also applies to the moduli space of domains with crosscaps,~$\cM_{\Si,\0}^c$.

\subsection{Topological observations}
\label{topol_subs}

\noindent
In light of \cite[Theorem~8.1.1]{FOOO}, \cite[Theorem~6.36]{Melissa}, and \cite[Theorem~1.1]{Sol},
the vanishing of  $w_2(L)$ or $w_2(L)\!+\!w_1(L)^2$ modulo the image of 
$H^2(X;\Z_2)$ in $H^2(L;\Z_2)$ plays an important role in the orientability 
question for moduli spaces of $J$-holomorphic maps from bordered Riemann surfaces
to a symplectic manifold~$X$ with boundary mapping to a Lagrangian submanifold~$L$;
see Section~\ref{MSopen_subs} for more details.
However, as can be seen immediately from \cite[Theorem~1.1]{Ge},
it is in fact the vanishing of  $w_2(L)$ modulo the image of 
$H^2(X;\Z_2)$ in $H^2(L;\Z_2)$ {\it and} the elements of $H^2(L;\Z_2)$
vanishing on all tori which is relevant to the orientability question.
By \cite[Theorem~1.1]{GZ}, the situation for moduli spaces of $J$-holomorphic maps
from Riemann sh-surfaces is similar.
In this section, we study some topological aspects of classes in $H^2(M;\Z_2)$,
for a topological space~$M$, vanishing on all maps from tori to~$M$
and classes vanishing on all maps from closed oriented surfaces to~$M$.
By Lemma~\ref{sqvan_lmm}, the latter are often squares of classes from $H^1(M;\Z_2)$.\\

\noindent
Throughout this paper, we take $\bI\!=\![0,1]$.
We recall that every compact connected unorientable surface~$\Si$ is 
the connected sum of $m$ copies
of~$\R\P^2$ and
$$H_1(\Si;\Z)\approx\Z^{m-1}\oplus\Z_2$$
for some $m\!\in\!\Z$; see \cite[Theorem~77.5]{Mu}.
We begin with two observations made in~\cite{GZ}.

\begin{lmm}[{\cite[Lemma~2.2]{GZ}}]\label{KleinSurface_lem}
Let $\Si$ be a compact connected unorientable surface and 
\hbox{$b_{\Si}\!\in\!H_1(\Si;\Z)$} be 
the nontrivial torsion class.
If $\ka\!\in\!H^1(\Si;\Z_2)$,
\BE{unorpair_e}\blr{\ka^2,[\Si]_{\Z_2}}=\lr{\ka,b_{\Si}}\,,\EE
where $[\Si]_{\Z_2}\!\in\!H_2(\Si;\Z_2)$ is the fundamental class with $\Z_2$-coefficients.
\end{lmm}

\begin{lmm}[{\cite[Corollary~2.3]{GZ}}]\label{sqvan_lmm}
For any topological space~$M$,
$$\big\{w\!\in\!H^2(M;\Z_2)\!:\,w(B)\!=\!0~\forall\,B\!\in\!H_2(M;\Z)\big\}
\supset\big\{\ka^2\!:\,\ka\!\in\!H^1(M;\Z_2)\big\}.$$
If $H_1(M;\Z)$ is finitely generated, the reverse inclusion holds
if and only if $H_1(M;\Z)$ has no \hbox{4-torsion}.
\end{lmm}

\begin{dfn}\label{Klein_dfn}
Let $M$ be a topological space.
\begin{enumerate}[label=(\arabic*),leftmargin=*]
\item\label{KlBd_it} A free homotopy class $b\!\in\!\pi_1(M)$ is a \sf{Klein boundary} 
if there exists a continuous map\linebreak $F\!:\bI\!\times\!S^1\!\lra\!M$ 
such~that 
$$\big[F|_{0\times S^1}\big]=b\in \pi_1(M)
\qquad\hbox{and}\qquad F|_{1\times S^1}=F|_{0\times S^1}\!\circ\!\fc_{S^1}\,,$$
where $\fc_{S^1}\!:S^1\!\lra\!S^1$ is the restriction of the standard conjugation on~$\C$.
\item\label{ator_it} A class $w\!\in\!H^2(M;\Z_2)$ is \sf{atorical} (resp.~\sf{spin})
if $f^*w\!=\!0$ for every continuous map \hbox{$f\!:\bT\!\lra\!M$}
(resp.~for every closed oriented surface $\Si$ and continuous map $f\!:\Si\!\lra\!M$).\\
\end{enumerate}
\end{dfn}

\noindent
If $b$ is a Klein boundary, $2[b]\!=\!0\in H_1(M;\Z)$.
Conversely, if $[b]\!\in\!H_1(M;\Z)$ and $2[b]\!=\!0$, there exist
\begin{enumerate}[label=$\bullet$,leftmargin=*]
\item a compact oriented surface $\Si$ with two boundary components $(\prt\Si)_1$ and 
$(\prt\Si)_2$,
\item orientation-preserving diffeomorphisms $\vph_1\!:S^1\!\lra\!(\prt\Si)_1$ and
$\vph_2\!:S^1\!\lra\!(\prt\Si)_2$, and
\item a continuous map $F\!:\Si\!\lra\!M$ such that 
$\big[F\circ\vph_1\big]\!=\![b]\in H_1(M;\Z)$ and 
$F\!\circ\!\vph_1=F\!\circ\!\vph_2$.
\end{enumerate}
Such a map $F$ descends to a continuous map $\hat{F}$ from the unorientable surface
$$\hat\Si\equiv \Si\big/\!\!\sim, \qquad
z\sim\vph_1(\vph_2^{-1}(z))~~~\forall\,z\!\in\!(\prt\Si)_2\,.$$
The image of each boundary component $(\prt\Si)_i$ in $\hat\Si$
represents the nonzero two-torsion element~$b_{\hat\Si}$ of $H_1(\hat\Si;\Z)$
and $\hat{F}_*b_{\hat\Si}\!=\![b]$.\\

\noindent
If $b\!\in\!\pi_1(M)$ is a Klein boundary, a map $F$ as 
in  Definition~\ref{Klein_dfn}\ref{KlBd_it} descends to a continuous map
from the Klein bottle, 
$$\hat{F}\!: (\bI\!\times\!S^1)/\!\!\sim \lra M,
\qquad\hbox{where}\quad (1,z)\sim(0,\bar{z})~\forall\,z\!\in\!S^1\subset\C,
\quad \hat{F}\big([s,z]\big)=F(s,z),$$
such that the loop $z\!\lra\!F(0,z)$ represents $b\!\in\!\pi_1(M)$.

\begin{lmmdfn}\label{spinpair_dfn}
Suppose $M$ is a topological space.
\begin{enumerate}[label=(\arabic*),leftmargin=*]
\item\label{atorpar_it} Let $b\!\in\!\pi_1(M)$ be a Klein boundary and 
$w\!\in\!H^2(M;\Z_2)$ be atorical.
The number 
\BE{Klatordfn_e}\flr{w,b}\equiv\blr{w,\hat{F}_*[\bK]_{\Z_2}}\in\Z_2,\EE
where $\hat{F}\!:\bK\!\lra\!M$ is induced by 
a map $F$ as  in Definition~\ref{Klein_dfn}\ref{KlBd_it},
is independent of the choice of~$F$. 
\item\label{spinpar_it} Let $b\!\in\!H_1(M;\Z)$ be a two-torsion class and 
$w\!\in\!H^2(M;\Z_2)$ be spin.
The number 
\BE{Torspindfn_e}\flr{w,b}\equiv\blr{w,\hat{F}_*[\hat\Si]_{\Z_2}}\in\Z_2,\EE
where $\hat{F}\!:\hat\Si\!\lra\!M$ is a continuous map from 
an unorientable surface~$\hat\Si$ such that $\hat{F}_*b_{\hat\Si}\!=\!b$
for the unique nonzero two-torsion element $b_{\hat\Si}$  of~$H_1(\hat\Si;\Z)$,
is independent of the choice of~$\hat{F}$. 
\item In either case,
\BE{flrka_e}\flr{\ka^2,b}=\lr{\ka,b} \qquad\forall~ \ka\in H^1(M;\Z).\EE
\end{enumerate}
\end{lmmdfn}

\begin{proof}
(1) Let $F,F'\!:\bI\!\times\!S^1\!\lra\!M$
be continuous maps such~that 
$$\big[F|_{0\times S^1}\big]=\big[F'|_{0\times S^1}\big]=b\in \pi_1(M),
\qquad F|_{1\times S^1}=F|_{0\times S^1}\!\circ\!\fc_{S^1},
\qquad F'|_{1\times S^1}=F'|_{0\times S^1}\!\circ\!\fc_{S^1}\,.$$
The class $\hat{F}_*[\bK]_{\Z_2}+\hat{F}_*'[\bK]_{\Z_2}$
is represented by a continuous map $f\!:\bT\!\lra\!M$, obtained by connecting
the images of $F|_{0\times S^1}$ and $F'|_{0\times S^1}$ by cylinders.
Since $w$ is atorical, it follows that 
$$\blr{w,\hat{F}_*[\bK]_{\Z_2}}+\blr{w,\hat{F}'_*[\bK]_{\Z_2}}
=\blr{w,f_*[\bT]_{\Z_2}}=0,$$
and so $\flr{w,b}$ is well-defined.\\

\noindent
(2) The proof is similar.\\
(3) This follows from \eref{unorpair_e}.
\end{proof}

\noindent
If $f\!:\R\P^2\!\lra\!M$ is a continuous map and $\al\!:S^1\!\lra\!\R\P^2$
represents the nonzero element of $H_1(\R\P^2;\Z)$, then $f\!\circ\!\al$
represents a Klein boundary~$b$ in~$M$.
A map~$F$ as in Definition~\ref{Klein_dfn}\ref{KlBd_it} can be obtained by 
precomposing~$f$ with a map $g\!:\bI\!\times\!S^1\!\lra\!\R\P^2$ such~that
$$g|_{1\times S^1}=g|_{0\times S^1}\circ\fc_{S^1} \qquad\hbox{and}\qquad
g_*[0\times S^1]=[\al]\in H_1(\R\P^2;\Z).$$
In this case,
$$\flr{w,b}=\blr{w,f_*[\R\P^2]_{\Z_2}}\in\Z_2\,$$
for any $w\!\in\!H^2(M;\Z_2)$ atorical.
The map on the $\Z_2$-quotients induced by the map~$f_{\P^1,\bK}$ in the proof 
of Lemma~\ref{P1T2_lmm} is an example of such a map~$g$.\\

\noindent
The two-torsion classes in $H_1(M;\Z)$ that are not Klein boundaries play 
no special role in the orientability problem in real Gromov-Witten theory, but unfortunately 
there appears to be no simple algebraic characterization of Klein boundaries.
Similarly, there appears to be no simple algebraic characterization of atorical classes
in $H^2(M;\Z_2)$, which are the most relevant to the orientability problem
in open and real Gromov-Witten theory, but Lemma~\ref{sqvan_lmm}
provides such a characterization for the smaller collection of spin classes
in many cases.
Since every element of $H_2(M;\Z)$ can be represented by a continuous map
$F\!:\Si\!\lra\!M$ for a closed oriented surface, a class $w\!\in\!H^2(M;\Z_2)$
is spin if and only if $w$ vanishes on the image of $H_2(M;\Z)$ in $H_2(M;\Z_2)$,
under the homomorphism induced by the surjective homomorphism $\Z\!\lra\!\Z_2$.
By the Universal Coefficient Theorem for Cohomology \cite[Theorem~53.1]{Mu2},
there is a split exact sequence
$$0\lra \Ext\big(H_1(M;\Z),\Z_2\big)\lra H^2(M;\Z_2)\lra
\Hom\big(H_2(M;\Z),\Z_2\big)\lra0\,.$$
Thus, the spin classes $w\!\in\!H^2(M;\Z)$ are the elements of the group
\BE{Extdfn_e}\Ext\big(H_1(M;\Z),\Z_2\big)\equiv \Hom\big(B_1(M),\Z_2\big)\big/\big\{\eta|_{B_1(M)}\!:\,
\eta\!\in\!\Hom(Z_1(M),\Z_2)\big\},\EE
where $B_1(M)$ and $Z_1(M)$ are the group of boundaries of 2-chains 
and the group of 1-cycles, respectively. 
If a one-cycle~$b$ represents a two-torsion element of~$H_1(M;\Z)$,
$2b\!\in\!B_1(M)$ and
$$\flr{w,[b]}=\ti{w}(2b)$$ for any $\ti{w}\!\in\!\Hom(B_1(M),\Z_2)$
representing~$w$.

\section{Equivariant cohomology}
\label{equivcoh_subs}

\noindent
In this section, we recall basic notions in equivariant cohomology,
in the case the group is~$\Z_2$, and apply them to real bundle pairs.
We classify the real bundle pairs over the torus and the Klein bottle,
with certain fixed-point-free involutions, and compute~$w_2^{\phi}$ 
of all rank~1 real bundle pairs over symmetric surfaces with fixed-point-free involutions;
see Lemma~\ref{T2KB_lmm} and Corollary~\ref{gbund_crl}, respectively.
We conclude with two examples illustrating the intriguing nature of~$w_2^{\phi}$ 
of real bundle pairs.

\subsection{Basic notions}
\label{equivdfn_e}

\noindent
The group $\Z_2$ acts freely on the contractible space $\bE\Z_2\!\equiv\!S^{\i}$
with the quotient $\bB\Z_2\!\equiv\!\R\P^{\i}$.
An involution $\phi\!:M\!\lra\!M$, where $M$ is a topological space,
corresponds to a $\Z_2$-action on~$M$.
We denote by
$$\bB_{\phi}M=\bE\Z_2\!\times_{\Z_2}\!M$$
the corresponding Borel construction and by
$$H^*_{\phi}(M)\equiv H^*(\bB_{\phi}M;\Z_2), \quad
 H_*^{\phi}(M)\equiv H_*(\bB_{\phi}M;\Z_2), \quad 
 H_*^{\phi}(M;\Z)\equiv H_*(\bB_{\phi}M;\Z)$$
the corresponding \textsf{$\Z_2$-equivariant cohomology} and \textsf{homology} of~$M$. 
Let 
\BE{BXfib_e} M\lra \bB_{\phi}M\lra \bB\Z_2=\R\P^{\i}\EE
be the fibration induced by the projection $p_1\!:\bE\Z_2\!\times\!M\!\lra\!\bE\Z_2$.\\

\noindent
If $(V,\ti\phi)\!\lra\!(M,\phi)$ is a real bundle pair,
$$\bB_{\ti\phi} V\equiv \bE\Z_2\!\times_{\Z_2}\!V\lra \bB_{\phi}M$$
is a real vector bundle;
this is the quotient of the vector bundle $p_2^*V\!\lra\!\bE\Z_2\!\times\!M$
by the natural lift of the free $\Z_2$-action on the base.
Let 
$$w_i^{\ti\phi}(V)\equiv w_i(\bB_{\ti\phi}V)\in H_{\phi}^i(M)$$
be the \textsf{$\Z_2$-equivariant Stiefel-Whitney classes of $V\!\lra\!M$}. 
If $M$ is a point and $V\!=\!\C\!=\!\R\!\oplus\!\fI\R$, 
$$\bB_{\ti\phi}V=\R\P^{\i}\!\times\!\R \oplus\cO_{\R\P^{\i}}(-1)\lra\R\P^{\i}\,,$$
where $\cO_{\R\P^{\i}}(-1)$ is the tautological line bundle;
thus, $w_1^{\ti\phi}(V)$ is the generator of $H^1_{\phi}(\pt)$
in this case.
The non-equivariant Stiefel-Whitney classes of~$V$ are recovered from 
the equivariant Stiefel-Whitney classes of~$V$ by restricting to the fiber
of the fibration~\eref{BXfib_e}.
If $f\!:\Si\!\lra\!M$ is a continuous map commuting with involutions~$c$ on~$\Si$
and~$\phi$ on~$M$, 
the involution~$\ti\phi$ on~$V$ induces an involution $f^*\ti\phi$ on~$f^*V$ 
lifting~$c$ and 
\BE{equivback_e} w_i^{f^*\ti\phi}(f^*V)=\{\bB_{\phi,c}f\}^*w_i^{\ti\phi}(V)\in H^i_c(\Si),\EE
where 
\BE{equivback_e2}
\bB_{\phi,c}f\!:\bB_c\Si\!\lra\!\bB_{\phi}M, \qquad
\{\bB_{\phi,c}f\}\big([e,z]\big)=\big[e,f(z)\big],\EE
is the map induced by~$f$.\\

\noindent
If an involution $c\!:\Si\!\lra\!\Si$ has no fixed points,
the projection $p_2\!:\bE\Z_2\!\times\!\Si\!\lra\!\Si$ 
descends to a fibration
\BE{Z2fib_e2}\bE\Z_2\lra \bB_c\Si\stackrel{q}{\lra} \Si/\Z_2\,.\EE
Since $\bE\Z_2$ is contractible, this fibration is a homotopy equivalence,
with a homotopy inverse provided by any section of~$q$;
in the case of the antipodal map~\eref{antip_e}, 
such a section is explicitly described in \cite[Section~2.2]{GZ}.
In particular, $q$~induces isomorphisms
\BE{Z2isom_e}q^*\!:H^*(\Si/\Z_2)\lra H^*_c(\Si), \qquad
q_*\!: H_*^c(\Si;\Z)\lra H_*(\Si/\Z_2;\Z).\EE
Any section of~$q$ embeds $\Si/\Z_2$ as a homotopy retract, and 
every two such sections are homotopic.
Thus, if $f\!:\Si\!\lra\!M$ is a continuous map commuting  
with the involutions~$c$ on~$\Si$ and~$\phi$ on~$M$, we also denote~by
$$\bB_{\phi,c}f\!:\Si/\Z_2\lra\bB_{\phi}M$$
the composition of $\bB_{\phi,c}f\!:\bB_c\Si\!\lra\!\bB_{\phi}M$ with
any section of~$q$; this is well-defined and unambiguous up~to homotopy.
If $\Si$ is a compact manifold and $c\!:\!\Si\!\lra\!\Si$ is smooth, let
$$[\Si]_{\Z_2}^c=[\Si/\Z_2]_{\Z_2}\in H_*^c(\Si) \qquad\hbox{and}\qquad
[f]_{\Z_2}^c= \{\bB_{\phi,c}f\}_*[\Si]_{\Z_2}^c \in H_*^{\phi}(M)$$
denote the \sf{$\Z_2$-fundamental class of~$\Si$} and its
\sf{equivariant push-forward}, respectively;
if $\Si/\Z_2$ is oriented, we similarly define 
$$[\Si]_{\Z}^c=[\Si/\Z_2]_{\Z} \in H_*^c(\Si;\Z) \qquad\hbox{and}\qquad
[f]_{\Z}^c= \{\bB_{\phi,c}f\}_*[\Si]_{\Z}^c \in H_*^{\phi}(M;\Z).$$
If $p\!:(V,\ti{c})\!\lra\!(\Si,c)$ is a real bundle pair, 
$V/\Z_2\!\lra\!\Si/\Z_2$ is a real vector bundle and
\begin{gather*}
\bB_{\ti\phi}V\lra q^*(V/\Z_2)\equiv
\big\{\big([e,x],[v]\big)\!\in\!\bB_c\Si\!\times\!(V/\Z_2)\!:\,
[x]\!=\![p(v)]\big\}, \\ 
[e,v]\lra \big([e,p(v)],[v]),
\end{gather*}
is a vector bundle isomorphism covering the identity on $\bB_c\Si$.
Thus,
\BE{Vquot_e} w_i^{\ti{c}}(V)=w_i\big(q^*(V/\Z_2)\big)=q^*w_i(V/\Z_2)
\in H^i_c(\Si;\Z_2).\EE 
We next recall two statements from~\cite{GZ}.

\begin{prp}[{\cite[Proposition~2.1]{GZ}}]\label{tensor_prp}
Let $(L_1,\ti\phi_1),(L_2,\ti\phi_2)\!\lra\!(M,\phi)$ be rank~1 real bundle pairs
over a topological space with an involution.
If $M$ is paracompact,
\BE{tensor_e}w_2^{\ti\phi_1\otimes_{\C}\ti\phi_2}(L_1\!\otimes_{\C}\!L_2)
 =w_2^{\ti\phi_1}(L_1) +w_2^{\ti\phi_2}(L_2).\EE
\end{prp}

\begin{crl}[{\cite[Corollary~2.4]{GZ}}]\label{SQrt_crl}
Let $(M,\phi)$ be a topological space with an involution and
$(L,\ti\phi)\!\lra\!(M,\phi)$ be a rank~1 real bundle pair.
\begin{enumerate}[label=(\arabic*),leftmargin=*]
\item\label{SQrtSC_it} If $M$ is simply connected and $w_2(L)\!=\!0$,  
$w_2^{\ti\phi}(L)$ is a square class.
\item\label{SQrt_it} If $M$ is paracompact and $(L,\ti\phi)$ admits a real square root,
$w_2^{\ti\phi}(L)=0$.
\end{enumerate}
\end{crl}

\subsection{Real bundle pairs over surfaces}
\label{classify_subs}

\noindent
The involution $\fa_{\bK}$ on the Klein bottle~$\bK$ given~by
$$\fa_{\bK}\!:\bK\!=\!\bI\!\times\!S^1/\!\!\sim \lra\bK,\quad [s,z]\lra[s,-z],
\qquad\hbox{where}\quad (1,z)\sim(0,\bar{z})~~\forall\,z\!\in\!S^1\subset\C,$$
has no fixed points.
The next lemma classifies real bundle pairs over $(\bK,\fa_{\bK})$
and $(\bT,\fa_{\bT})$, where $\fa_{\bT}\!=\!\id_{S^1}\!\times\!\fa$.

\begin{lmm}\label{T2KB_lmm}
Let $(\Si,c)\!=\!(\bT,\fa_{\bT}),(\bK,\fa_{\bK})$ and $n\!\in\!\Z^+$.
A rank~$n$ real bundle pair $(V,\ti{c})\!\lra\!(\Si,c)$ is isomorphic to the trivial 
real bundle pair $(\Si\!\times\!\C^n,c\!\times\!\fc_{\C^n})$ 
if and only~if 
$$\blr{w_2^{\ti{c}}(V),[\Si]^c}=
\begin{cases}0,&\hbox{if}~(\Si,c)\!=\!(\bT,\fa_{\bT}),\\
\binom{n}{2}\!+\!2\Z,&\hbox{if}~(\Si,c)\!=\!(\bK,\fa_{\bK}).
\end{cases}$$
\end{lmm}

\begin{proof}
(1) Since the torus case is addressed by \cite[Lemma~2.2]{GZ} and the proof of
\cite[Lemma~2.3]{GZ}, it is enough to consider the case
$(\Si,c)\!=\!(\bK,\fa_{\bK})$.
Let 
$$V_{\pm}=\big(\bI\!\times\!S^1\!\times\!\C\big)\big/\!\!\sim, \quad
(0,z,v)\sim(1,\bar{z},\pm v)~~~\forall~z\!\in\!S^1,\,v\!\in\!\C,$$
where the involutions $\ti{c}_{\pm}$ are induced by the standard conjugation on~$\C$.
By \cite[Lemma~2.3]{Teh}, every rank~$n$ real bundle pair over $(S^1,\fa)$ is trivial
and admits precisely two homotopy classes of isomorphisms covering~$\id_{S^1}$.
The non-trivial class contains the isomorphism given by the constant function on~$S^1$
with the value equal to the diagonal matrix with one entry $-1$ and the remaining entries~1.
By composing isomorphisms of the trivial rank~$n$ real bundle pair over~$(S^1,\fa)$
covering the conjugation~$\fc_{S^1}$ on~$S^1$ with $\fc_{S^1}\!\times\!\id_{\C^n}$,
we see that there are also precisely two homotopy classes of such isomorphisms.
Thus, $(V,\ti{c})$ is isomorphic~to one of the real bundle pairs 
$$(V,\ti{c})=(nV_+,n\ti{c}_+), (V_-\!\oplus\!(n\!-\!1)V_+,\ti{c}_-\oplus (n\!-\!1)\ti{c}_+);$$
note that $(nV_+,n\ti{c}_+)$ is the trivial rank~$n$ real bundle pair.  
These real bundle pairs canonically decompose into two $\Z_2$-equivariant real vector bundles,  
induced by the real and imaginary axes in~$\C$.
By \eref{Vquot_e},
$$\blr{w_2^{\ti{c}}(V),[\bK]^{\fa_{\bK}}}
=\blr{w_2(V/\Z_2),[\bK/\Z_2]} 
=\blr{w_2(V_{\R}\!\oplus\!V_{\fI\R}),[\bK/\Z_2]},$$
where
$$V_{\R}=n(V_+)_{\R},(V_-)_{\R}\!\oplus\!(n\!-\!1)(V_+)_{\R},\qquad
V_{\fI\R}=n(V_+)_{\fI\R},(V_-)_{\fI\R}\!\oplus\!(n\!-\!1)(V_+)_{\fI\R}$$
are the $\Z_2$-quotients of the real and imaginary parts of~$V$ over the Klein bottle
\BE{KBquot_e}\bK/\Z_2\equiv  \bI\!\times\!\bI\big/\!\!\sim, \qquad
(0,t)\sim(1,1\!-\!t),~~(s,0)\sim(s,1)\quad \forall~s,t\!\in\!\bI\,.\EE
We note that
\begin{alignat*}{3}
(V_{\pm})_{\R}&=\big(\bI\!\times\!\bI\!\times\!\R\big)\big/\!\!\sim,&\qquad
(0,t,v)&\sim(1,1\!-\!t,\pm v),~~(s,0,v)\sim(s,1,v)&\quad &\forall~s,t\!\in\!\bI,\,v\!\in\!\R,\\
(V_{\pm})_{\fI\R}&=\big(\bI\!\times\!\bI\!\times\!\R\big)\big/\!\!\sim,&\qquad
(0,t,v)&\sim(1,1\!-\!t,\pm v),~~(s,0,v)\sim(s,1,-v)&\quad &\forall~s,t\!\in\!\bI,\,v\!\in\!\R 
\end{alignat*} 
over the Klein bottle~\eref{KBquot_e}.\\

\noindent
(2) The first $\Z_2$-homology of the Klein bottle~\eref{KBquot_e} is generated by the loops
$$\al,\be\!:\bI\lra \bK/\Z_2\,, \qquad \al(t)=[0,t],\quad
\be(s)=[s,1/2],$$ 
with the intersections $\al^2\!=\!0$, $\al\!\cdot\!\be,\be^2\!=\!1$;
we denote the Poincar\'e duals of the homology classes represented by these loops
by the same symbols.
Since the restriction of $(V_-)_{\R}$ to~$\al$ is trivial and to~$\be$ is 
the Mobius band line bundle,
$$\blr{w_1((V_-)_{\R}),\al}=0, ~~~ \blr{w_1((V_-)_{\R}),\be}=1
\quad\Lra\quad w_1((V_-)_{\R})=\al\in H^1(\bK/\Z_2;\Z_2).$$
Since the restrictions of $(V_-)_{\fI\R}$ to~$\al$ and~$\be$ are 
the Mobius band line bundles,
$$\blr{w_1((V_-)_{\fI\R}),\al}=1, ~~~ \blr{w_1((V_-)_{\fI\R}),\be}=1
\quad\Lra\quad w_1((V_-)_{\fI\R})=\be\in H^1(\bK/\Z_2;\Z_2).$$
We conclude that 
$$w(V_-/\Z_2)=1+(\al\!+\!\be)+\al\be=1+(\al\!+\!\be)+\be^2\,.$$
On the other hand, the restriction of $(V_+)_{\fI\R}$ to~$\al$ is the Mobius
band line bundle and to~$\be$ is trivial.
Thus,
$$\blr{w_1((V_+)_{\fI\R}),\al}=1, ~~~ \blr{w_1((V_+)_{\fI\R}),\be}=0
\quad\Lra\quad w_1((V_+)_{\fI\R})=\al+\be\in H^1(\bK/\Z_2;\Z_2).$$
Since $(V_+)_{\R}$ is the trivial line bundle, we conclude that 
$$w(V_+/\Z_2)=1+(\al\!+\!\be)\,.$$
Putting the two conclusions together, we find that
\begin{equation*}\begin{split}
w_2\big(n(V_+/\Z_2)\big)&=\binom{n}{2}(\al\!+\!\be)^2
=\binom{n}{2}\be^2\,,\\
w_2\big((V_-\!\oplus\!(n\!-\!1)V_+)/\Z_2\big)
&=\binom{n}{2}(\al\!+\!\be)^2+\be^2
=\binom{n}{2}\be^2+\be^2.
\end{split}\end{equation*}
Since $\be^2\neq0$, this establishes the claim.
\end{proof}

\noindent 
Let $\eta\!:\P^1\!\lra\!\P^1$ be as in~\eref{tauetadfn_e} 
and $\ti\eta$ be the lift of~$\eta$ to the line bundle 
$\cO_{\P^1}(-2)\!\lra\!\P^1$ given~by
$$\ti\eta\big(\ell,(u,v)^{\otimes2}\big)
=\big(\eta(\ell),(-\bar{v},\bar{u})^{\otimes2}\big)
\qquad\forall\,\big(\ell,(u,v)\big)\in\cO_{\P^1}(-1)\subset\P^1\!\times\!\C^2 .$$

\begin{lmm}\label{P1T2_lmm}
With notation as above,
$$w_2^{\ti\eta}(\cO_{\P^1}(-2))\neq0\in H^2_{\eta}(\P^1).$$
\end{lmm}

\begin{proof}
The real~map $f_{\P^1,\bK}\!: (\bK,\fa_{\bK})\!\lra\!(\P^1,\eta)$ given~by
$$[s,z]\lra \big[\cos\frac{\pi s}{2}+\fI z\sin\frac{\pi s}{2},
\fI\sin\frac{\pi s}{2}+z\cos\frac{\pi s}{2}\big]
\qquad\forall\,s\!\in\!\bI, \,z\!\in\!S^1\subset\C$$
sends $0\!\times\!S^1$ to the circle $|u|\!=\!|v|$ in $\P^1$ and then spins
$S^1$ around the points $[1,\pm1]\!\in\!\P^1$ so that $1\!\times\!S^1$ is again mapped
to the circle $|u|\!=\!|v|$, but in the conjugate way.
Let $\ti{f}_{\P^1,\bK}\!=\!f_{\P^1,\bK}\!\circ\!q$, where $q\!:\bI\!\times\!S^1\!\lra\!\bK$
is the quotient map.
We trivialize the bundle $\ti{f}_{\P^1,\bK}^*\cO_{\P^1}(-2)$ by 
\begin{equation*}\begin{split}
\bI\!\times\!S^1\times\C&\lra \ti{f}_{\P^1,\bK}^*\cO_{\P^1}(-2),\\
(s,z,\la)&\lra \Big(s,z,\fI\bar{z}\la \big(\cos\frac{\pi s}{2}+\fI z\sin\frac{\pi s}{2},
\fI\sin\frac{\pi s}{2}+z\cos\frac{\pi s}{2}\big)^{\otimes2}\Big).
\end{split}\end{equation*}
The conjugation on $\bI\!\times\!S^1\times\C$ induced by~$\ti\eta$ 
via this trivialization is the standard one:
$$\bI\!\times\!S^1\times\C\lra \bI\!\times\!S^1\times\C,\qquad
(s,z,\la)\lra (s,-z,\bar\la).$$
Since 
$$f_{\P^1,\bK}^*\cO_{\P^1}(-2)=
\ti{f}_{\P^1,\bK}^*\cO_{\P^1}(-2)\big/\!\!\sim, \qquad
(1,z,\la)\sim(0,\bar{z},-\la)~~\forall\,(z,\la)\in S^1\!\times\!\C\,,$$
the real bundle pair $f_{\P^1,\bK}^*\cO_{\P^1}(-2)$ is 
$(V_-,\ti{c}_-)\!\lra\!(\bK,\fa_{\bK})$, in the notation of 
the proof of Lemma~\ref{T2KB_lmm}.
Thus, by the proof of Lemma~\ref{T2KB_lmm},
$$f_{\P^1,\bK}^*\,w_2^{\ti\eta}\big(\cO_{\P^1}(-2)\big)
= w_2^{\ti{c}_-}(V_-)\neq0.$$
This establishes the claim.
\end{proof}


\begin{crl}\label{gbund_crl}
Let $(\hat\Si,\si)$ be a symmetric surface so that $\hat\Si^{\si}\!=\!\eset$.
If $(L,\ti\phi)\!\lra\!(\hat\Si,\si)$ is a rank~1 real bundle pair, 
$$\blr{w_2^{\ti\phi}(L),[\hat\Si]_{\Z_2}^{\si}}=\frac12\lr{c_1(L),[\hat\Si]_{\Z}}+2\Z.$$
\end{crl}

\begin{proof}
By Lemma~\ref{P1T2_lmm} and Proposition~\ref{tensor_prp}, the claim holds for 
the real bundle pairs
$$(\cO_{\P^1}(-2a),\ti\eta^a)
    \equiv(\cO_{\P^1}(-2),\ti\eta)^{\otimes a}\lra(\P^1,\eta) $$
with $a\!\in\!\Z$.
By \cite[Propositions~4.2]{BHH}, these are all the rank~1 real bundle pairs 
over~$(\hat\P^1,\eta)$.
We thus assume that $\hat\Si\!\neq\!\P^1$ for the remainder of the proof.\\

\noindent
By \cite[Theorem~1.2]{Nat},
there exists a $\Z_2$-invariant circle $S\!\subset\!\hat\Si$.
Let $U\!\subset\!\hat\Si$ be a $\Z_2$-invariant tubular neighborhood of~$S$
and $\hat\Si'$ be the two-nodal surface obtained from~$\hat\Si$ by collapsing 
the boundary circles of~$U$.
The involution~$\si$ descends to an involution~$\si'$ on~$\hat\Si'$, 
which has no fixed points, so that the quotient map $q\!:\hat\Si\!\lra\!\hat\Si'$ 
intertwines the two involutions.
The image of~$\bar{U}$ in~$\hat\Si'$ is an irreducible component~$C$ of~$\hat\Si'$ 
homeomorphic to~$\P^1$ and preserved by~$\si'$; 
let $C'$ denote the remaining component of~$\hat\Si'$.
Given $a\!\in\!\Z$, let $(L',\ti\phi')\!\lra\!(\hat\Si',\si')$ be the real bundle pair so~that 
$$\blr{c_1(L'),[C]_{\Z}}=-2a, \qquad 
(L',\phi')|_{C'}=(C'\!\times\!\C,\si'|_{C'}\!\times\!\fc_{\C}).$$
By the previous paragraph, 
\begin{equation*}\begin{split}
\blr{w_2^{\ti\phi'}(L'),[\hat\Si']_{\Z_2}^{\si'}}
&=\blr{w_2^{\ti\phi'}(L'),[C]_{\Z_2}^{\si'}}
+\blr{w_2^{\ti\phi'}(L'),[C']_{\Z_2}^{\si'}}\\
&=\frac12\blr{c_1(L'),[C]_{\Z}}+0+2\Z
=\frac12\blr{c_1(L'),[\hat\Si']_{\Z}}+2\Z\,.
\end{split}\end{equation*}
Since the degree of $q$ is~1, it follows that 
\begin{equation*}\begin{split}
&\blr{w_2^{q^*\ti\phi'}(q^*L'),[\hat\Si]_{\Z_2}^{\si}}
=\blr{q^*w_2^{\ti\phi'}(L'),[\hat\Si]_{\Z_2}^{\si}}
=\blr{w_2^{\ti\phi'}(L'),q_*[\hat\Si]_{\Z_2}^{\si}}
=\blr{w_2^{\ti\phi'}(L'),[\hat\Si']_{\Z_2}^{\si'}}\\
&\qquad=\frac12\blr{c_1(L'),[\hat\Si']_{\Z}}+2\Z
=\frac12\blr{c_1(L'),q_*[\hat\Si]_{\Z}}+2\Z
=\frac12\blr{c_1(q^*L'),[\hat\Si]_{\Z}}+2\Z.
\end{split}\end{equation*}
This establishes the claim for the real bundle pairs 
$(L,\ti\phi)=q^*(L',\ti\phi')$, for each $a\!\in\!\Z$ as above.
By \cite[Propositions~4.2]{BHH}, these are all the rank~1 real bundle pairs 
over~$(\hat\Si,\si)$.
\end{proof}

\begin{crl}\label{equivSQrt_crl}
Suppose $(M,\phi)$ is a topological space with an involution and
$(L,\ti\phi)\!\lra\!(M,\phi)$ is a rank~1 real bundle pair so that 
$w_2^{\ti\phi}(L)\!\in\!H^2_{\phi}(M)$ is a spin class.
Let $\al\!:(S^1,\fa)\!\lra\!(M,\phi)$ be a real map.
If $f\!:\Si\!\lra\!M$ is a continuous map from an oriented bordered Riemann surface
such that $\prt f\!=\!\al$, then
$$\flr{w_2^{\ti\phi}(L),[\al]_{\Z_2}^{\phi}}=
\frac12\blr{\hat{f}^*c_1(L),[\hat\Si]_{\Z}}+2\Z,$$
where $\hat{f}\!: (\hat\Si,\si)\!\lra\!(M,\phi)$ is the double of~$f$. 
If $\Si$ is a disk, the conclusion holds even if 
\hbox{$w_2^{\ti\phi}(L)\!\in\!H^2_{\phi}(M)$} is just atorical.
\end{crl}

\begin{proof}
The map $\bB_{\phi,\si}\hat{f}\!:\hat\Si/\Z_2\!\lra\!\bB_{\phi}M$ takes 
the image of $\prt\Si$ in~$\hat\Si/\Z_2$, which represents the nonzero two-torsion
class in $H_1(\hat\Si/\si;\Z)$, to $\bB_{\phi,\fa}\al$.
Thus,
\begin{equation*}\begin{split}
\flr{w_2^{\ti\phi}(L),[\al]_{\Z_2}^{\phi}}
=\flr{w_2(\bB_{\ti\phi}L),[\bB_{\phi,\fa}\al]_{\Z_2}}
&=\blr{w_2(\bB_{\ti\phi}L),\{\bB_{\phi,\si}\hat{f}\}_*[\hat\Si/\Z_2]_{\Z_2}}\\
&=\lr{w_2^{\ti\phi}(L),\{\bB_{\phi,\si}\hat{f}\}_*[\hat\Si/\Z_2]_{\Z_2}}
=\lr{\hat{f}^*w_2^{\ti\phi}(L),[\hat\Si]_{\Z_2}^{\si}};
\end{split}\end{equation*}
see Section~\ref{topol_subs} for the second equality.
The claim now follows from Corollary~\ref{gbund_crl}.
\end{proof}

\subsection{Some examples}
\label{examples_sec}

\noindent
We now give two examples.
The first describes rank~1 real bundle pairs $(V,\ti\phi)$ over 
a simply connected space $(M,\phi)$ such that $w_2^{\ti\phi}(V)$ is a non-trivial square.
The second example describes  a rank~1 real bundle pair $(V,\ti\phi)$
so that $w_2^{\ti\phi}(V)$ is atorical, but not spin.
These examples imply that there is no simple description of the condition
of $w_2^{\ti\phi}(V)$ vanishing on the tori~$\bB_{\phi,\fa_{\bT}}f$;
see \eref{equivback_e2} and the beginning of Section~\ref{classify_subs} for the notation.

\begin{eg}\label{SQ_eg}
For each $m\!\in\!\Z$, define
\begin{equation*}\begin{split}
\eta_{2m-1}\!:\P^{2m-1}&\lra \P^{2m-1} \qquad\hbox{by}\\
\big[W_1,W_2,\ldots,W_{2m-1},W_{2m}\big]
&\lra\big[-\ov{W}_2,\ov{W}_1,\ldots,-\ov{W}_{2m},\ov{W}_{2m-1}\big].
\end{split}\end{equation*}
In particular, $\eta_1\!=\!\eta$.
Let $(\cO_{\P^{2m-1}}(-2),\ti\eta_{2m-1})\lra(\P^{2m-1},\eta_{2m-1})$ be
the real bundle pair given~by
$$\ti\eta_{2m-1}\big(\ell,(W_1,\ldots,W_{2m})^{\otimes2}\big)
=\big(\eta_{2m-1}(\ell),(-\ov{W}_2,\ov{W}_1,
\ldots,-\ov{W}_{2m},\ov{W}_{2m-1})^{\otimes2}\big).$$
Fix $a\!\in\!\Z$ and take
$$(M,\phi)=(\P^{2m-1},\eta_{2m-1}), \quad 
(V,\ti\phi)=(\cO_{\P^{2m-1}}(2a),\ti\eta_{2m-1})
\equiv(\cO_{\P^{2m-1}}(-2),\ti\eta_{2m-1})^{\otimes(-a)}\,.$$
By Corollary~\ref{SQrt_crl}\ref{SQrt_it}, $w_2^{\ti\phi}(V)$
is a square class.
Since the fixed locus of~$\phi$ is empty, the natural projection
$$\bB_{\phi}M\lra M/\Z_2$$ 
is a homotopy equivalence
and $w_2^{\ti\phi}(V)$ corresponds to $w_2(V/\Z_2)$.
We show that the rank~2 vector bundle
\BE{Pnbund_e} V/\Z_2\equiv \cO_{\P^{2m-1}}(2a)/\Z_2\lra M/\Z_2\equiv \P^{2m-1}/\Z_2\,\EE
is non-orientable, is non-split if $a\!\neq\!0$, 
and has a non-zero $w_2$ if $a$ is odd.
Since it is sufficient to verify these statements for the restriction of this bundle
to $\P^1/\Z_2\!=\!\R\P^2$, we can assume that $m\!=\!1$.
An orientation on $\cO_{\P^1}(2a)/\Z_2$ would lift to an orientation on 
$\cO_{\P^1}(2a)$
preserved by~$\ti\eta_1$; since $\ti\eta_1$ is orientation-reversing,
$\cO_{\P^1}(2a)/\Z_2$  is not orientable.
A splitting of $\cO_{\P^1}(2a)/\Z_2$ would induce a splitting of $\cO_{\P^1}(2a)$ 
into two real line bundles, each of which must be trivial, since $\P^1$ is simply connected.
Since the euler class of $\cO_{\P^1}(2a)$ is~$2a$, this is impossible if $a\!\neq\!0$,
and so $\cO_{\P^1}(2a)/\Z_2$ does not split as a sum of line bundles if $a\!\neq\!0$.
By Corollary~\ref{gbund_crl}, $w_2^{\ti\eta_1}(\cO_{\P^1}(2a))\!\neq\!0$ if $a$ is odd.
We note that the bundles~\eref{Pnbund_e} are pairwise distinct, since 
an isomorphism between a pair of them would lift
to an isomorphism of the bundles $\cO_{\P^1}(2a)$ as real vector bundles.
\end{eg}

\begin{eg}\label{nontoric_eg}
We now describe a rank~1 real bundle pair $(V,\ti\phi)\!\lra\!(M,\phi)$
so that $w_2^{\ti\phi}(V)$ vanishes on every homology class represented by 
a torus (and in particular on the tori of the form $\bB_{\phi,\fa_{\bT}}f$),
but not on the image of $H_2^{\phi}(M;\Z)$ in $H_2^{\phi}(M)$
and thus is not a square class or even spin.
Let $\pi\!:\Si_3\!\lra\!\Si_2$ be a double cover of a genus~2 closed oriented surface
by a  genus~3 surface and $\phi\!: \Si_3\!\lra\!\Si_3$ be the deck transformation
(which is orientation-preserving).
Since $\phi$ has no fixed points,
$$H^2_{\phi}(\Si_3)=H^2(\Si_2;\Z_2).$$
The trivial rank~1 real bundle pair $(\Si_3\!\times\!\C,\ti\phi_0)\!\lra\!(\Si_3,\phi)$ 
induces a rank~2 bundle over~$\Si_2$ which splits into the line bundles
$$L_{\R}\!\equiv\!(\Si_3\!\times\!\R)/\Z_2\approx\Si_2\!\times\!\R,\qquad
L_{\fI\R}\!\equiv\!(\Si_3\!\times\!\fI\R)/\Z_2\lra\Si_2.$$
The line bundle $L_{\fI R}$ is not trivial: it restricts to the Mobius band line bundle
along any loop in $\Si_2$ not in the image of a loop from~$\Si_3$.
Thus, there exists a line bundle $L\!\lra\!\Si_2$ such that
$$w_1(L_{{\fI}\R})w_1(L)\neq 0 \qquad\Lra\qquad
\qquad w_2\big((\Si_3\!\times\!\C)/\Z_2\otimes_{\R}\!L\big)\neq0.$$
The involution~$\phi$ naturally lifts to an involution~$\ti\phi_L$ on
$\pi^*L\!\lra\!\Si_3$.
We~take
$$(V,\ti\phi)=(\Si_3\!\times\!\C\otimes_{\R}\pi^*L,\ti\phi_0\!\otimes_{\R}\!\ti\phi_L)
\lra(\Si_3,\phi).$$
Since $w_2^{\ti\phi}(V)$ corresponds~to
$$w_2(V/\Z_2)=w_2\big((\Si_3\!\times\!\C)/\Z_2\otimes_{\R}\!L\big)\in H^2(M;\Si_2),$$
$w_2^{\ti\phi}(V)$ does not vanish on the image of $H_2^{\phi}(M;\Z)$ in $H_2^{\phi}(M)$.
However, it vanishes on the homology classes represented by tori, since
every map $\bT\!\lra\!\Si_2$ is of degree~0.
\end{eg}

\section{Relative signs}
\label{main_sec}

\noindent
This section applies the relative sign principle introduced in \cite[Section~6]{Ge2}
in the basic case $(\Si,c)\!=\!(D^2,\id_{S^1})$ to arbitrary oriented sh-surfaces~$(\Si,c)$.
For a real Cauchy-Riemann operator~$D$ on a real bundle pair $(V,\ti{c})$ over 
an oriented sh-surface~$(\Si,c)$ and a point $x_i\!\in\!(\prt\Si)_i$ on each boundary
component with $|c_i|\!=\!0$, it is convenient to define
\BE{wtdetDdfn_e}\wt\det(D)= \det(D)\otimes\hspace{-.3in}
\bigotimes_{\begin{subarray}{c}|c_i|=0\\ \lr{w_1(V^{\ti{c}}),(\prt\Si)_i}=0\end{subarray}}
\hspace{-.3in}\big(\La_{\R}^{\top}V_{x_i}^{\ti{c}}\big)^*\,;\EE
a similar, but not identical, twisted determinant line is introduced 
in \cite[Section~5]{Sol}.
If $(V,\ti{c})$ is induced from a real bundle pair over a smooth manifold $(X,\phi)$
with an orientable maximal totally real subbundle,
the additional factors in $\wt\det(D)$ are systematically orientable and can be essentially 
ignored 
(dropping them would change~\eref{epsConjTrar_e} by 
the sign of the permutation induced by~$h$ on the set of  boundary components
with $|c_i|\!=\!0$ and $\lr{w_1(V^{\ti\phi}),b_i}\!=\!0$).
We study the changes in the orientation of $\wt\det(D)$ under various operations on~$D$.\\

\noindent
The (twisted) orientations of the  determinant line of a real Cauchy-Riemann operator~$D$ 
and of the moduli space of $J$-holomorphic maps from an oriented sh-surface at 
a point~$u$ are determined by certain collections of trivializations,
which we call \sf{orienting collections}.
The conjugation operation of Section~\ref{conjsign_subs} transforms such a collection
to an orienting collection for the conjugate operator~$\bar{D}$.
There is a natural isomorphism between $\wt\det(D)$ and $\wt\det(\bar{D})$;
we call the sign of this isomorphism with respect to an orienting collection and its conjugate 
\sf{the relative sign of the conjugation}.
We compute it in Proposition~\ref{relsign_prp} below and  show that in particular 
it is independent of the choice of the initial orienting collection.
This computation is similar to the analogous computations in
\cite[Section~4]{FOOO9} and \cite[Section~2]{Sol}, but 
we do not assume the existence of any ambient structure.
This leads to more general orientability results and fits well with
studying local systems of orientations when the index bundles are not orientable.
Even in the basic case of the disk, the approach of comparing relative orientations
without imposing additional structure is more systematic, 
helps avoid the sort of mistake made in \cite[Proposition~11.5]{FOOOold}, 
which continued on as \cite[Proposition~2.1]{Cho},
and combines cases that appear to be unnecessarily separated in~\cite{FOOO9}.\\

\noindent
A diffeomorphism~$h$ of~$\Si$, possibly interchanging the boundary components of~$\Si$, 
likewise induces an isomorphism  between $\wt\det(D)$ and $\wt\det(h^*D)$.
If the operators $D$ and $h^*D$ are homotopic in a natural way,
the orienting collection for~$D$ can be transferred along a path~$\gm$ 
to an orienting collection for~$h^*D$.
We compute the sign of the natural isomorphism between $\wt\det(D)$ and $\wt\det(h^*D)$
with respect to the corresponding orientations in Proposition~\ref{transfsign_prp}. 
It is again independent of the initial orienting collection,
but does depend on the choice of the transferring path~$\gm$ 
unless an appropriate determinant line bundle is orientable.\\

\noindent
In Section~\ref{MSreal_subs}, we determine analogous signs for the determinant line
bundles of the moduli spaces of $J$-holomorphic maps from the disk and the annulus
with the five possible boundary involutions.
This allows us to establish more general versions of Theorems~\ref{g0orient_thm} 
and~\ref{g1orient_thm}.

\subsection{Conjugation}
\label{conjsign_subs}

\noindent
Let $(\Si,c)$ be an oriented sh-surface, $\fJ\!\in\!\cJ_c$,
and $D\!=\!\dbar\!+\!A$ be a compatible real Cauchy-Riemann operator on 
a real bundle pair $(V,\ti{c})$ over~$(\Si,c)$.
We denote by $\bar{V}$ the complex vector bundle obtained from~$V$ 
by multiplying the complex structure by~$(-1)$;
$(\bar{V},\ti{c})$ is still a real bundle pair over $(\Si,c)$.
Let
$$\bar{D}\!=\!\dbar\!+\!A\!:
\Ga(\Si;\bar{V})^{\ti{c}}\lra\Ga_{-\fJ}^{0,1}(\Si;\bar{V}).$$
Thus, $D$ and $\bar{D}$ are the same operators between the same real vector spaces
with different complex structures.\\

\noindent
By \cite[Proposition~5.4]{GZ}, an orientation of $\wt\det\, D$ is induced by
\begin{enumerate}[label=(OC\arabic*),leftmargin=*]
\item for each boundary component $(\prt\Si)_i$ with $|c_i|\!=\!0$, 
a choice of trivialization of the canonically oriented vector bundle $V^{\ti{c}}\!\oplus\!3\La^{\top}_{\R}V^{\ti{c}}$
over~$(\prt\Si)_i$,
\item for each boundary component $(\prt\Si)_i$ with $|c_i|\!=\!1$,
a choice of trivialization of the real bundle pair $(V,\ti{c})$ over~$(\prt\Si)_i$.  
\end{enumerate}
We will call such a collection of trivializations an \textsf{orienting collection}
for~$D$.
It induces an orienting collection of  trivializations for~$\bar{D}$: 
\begin{enumerate}[label=(OC\arabic*$_c$),leftmargin=*]
\item for each boundary component $(\prt\Si)_i$ with $|c_i|\!=\!0$, 
take the same choices of trivializations as for~$D$;
\item for each boundary component $(\prt\Si)_i$ with $|c_i|\!=\!1$,
compose the corresponding choice for~$D$ with the standard conjugation on~$\C^{\rk_{\C}V}$.
\end{enumerate}
We will call the latter collection of trivializations 
the \textsf{conjugate orienting collection}.\\

\noindent
Since $D$ and $\bar{D}$ have the same kernel and cokernel, the identity maps
$$\ker\,D\lra\ker\,\bar{D} \qquad\hbox{and}\qquad \cok\,D\lra\cok\,\bar{D}$$
induce an isomorphism 
\BE{conjisom_e}\wt\det(D)\lra\wt\det(\bar{D}).\EE
We call the sign of this isomorphism  with respect to the orientations induced 
by an orienting collection for~$D$ and its conjugate for~$\bar{D}$ 
\textsf{the relative sign of the conjugation on~$D$}. 
By the next proposition, which generalizes the conclusion of the main step in the proof of
 \cite[Lemma~5.1]{Ge2}, it is independent
of the orienting collection for~$D$ and of~$D$ itself.\\

\noindent
For a real vector bundle $W\!\lra\!(\prt\Si)_i$, we define 
\BE{tiw1dfn_e}\ti{w}_1\big(W,(\prt\Si)_i\big) =\begin{cases}
1,&\hbox{if}~\lr{w_1(W),[(\prt\Si)_i]_{\Z_2}}\neq0;\\
0,&\hbox{if}~\lr{w_1(W),[(\prt\Si)_i]_{\Z_2}}=0.
\end{cases}\EE

\begin{prp}\label{relsign_prp}
Let $(\Si,c)$ be a genus~$g$ oriented sh-surface, $\fJ\!\in\!\cJ_c$, and
$D$  be a real Cauchy-Riemann operator on 
a real bundle pair $(V,\ti{c})$ over~$(\Si,c)$ compatible with~$\fJ$.
The relative sign of the conjugation on~$D$
is $(-1)^{\eps_D}$, where 
\BE{epsD_e}\eps_D=\frac{1}{2}\bigg(\mu(V,\ti{c})
+\sum_{|c_i|=0}\ti{w}_1\big(V^{\ti{c}},(\prt\Si)_i\big)\bigg)
+\big(1\!-\!g+|c|_0\!+\!|c|_1\big)\rk_{\C}V+2\Z\in\Z_2\,,\EE
for every orienting collection for~$D$.
\end{prp}

\begin{proof}
Let $n\!=\!\rk\,V$.
By the proof of \cite[Theorem~1.1]{GZ}, $\det(D)$ is oriented by first pinching a 
circle near each boundary component $(\prt\Si)_i$ with $|c_i|\!=\!1$ to obtain 
\begin{enumerate}[label=$\bullet$,leftmargin=*]
\item a real Cauchy-Riemann operator~$D'$ on a real bundle pair $(V',\ti{c}')$ over 
an oriented sh-surface $(\Si',\fJ')$ with the boundary components $(\prt\Si)_i$
with $|c_i|\!=\!0$ and an interior marked point $p_i\!\in\!\Si'$ for each  
boundary component $(\prt\Si)_i$ with $|c_i|\!=\!1$ such~that 
$$g(\Si')=g(\Si), \qquad
\big(V',\ti{c}'\big)\big|_{(\prt\Si)_i}=\big(V,\ti{c}\big)\big|_{(\prt\Si)_i}
~~\hbox{if}~|c_i|=0, 
\quad\hbox{and}\quad  \mu(V',\ti{c}')=\mu(V,\ti{c}),$$
\item for each boundary component $(\prt\Si)_i$ with $|c_i|\!=\!1$,
a real Cauchy-Riemann operator~$D_i$ on a real bundle pair $(V_i,\ti{c}_i)$ over 
the disk with boundary $(\prt\Si)_i$ and  boundary involution~$c_i$
such~that 
$$(V_i,\ti{c}_i)=(V,\ti{c})\big|_{(\prt\Si)_i}\ \qquad\hbox{and}\qquad \mu(V_i,\ti{c}_i)=0.$$
\end{enumerate}
The lines $\det(D)$ and $\det(\bar{D})$ are then oriented via isomorphisms
\BE{Dsplit_e1}\begin{split} 
\det(D)&\approx  \det(D')\otimes
\bigotimes_{|c_i|=1}\!\!\big(\det(D_i)\!\otimes\!\La_{\R}^{\top}(V'_{p_i})^*\big),\\
\det(\bar{D})&\approx  \det(\bar{D}')\otimes
\bigotimes_{|c_i|=1}\!\!\big(\det(\bar{D}_i)\!\otimes\!\La_{\R}^{\top}(\bar{V}'_{p_i})^*\big).
\end{split}\EE
For each $i$ with $|c_i|\!=\!1$, $\det(D_i)$ is oriented by homotoping $D_i$ 
to the standard real $\dbar$-operator on $D^2\!\times\C^n$ with the antipodal boundary
involution via the chosen trivialization of~$V|_{(\prt\Si)_i}$.
The latter operator is surjective and its kernel consists of constant functions with values 
in $\R^n\!\subset\!\C^n$; this determines an orientation on~$\det(D_i)$.
In particular, the identity map from $\det(D_i)$ to $\det(\bar{D}_i)$ is 
orientation-preserving for each~$i$ with $|c_i|\!=\!1$.
On the other hand, $\La_{\R}^{\top}V'_{p_i}$ and $\La_{\R}^{\top}\bar{V}'_{p_i}$
are oriented by the complex orientations of $V'$ and $\bar{V}'$ and so
the sign of the identity isomorphism between the last factors in~\eref{Dsplit_e1} is $(-1)^n$ for
each~$i$ with $|c_i|\!=\!1$. This accounts for $|c|_1n$ in~\eref{epsD_e}.\\

\noindent
By the proof of \cite[Theorem~1.1]{Ge}, $\det(D')$ and $\det(\bar{D}')$ 
in \eref{Dsplit_e1} are oriented via isomorphisms
\BE{Dsplit_e2}\begin{split} 
\det(D')\otimes \det(D_1)^{\otimes4} &\approx \det(D_{+3})\otimes \det(D_1),\\
\det(\bar{D}')\otimes \det(\bar{D}_1)^{\otimes4} &\approx 
\det(\bar{D}_{+3})\otimes \det(\bar{D}_1),
\end{split}\EE  
where $D_1$ and $D_{+3}$ are real Cauchy-Riemann operators on the bundles
$$V_1\!\equiv\!\La_{\C}^{\top}V',~V_{+3}\!\equiv\!V'\oplus3\La_{\C}^{\top}V'\lra\Si'$$
with the involutions $\ti{c}_1$ and $\ti{c}_{+3}$ induced by~$\ti{c}'$ and 
$\bar{D}_1$ and $\bar{D}_{+3}$ are the conjugates of~$D_1$ and~$D_{+3}$,
respectively.
The identity map between the second factors on the left-hand sides of the two expressions
in~\eref{Dsplit_e2} is clearly orientation-preserving.\\

\noindent
By the proof of \cite[Proposition~3.1]{Ge}, $\det(D_{+3})$ is oriented by first pinching 
a  circle near each boundary component $(\prt\Si')_i$ to obtain 
\begin{enumerate}[label=$\bullet$,leftmargin=*]
\item a real Cauchy-Riemann operator~${D}_{+3}^{\cl}$ on a bundle $V_{+3}^{\cl}$ over 
a closed Riemann surface $\Si^{\cl}$ with a marked point $p_i$ for each  
boundary component $(\prt\Si')_i$ such~that 
$$g(\Si^{\cl})=g(\Si')=g(\Si) \qquad\hbox{and}\qquad
2\deg\big(V_{+3}^{\cl}\big)=\mu\big(V_{+3},\ti{c}_{+3}\big)=4\mu\big(V,\ti{c}\big)\,,$$
\item for each boundary component $(\prt\Si')_i$,
the standard Cauchy-Riemann operator~$D_{+3;i}$ on the trivial bundle $D^2\!\times\!\C^{n+3}$
obtained by extending the chosen trivialization~of  
$$V'^{\ti{c}'}\oplus 3\La^{\top}_{\R}V'^{\ti{c}'}=
V^{\ti{c}}\!\oplus\!3\La^{\top}_{\R}V^{\ti{c}}$$
over $(\prt\Si)_i$ to a tubular neighborhood.
\end{enumerate}
The lines $\det(D_{+3})$ and $\det(\bar{D}_{+3})$ are then oriented via isomorphisms
\BE{Dsplit_e5}\begin{split} 
\det(D_{+3})&\approx  \det(D_{+3}^{\cl})\otimes
\bigotimes_{|c_i|=0}\!\!\big(\det(D_{+3;i})\!\otimes\!\La_{\R}^{\top}(V_{+3}^{\cl}|_{p_i})^*\big),\\
\det(\bar{D}_{+3})&\approx  \det(\bar{D}_{+3}^{\cl})\otimes
\bigotimes_{|c_i|=0}\!\!\big(\det(\bar{D}_{+3;i})\!\otimes\!
\La_{\R}^{\top}(\bar{V}_{+3}^{\cl}|_{p_i})^*\big).
\end{split}\EE
The line $\det(D_{+3}^{\cl})$ is oriented by deforming $D_{+3}^{\cl}$
to a complex Cauchy-Riemann operator on~$V^{\cl}$.
Thus,  the sign of the identity isomorphism
between the first factors on the right-hand sides in~\eref{Dsplit_e5} 
contributes
$$\ind_{\C}D_{+3}^{\cl}=\deg V_{+3}^{\cl}+\big(1\!-\!g(\Si^{\cl})\big)(n\!+\!3)
=2\mu\big(V,\ti{c}\big)+\big(1\!-\!g(\Si)\big)(n\!+\!3)$$
to~\eref{epsD_e}.
For each $i$ with $|c_i|\!=\!0$, $D_{+3;i}$ is a surjective operator and its kernel 
consists of constant functions with values in $\R^{n+3}\!\subset\!\C^{n+3}$; 
this determines an orientation on~$\det(D_{+3;i})$.
In particular, the identity map from $\det(D_{+3;i})$ to $\det(\bar{D}_{+3;i})$ is 
orientation-preserving for each~$i$ with $|c_i|\!=\!0$.
On the other hand, $\La_{\R}^{\top}V_{+3}^{\cl}|_{p_i}$ and
$\La_{\R}^{\top}\bar{V}_{+3}^{\cl}|_{p_i}$
are oriented by the complex orientations of $V_{+3}^{\cl}$ and $\bar{V}_{+3}^{\cl}$ 
and so the sign of the identity isomorphism
between the last factors in~\eref{Dsplit_e5} is $(-1)^{n+3}$ for
each~$i$ with $|c_i|\!=\!0$. This contributes $|c|_0(n\!+\!3)$ to~\eref{epsD_e}.\\

\noindent
By the proof of \cite[Proposition~3.2]{Ge}, $\det(D_1)$ is also oriented by first pinching 
a  circle near each boundary component $(\prt\Si')_i$ to obtain 
\begin{enumerate}[label=$\bullet$,leftmargin=*]
\item a real Cauchy-Riemann operator~${D}_1^{\cl}$ on a line bundle $V_1^{\cl}$ over 
$\Si^{\cl}$ such~that 
$$2\deg\,V_1^{\cl}=\mu(V',\ti{c}')-\sum_{|c_i|=0}\!\!
\ti{w}_1\big(V'^{\ti{c}},(\prt\Si')_i\big)
=\mu(V,\ti{c})-\sum_{|c_i|=0}\!\!\ti{w}_1\big(V^{\ti{c}},(\prt\Si)_i\big)\,,$$
\item for each boundary component $(\prt\Si')_i$,
a real Cauchy-Riemann operator~$D_{1;i}$ 
on a real bundle pair $(V_{1;i},\ti{c}_{1;i})$ over 
the disk with boundary~$(\prt\Si)_i$ and trivial boundary involution
such~that 
$$(V_{1;i},\ti{c}_{1;i})=(V_1,\ti{c}_1)\big|_{(\prt\Si)_i}\ 
\qquad\hbox{and}\qquad \mu(V_{1;i},\ti{c}_{1;i})\in\{0,1\}.$$
\end{enumerate}
The lines $\det(D_1)$ and $\det(\bar{D}_1)$ are then oriented via isomorphisms
\BE{Dsplit_e7}\begin{split} 
\det(D_1)&\approx  \det(D_1^{\cl})\otimes
\bigotimes_{|c_i|=0}\!\!\big(\det(D_{1;i})\!\otimes\!(V_1^{\cl}|_{p_i})^*\big),\\
\det(\bar{D}_1)&\approx  \det(\bar{D}_1^{\cl})\otimes
\bigotimes_{|c_i|=0}\!\!\big(\det(\bar{D}_{1;i})\!\otimes\!
(\bar{V}_1^{\cl}|_{p_i})^*\big).
\end{split}\EE
The line $\det(D_1^{\cl})$ is oriented by deforming $D_1^{\cl}$
to a complex Cauchy-Riemann operator on the line bundle~$V_1^{\cl}$.
Thus,  the sign of the identity isomorphism
between the first factors on the right-hand sides in~\eref{Dsplit_e7} 
contributes
$$\ind_{\C}D_1^{\cl}=\deg V_1^{\cl}+1\!-\!g(\Si^{\cl})
=\frac12\bigg(\mu(V,\ti{c})-\sum_{|c_i|=0}\!\!\ti{w}_1\big(V^{\ti{c}},(\prt\Si)_i\big)\bigg)
+1\!-\!g(\Si)$$
to~\eref{epsD_e}.
For each $i$ with $|c_i|\!=\!0$,
the identity isomorphism between the last factors in~\eref{Dsplit_e7} is 
orientation-reversing because the lines
$V_1^{\cl}|_{p_i}$ and $\bar{V}_1^{\cl}|_{p_i}$
are oriented by their complex orientations. 
This contributes $|c|_0$ to~\eref{epsD_e}.\\

\noindent
For each $i$ with $|c_i|\!=\!0$ and $\mu(V_{1;i},\ti{c}_{1;i})=0$,
i.e.~$\ti{w}_1(V^{\ti{c}},(\prt\Si)_i)\!=\!0$,
the operator $D_{1;i}$ is surjective and the evaluation homomorphism
$$\ker D_{1;i}\lra V_{1;i}^{\ti{c}_{1;i}}\big|_{x_i}
=\La^{\top}_{\R}V^{\ti{c}}\big|_{x_i}$$
is an isomorphism and induces the same orientations on $\det(D_{1;i})$
and $\det(\bar{D}_{1;i})$ from any orientation of the right-hand side.
For each $i$ with $|c_i|\!=\!0$ and $\mu(V_{1;i},\ti{c}_{1;i})=1$,
i.e.~$\ti{w}_1(V^{\ti{c}},(\prt\Si)_i)\!=\!1$,
the operator $D_{1;i}$ is surjective and the evaluation homomorphism
\BE{2ptisom_e}\ker D_{1;i}\lra
V_{1;i}^{\ti{c}_{1;i}}\big|_{x_i} \oplus
V_{1;i}^{\ti{c}_{1;i}}\big|_{x_i'}
=
\La^{\top}_{\R}V^{\ti{c}}\big|_{x_i}\oplus 
\La^{\top}_{\R}V^{\ti{c}}\big|_{x_i'}\EE
is an isomorphism for any $x_i'\!\in\!(\prt\Si)_i\!-\!x_i$.
This isomorphism again induces orientations on $\det(D_{1;i})$
and $\det(\bar{D}_{1;i})$ from orientations of the right-hand side.
However, in this case, the two orientations of the second component 
on the right-hand side of~\eref{2ptisom_e} are opposite, since they are obtained
by transporting the same orientation of the first component in the positive 
direction along $(\prt\Si)_i$ with respect to the orientations induced
by $\fJ$ for  $\det(D_{1;i})$ and $-\fJ$ for~$\det(\bar{D}_{1;i})$.
Thus, in either case, the isomorphism between the first components in
the $i$-th factors in~\eref{Dsplit_e7}   contributes
$\ti{w}_1(V^{\ti{c}},(\prt\Si)_i)$ to~\eref{epsD_e}.
This completes the proof.
\end{proof}

\noindent
In the case $|c|_1\!=\!0$, \eref{epsD_e} agrees with the pin$^+$  formula of
\cite[Proposition~2.12]{Sol}.
In turn, the latter specializes to 
\cite[Theorem~1.3]{FOOO9} whenever $\Si$ is the disk and $V^{\ti{c}}$ is 
orientable.

\subsection{Interchange of boundary components}
\label{bdintersign_subs}

\noindent
Let $(\Si,c)$ be an oriented sh-surface, $h\!:\Si\!\lra\!\Si$ be a diffeomorphism
such that $h\!\circ\!c=c\!\circ\!h$ on~$\prt\Si$, and
$|h|,\sgn_h\!\in\!\{0,1\}$ be as at the beginning of Section~\ref{DMsign_subs}.
We define $\sgn_h^0\!\in\!\{0,1\}$ to be the sign of the permutation 
induced by~$h$ on the~set of boundary components~$(\prt\Si)_i$ with $|c_i|\!=\!0$. 
For $k\!=\!0,1$, let
$$\cT_{h;k}=\bI\!\times\!\bigsqcup_{|c_i|=k}\!\!(\prt\Si)_i\big/\!\!\sim, 
\qquad (1,z)\sim\big(0,c^{|h|}(h(z))\big)~~\forall\,z\!\in\!
\bigsqcup_{|c_i|=k}\!\!(\prt\Si)_i\,,;$$
where $c^0\!\equiv\!\id$;
this is a union of tori if $|h|\!=\!0$ and 
of tori  and Klein bottles if $|h|\!=\!1$.
The boundary involutions~$c_i$ induce an involution~$c_{h;k}$ on~$\cT_{h;k}$,
which is trivial if $k\!=\!0$ and has no fixed points if $k\!=\!1$.\\

\noindent
If  $(\Si,c)$ and $h$ are as above, $(X,\phi)$ is a manifold with an involution, 
$\k\!\equiv\!(k_1,\ldots,k_{|c|_0+|c|_1})$ is a tuple of non-negative integers,
and $\b$ is as in~\eref{btuple_eq},
we will~call
$$(u_0,\x_0,\fJ_0),(u_1,\x_1,\fJ_1)\in \fB_{\k}(X,\b)^{\phi,c}\!\times\!\cJ_c$$
\textsf{$h$-related} if 
$$u_1=\phi^{|h|}\!\circ\!u_0\!\circ\!h, \qquad h(\x_1)=\x_0,
\quad\hbox{and}\quad \fJ_1=(-1)^{|h|}h^*\fJ_0.$$
If $\gm\!\equiv\!(u_t,\x_t,\fJ_t)$ is any path in $\fB_{\k}(X,\b)^{\phi,c}\!\times\!\cJ_c$ 
such that $(u_0,\x_0,\fJ_0)$ and $(u_1,\x_1,\fJ_1)$ are $h$-related and $k\!=\!0,1$, let
$[\gm]_{h;k}\!\in\!H_2(X;\Z_2)$ denote the push-forward of the 
fundamental homology class of~$\cT_{h;k}$ with $\Z_2$-coefficients
by the continuous~map 
$$\gm_{h;k}\!:\,\cT_{h;k}\lra X, \qquad [t,z]\lra u_t(z)~~\forall\,[t,z]\!\in\!\cT_{h;k}\,;$$
this map is $\Z_2$-equivariant.
Let $[\gm]_{h;1}^c\!\in\!H_2^{\phi}(X)$
denote the equivariant push-forward of the $\Z_2$-equivariant fundamental 
class $[\cT_{h;1}]_{\Z_2}^{c_{h;1}}\!\in\!H_2^{c_{h;1}}(\cT_{h;1})$ by~$\gm_{h;1}$.\\

\noindent
Let $(\Si,c)$, $h$, $(X,\phi)$, and $\gm$ be as above,
$(V,\ti\phi)\!\lra\!(X,\phi)$ be a real bundle pair, 
$D_0$ and $D_1$ be the real Cauchy-Riemann operators on the real bundles
pairs $u_0^*(V,\ti\phi)$ and $u_1^*(V,\ti\phi)$ induced
from a connection~$\na$ and a deformation~$A$ on~$V$
as in 
the last paragraph of Section~\ref{analys_subs}, and
$$(V^h,D_0^h)=\begin{cases}(V,D_0),&\hbox{if}~h~\hbox{is orientation-preserving};\\
(\bar{V},\bar{D}_0),&\hbox{if}~h~\hbox{is orientation-reversing}.
\end{cases}$$
If $(u_0,\fJ_0)$ and $(u_1,\fJ_1)$ are $h$-related,
\BE{huphi_e}h^{-1}\!:(\Si,(-1)^{|h|}\fJ_0)\lra (\Si,\fJ_1) \qquad\hbox{and}\qquad 
u_0^*\ti\phi^{|h|}\!:u_0^*(V^h,\ti\phi)\lra u_1^*(V,\ti\phi)\EE
are isomorphisms and thus induce an isomorphism
\BE{wtDisom_e}\wt\det(D_0^h)\equiv  \det(D_0^h)\otimes\hspace{-.4in}
\bigotimes_{\begin{subarray}{c}|c_i|=0\\ \lr{u_0^*w_1(V^{\ti\phi}),(\prt\Si)_i}=0\end{subarray}}
\hspace{-.4in}\big(\La_{\R}^{\top}V_{u_0(x_i)}^{\ti\phi}\big)^*
  \lra \wt\det(D_1)\equiv 
  \det(D_1)\otimes\hspace{-.4in}
\bigotimes_{\begin{subarray}{c}|c_i|=0\\ \lr{u_1^*w_1(V^{\ti\phi}),(\prt\Si)_i}=0\end{subarray}}
\hspace{-.4in}\big(\La_{\R}^{\top}V_{u_1(x_i)}^{\ti\phi}\big)^*,\EE
where $x_i$ is the first marked point on the boundary component $(\prt\Si)_i$.
The path~$\gm$ can be used to transfer an orienting collection for~$D_0$
to an orienting collection for~$D_1$,
which we will call \textsf{the $\gm$-transferred orienting collection}.
The next proposition describes the sign of the above isomorphism.

\begin{prp}\label{transfsign_prp}
Suppose $(X,\phi)$ is a manifold with an involution, 
$(V,\ti\phi)\!\lra\!(X,\phi)$ is a real bundle pair,
$(\Si,c)$ is an oriented sh-surface,
$\k\!\equiv\!(k_1,\ldots,k_{|c|_0+|c|_1})$ is a tuple of non-negative integers,
$\b$ is as in~\eref{btuple_eq},
$\gm\!\equiv\!(u_t,\x_t,\fJ_t)$ is a path in $\fB_{\k}(X,\b)^{\phi,c}\!\times\!\cJ_c$,
and $h\!:\Si\!\lra\!\Si$ is a diffeomorphism such~that
$(u_0,\x_0,\fJ_0)$ and $(u_1,\x_1,\fJ_1)$ are $h$-related.
If $V^{\ti\phi}\!\lra\!X^{\phi}$ is not orientable, assume that
$k_i\!>\!0$ for every boundary component $(\prt\Si)_i$ such that $|c_i|\!=\!0$ 
and $\lr{w_1(V^{\ti\phi}),b_i}\!=\!0$.
Denote by $D_0$ and $D_1$ the real Cauchy-Riemann operators on the bundle
pairs $u_0^*(V,\ti\phi)$ and $u_1^*(V,\ti\phi)$ induced as in Section~\ref{analys_subs}
and choose an orienting collection for~$D_0$.
The sign of the isomorphism~\eref{wtDisom_e} is $(-1)^{\eps^{\ti\phi}_{\gm,h}}$, where
\BE{epsConjTrar_e}\begin{split}
\eps^{\ti\phi}_{\gm,h}=\blr{w_2(V^{\ti\phi}),[\gm]_{h;0}}
+\blr{w_2^{\La^{\top}_{\C}\ti\phi}(\La^{\top}_{\C}V),[\gm]_{h;1}^c}
+(\rk_{\C}V)\sgn_h+\sgn^0_h+2\Z\in\Z_2\,,
\end{split}\EE
if $\wt\det(\bar{D}_0)$  and $\wt\det(D_1)$ are oriented using the 
conjugate and $\gm$-transferred orienting collections, respectively. 
\end{prp}

\begin{proof}
Let $n\!=\!\rk_{\C}V$. 
The action of $h$ on the set of boundary components
of~$\Si$ preserves the~subsets
\BE{hactsplit_e}\big\{(\prt\Si)_i\!:\,|c_i|\!=\!0\big\} \qquad\hbox{and}\qquad
\big\{(\prt\Si)_i\!:\,|c_i|\!=\!1\big\}\EE
and thus breaks these subsets into cycles corresponding to topological components
of~$\cT_{h;0}$ and~$\cT_{h;1}$.
Since the maps~$h^{-1}$ and $u_0^*\ti\phi^{|h|}$ in~\eref{huphi_e}
are holomorphic and $\C$-linear, respectively, it is sufficient to compare
the orientation of the right-hand side of~\eref{wtDisom_e} 
with the orientation induced by the push-forward of 
the orienting collection for $D_0^h$ by $(h^{-1},u_0^*\ti\phi^{|h|})$.\\

\noindent
For a subset $(\prt\Si)_{i_1},\ldots,(\prt\Si)_{i_k}$ of boundary components of~$\Si$
such that $|c_{i_1}|\!=\!0$ and
\BE{BndComp_e}
h\big((\prt\Si)_{i_j}\big)=(\prt\Si)_{i_{j+1}}\qquad\forall\,~j\!=\!1,\ldots,k\,,\EE
where $i_{k+1}\!\equiv\!i_1$, let 
\BE{cTcomp_e}  \cT=\bI\!\times\!\bigsqcup_{j=1}^k\!(\prt\Si)_{i_j}\big/\!\!\sim, 
\qquad (1,z)\sim\big(0,c^{|h|}(h(z))\big)~~\forall\,z\!\in\!
\bigsqcup_{j=1}^k\!(\prt\Si)_{i_j};\EE
this is a connected component of $\cT_{h;0}$.
The maps~$h^{-1}$ and $u_0^*\ti\phi^{|h|}$ in~\eref{huphi_e} induce trivializations~of
\BE{LaRcomp_e2} 
\big(u_1^*V^{\ti\phi}\!\oplus\!3\La_{\R}^{\top}u_1^*V^{\ti\phi}\big)\big|_{(\prt\Si)_{i_j}}
=h^*\big(u_0^*V^{\ti\phi}\!\oplus\!3\La_{\R}^{\top}u_0^*V^{\ti\phi}\big)\big|_{(\prt\Si)_{i_j}}
\EE
from the original orienting collection for~$D_0$;
if $|h|\!=\!1$, this is the same trivialization as the one induced from
the conjugate trivialization for~$D_0^h\!=\!\bar{D}_0$. 
Since $\pi_1(\SO(n\!+\!3))\!=\!\Z_2$,
there are two homotopy classes of such trivializations.
If the $\gm$-transferred trivialization of the left-hand side in~\eref{LaRcomp_e2} 
agrees with  the original trivialization of the right-hand side up to homotopy, 
let $\eps_j\!=\!0$;
otherwise, let $\eps_j\!=\!1$.
The oriented vector bundle 
$\gm_{h;0}^*(V^{\ti\phi}\!\oplus\!3\La_{\R}^{\top}V^{\ti\phi})|_{\cT}$ 
is then isomorphic~to
$$\bI\!\times\!\bigsqcup_{j=1}^k\!(\prt\Si)_{i_j}\!\times\!\R^{n+3}\big/\!\!\sim, 
\qquad (1,z,v)\sim\big(0,h(z),g(z)v\big)~~
\forall\,z\!\in\!\bigsqcup_{j=1}^k\!(\prt\Si)_{i_j}\,,~
v\!\in\!\R^{n+3}\,,$$
for some $g\!:\bigsqcup_{j=1}^k\!(\prt\Si)_{i_j}\!\lra\!\SO(n\!+\!3)$ such that 
$g|_{(\prt\Si)_{i_j}}$ is homotopic to the constant map~$\bI_{n+3}$ if and only~if
$\eps_j\!=\!0$.
Since the oriented vector bundles over $\cT$ (of rank at least~3) are classified by 
the value of their~$w_2$ in $H^2(\cT;\Z_2)\!\approx\Z_2$, 
\BE{eps1bsum_e}\eps_1+\ldots+\eps_k
=\blr{\gm_{h;0}^*w_2(V^{\ti\phi}\!\oplus\!3\La_{\R}^{\top}V^{\ti\phi}),[\cT]_{\Z_2}}
=\blr{\gm_{h;0}^*w_2(V^{\ti\phi}),[\cT]_{\Z_2}}.\EE
By \cite[Proposition~5.4]{GZ}, changing the choice~(OC1) for a given boundary 
component $(\prt\Si)_i$ with $|c_i|\!=\!0$ changes the orientation of 
the determinant line bundle.
Since the two choices~(OC1) for~$D_1$
obtained from the trivializing collection for~$u_0$ via~$h$
and $\gm$-transfer agree if and only~if $\eps_j\!=\!0$, 
the possible differences in these choices for $(\prt\Si)_{i_1},\ldots,(\prt\Si)_{i_k}$
contribute~\eref{eps1bsum_e} to~\eref{epsConjTrar_e}.
Summing over all cycles of the action of~$h$ on the set $\{(\prt\Si)_i\!:|c_i|\!=\!0\}$
gives the first term on the right-hand side of~\eref{epsConjTrar_e}.
However, $h$ interchanges the order of the $(n\!+\!3)$-dimensional target spaces of
the evaluation isomorphisms orienting $\det(D_{+3;i})$ as below~\eref{Dsplit_e5}.
This contributes $(n\!+\!3)$ times the sign of the permutation on 
$\{(\prt\Si)_i\!:|c_i|\!=\!0\}$ induced by~$h$.
\\

\noindent
For a subset $(\prt\Si)_{i_1},\ldots,(\prt\Si)_{i_k}$ of boundary components of~$\Si$
such that $|c_{i_1}|\!=\!1$ and \eref{BndComp_e} holds, let~$\cT$ be as in~\eref{cTcomp_e}. 
The maps~$h^{-1}$ and $u_0^*\ti\phi^{|h|}$ in~\eref{huphi_e} induce trivializations~of
$u_1^*(V,\ti\phi)\big|_{(\prt\Si)_i}$ from the orienting collection for~$D_0^h$;
this is the same trivialization as the trivialization~of 
\BE{LaRcomp_e3}  h^*c^{|h|*}u_0^*(V,\ti\phi)\big|_{(\prt\Si)_i}
=h^*u_0^*\phi^{|h|*}(V,\ti\phi)\big|_{(\prt\Si)_i}
=u_1^*(V,\ti\phi)\big|_{(\prt\Si)_i} \EE
induced by the pull-back from the orienting collection for~$D_0$.
If this trivialization agrees with
the $\gm$-transferred trivialization  up to homotopy, 
let $\eps_j\!=\!0$; otherwise, let $\eps_j\!=\!1$.
The real bundle $\gm_{h;1}^*(V,\ti\phi)|_{\cT_{h;1}}$ is then isomorphic~to 
$$\bI\!\times\!\bigsqcup_{j=1}^k\!(\prt\Si)_{i_j}\!\times\!\C^n\big/\!\!\sim, 
\qquad (1,z,v)\sim\big(0,c^{|h|}(h(z)),g(z)v\big)~~
\forall\,z\!\in\!\bigsqcup_{j=1}^k\!(\prt\Si)_{i_j}\,,~
v\!\in\!\C^n\,,$$
for some $g\!:\bigsqcup_{j=1}^k\!(\prt\Si)_{i_j}\!\lra\!\GL_n\C$ with  
$g(c(z))\!=\!\ov{g(z)}$ such that $g|_{(\prt\Si)_{i_j}}$ is homotopic,
in the space of such maps, to the constant map~$\bI_n$ if and only~if
$\eps_j\!=\!0$.
By \cite[Lemma~2.3]{Teh}, $\La^{\top}_{\C}$ induces a bijection between sets of such
homotopy classes for rank~$n$ and rank~1 real bundle pairs.
By Lemma~\ref{T2KB_lmm}, rank~1 real bundle pairs over~$\cT$
are classified by the evaluation of their $\Z_2$-equivariant $w_2$-class 
on~$[\cT]^{c_{h;1}}$.
Thus,
\BE{eps2sum_e}\eps_1\!+\!\ldots\!+\!\eps_k=
\blr{\gm_{h;1}^*w_2^{\La^{\top}_{\C}\ti\phi}(\La^{\top}_{\C}V),[\cT]^{c_{h;1}}}\,.\EE
By \cite[Proposition~5.4]{GZ}, changing the choice~(OC2) for a given boundary 
component $(\prt\Si)_i$ with $|c_i|\!=\!1$ changes the orientation of 
the determinant line bundle.
Since the two choices~(OC2) for~$D_1$ obtained from the trivializing collection for~$u_0$ via~$h$
and $\gm$-transfer agree if and only~if $\eps_j\!=\!0$, 
the possible differences in these choices for $(\prt\Si)_{i_1},\ldots,(\prt\Si)_{i_k}$
contribute~\eref{eps2sum_e} to~\eref{epsConjTrar_e}.
Summing over all cycles of the action of~$h$ on the set $\{(\prt\Si)_i\!:|c_i|\!=\!1\}$
gives the second term on the right-hand side of~\eref{epsConjTrar_e}.
However, $h$ interchanges the order of the $n$-dimensional target spaces for 
the evaluation isomorphisms orienting $\det(D_i)$ as below~\eref{Dsplit_e1}.
This contributes $n$~times the sign of the permutation on $\{(\prt\Si)_i\!:|c_i|\!=\!1\}$
induced by~$h$.
\end{proof}

\subsection{Some consequences}
\label{relsgncrl_subs}

\noindent
We now combine the conclusions of Propositions~\ref{relsign_prp} and~\ref{transfsign_prp}
with some topological assumptions on~$(X,\phi)$ related to 
the orientability of moduli spaces of real maps to~$(X,\phi)$.
Corollaries~\ref{transfsign_crl0} and~\ref{transfsign_crl} below encode the orientability
of $\wt\det(D)$ over a loop in the moduli space of $J$-holomorphic maps.
They are used in Section~\ref{MSreal_subs} to establish more general
versions of Theorems~\ref{g0orient_thm}, \ref{g1orient_thm}, and~\ref{seorient_thm}.

\begin{rmk}\label{Vorient_rmk}
Corollaries~\ref{transfsign_crl0} and~\ref{transfsign_crl} also determine the signs of
the corresponding isomorphisms on $\det(D)$ whenever $V^{\ti\phi}\!\lra\!X^{\phi}$
is orientable: the difference with $\wt\det(D)$ is given by  $\sgn_h^0$ in the
case of Corollary~\ref{transfsign_crl0}
and by $\sgn_h$ in the case of  Corollary~\ref{transfsign_crl}.
\end{rmk}

\begin{crl}\label{transfsign_crl0}
Let $(X,\phi)$ be a manifold with an involution,
$(V,\ti\phi)\!\lra\!(X,\phi)$ be a real bundle pair,
and $(\Si,c)$ be a genus~$g$ oriented sh-surface so~that
$$w_2(V^{\ti\phi})\in H^2(X^{\phi};\Z_2)
\qquad\hbox{and}\qquad 
w_2^{\La^{\top}_{\C}\ti\phi}(\La^{\top}_{\C}V)\in H^2_{\phi}(X)$$
are spin classes in the sense of Definition~\ref{Klein_dfn}
if $|c|_0\!\neq\!0$ and $|c|_1\!\neq\!0$, respectively.
Suppose \hbox{$\k\!\equiv\!(k_1,\ldots,k_{|c|_0+|c|_1})$} 
is a tuple of non-negative integers, $\b$ is as in~\eref{btuple_eq},
$\gm\!\equiv\!(u_t,\x_t,\fJ_t)$ is a path in $\fB_{\k}(X,\b)^{\phi,c}\!\times\!\cJ_c$,
and $h\!:\Si\!\lra\!\Si$ is a diffeomorphism such~that
$(u_0,\x_0,\fJ_0)$ and $(u_1,\x_1,\fJ_1)$ are $h$-related.
If $V^{\ti\phi}\!\lra\!X^{\phi}$ is not orientable, assume that
$k_i\!>\!0$ for every boundary component $(\prt\Si)_i$ such that $|c_i|\!=\!0$ 
and $\lr{w_1(V^{\ti\phi}),b_i}\!=\!0$.
Denote by $D_0$ and $D_1$ the real Cauchy-Riemann operators on the bundle
pairs $u_0^*(V,\ti\phi)$ and $u_1^*(V,\ti\phi)$ induced as in Section~\ref{analys_subs},
choose an orienting collection for~$D_0$,
and take the $\gm$-transferred orienting collection for~$D_1$.
\begin{enumerate}[label=(\arabic*),leftmargin=*]
\item If $h$ is orientation-preserving, the sign of the isomorphism~\eref{wtDisom_e} 
is $(-1)^{(\rk_{\C}V)\sgn_h+\sgn_h^0}$. 

\item If $h$ is orientation-reversing, the sign of the isomorphism~\eref{wtDisom_e} 
composed with the isomorphism~\eref{conjisom_e} with $D\!=\!D_0$
is $(-1)^{\ti\eps^{\ti\phi}_{\gm,h}}$, where 
\BE{flipsign_e0}\begin{split}
\ti\eps^{\ti\phi}_{\gm,h}=
\frac{1}{2}\bigg(\!\!\lr{c_1(V),B}
+\sum_{i=1}^{|c|_0}\ti{w}_1\big(V^{\ti\phi},b_i\big)\!\!\bigg)
+\sum_{i=1}^{|c|_0}\flr{w_2(V^{\ti\phi}),b_i}
+\!\!\!\sum_{i=|c|_0+1}^{|c|_0+|c|_1}\!\!\!\!
\flr{w_2^{\La^{\top}_{\C}\ti\phi}(\La^{\top}_{\C}V),b_i}&\\
+\big(1\!-\!g+|c|_0+|c|_1+\sgn_h\big)\rk_{\C}V +\sgn_h^0+2\Z\in\Z_2&,
\end{split}\EE
\end{enumerate}
where $\flr{w,b_i}$ is as in~\eref{Torspindfn_e} if $b_i$ is a two-torsion class
and~0 otherwise.
\end{crl}

\begin{proof}
(1) Suppose $h$ is orientation-preserving.
Since $[\gm]_{h;0}$ and $[\gm]_{h;1}^c$ are sums of classes represented by tori,
while $w_2(V^{\ti\phi})$ and $w_2^{\La^{\top}_{\C}\ti\phi}(\La^{\top}_{\C}V)$ are
spin classes
$$\lr{w_2(V^{\ti\phi}),[\gm]_{h;0}}=0 \qquad\hbox{and}\qquad 
\blr{w_2^{\La^{\top}_{\C}\ti\phi}(\La^{\top}_{\C}V),[\gm]_{h;1}^c}=0.$$
The first claim now follows from~\eref{epsConjTrar_e2}.\\

\noindent
(2) Suppose $h$ is orientation-reversing.
The class $[\gm]_{h;0}$ is the sum of the classes represented~by
$$\gm_{h;0}\!:\cT\equiv\bI\!\times\!\bigsqcup_{j=1}^k\!(\prt\Si)_{i_j}\big/\!\!\sim\lra X,
\qquad (1,z)\sim\big(0,c(h(z))\big)~~\forall\,z\!\in\!\bigsqcup_{j=1}^k\!(\prt\Si)_{i_j},$$
with $\cT$ as in~\eref{cTcomp_e}. 
Since $h$ is orientation-reversing,  $b_{i_{j+1}}\!=\!-b_{i_j}$ 
for every $j\!=\!1,\ldots,k$; in particular, $b_{i_1}\!=\!(-1)^kb_{i_1}$.
If $k$ is even, then $\cT$ is a torus and~so
$$\lr{w_2(V^{\ti\phi}),\gm_{h;0*}[\cT]_{\Z_2}}=0=\sum_{j=1}^k\flr{w_2(V^{\ti\phi}),b_{i_j}}\,.$$
If $k$ is odd, then $b_{i_j}\!=\!b_{i_1}$ is a two-torsion class for all $j\!=\!1,\ldots,k$
and $\cT$ is a Klein bottle. Thus,
$$\lr{w_2(V^{\ti\phi}),\gm_{h;0*}[\cT]_{\Z_2}}=\flr{w_2(V^{\ti\phi}),b_{i_1}}
=\sum_{j=1}^k\flr{w_2(V^{\ti\phi}),b_{i_j}}\,;$$
the first equality holds by~\eref{Torspindfn_e}.
Putting these two cases together, we find~that 
$$\lr{w_2(V^{\ti\phi}),[\gm]_{h;0}}=\sum_{i=1}^{|c|_0}\flr{w_2(V^{\ti\phi}),b_i}\,.$$
By the same argument, 
$$\blr{w_2^{\La^{\top}_{\C}\ti\phi}(\La^{\top}_{\C}V),[\gm]_{h;1}^c}
=\!\sum_{i=|c|_0+1}^{|c|_0+|c|_1}\!\!\!\!
\flr{w_2^{\La^{\top}_{\C}\ti\phi}(\La^{\top}_{\C}V),b_i}\,.$$
Combining the last two equations with~\eref{epsConjTrar_e2}, we find that 
$$\ti\eps^{\ti\phi}_{\gm,h}
=\eps^{\ti\phi}_{\gm,h}+\eps_{D_0}
=\sum_{i=1}^{|c|_0}\flr{w_2(V^{\ti\phi}),b_i}+
\!\sum_{i=|c|_0+1}^{|c|_0+|c|_1}\!\!\!\!
\flr{w_2^{\La^{\top}_{\C}\ti\phi}(\La^{\top}_{\C}V),b_i}
+(\rk_{\C}V)\sgn_h+\sgn_h^0+\eps_{D_0}\,,$$
where $\eps_{D_0}$ is the relative sign of the conjugation on~$D_0$ computed
by~\eref{epsD_e}.
Along with~\eref{Maslovdfn_e}, this establishes the second claim.
\end{proof}

\begin{crl}\label{transfsign_crl}
Let $(X,\phi)$ be a manifold with an involution and
$(V,\ti\phi)\!\lra\!(X,\phi)$ be a real bundle pair such~that
$$w_2(V^{\ti\phi})=w+\vp|_{X^{\phi}}$$ 
for a spin class $w\!\in\!H^2(X^{\phi};\Z_2)$ and some $\vp\!\in\!H^2(X;\Z_2)$. 
Suppose
$\Si$ is a genus~$g$ oriented surface with $m$ ordered boundary components,
$\k\!\equiv\!(k_1,\ldots,k_m)$ is a tuple of non-negative integers,
$\b$ is as in~\eref{btupldfn_e},
$\gm\!\equiv\!(u_t,\x_t,\fJ_t)$ is a path in 
$\fB_{\k}(X,\b)^{\phi,\id_{\prt\Si}}\!\times\!\cJ_{\Si}$,
and \hbox{$h\!:\Si\!\lra\!\Si$} is a diffeomorphism such~that
$(u_0,\x_0,\fJ_0)$ and $(u_1,\x_1,\fJ_1)$ are $h$-related.
If $V^{\ti\phi}\!\lra\!X^{\phi}$ is not orientable, assume that
$k_i\!>\!0$ for every boundary component $(\prt\Si)_i$ such that
$\lr{w_1(V^{\ti\phi}),b_i}\!=\!0$.
Denote by $D_0$ and $D_1$ the real Cauchy-Riemann operators on the bundle
pairs $u_0^*(V,\ti\phi)$ and $u_1^*(V,\ti\phi)$ induced as in Section~\ref{analys_subs},
choose an orienting collection for~$D_0$, and take the 
$\gm$-transferred orienting collection for~$D_1$.
\begin{enumerate}[label=(\arabic*),leftmargin=*]
\item If $h$ is orientation-preserving, the sign of the isomorphism~\eref{wtDisom_e} 
is $(-1)^{(\rk_{\C}V)\sgn_h+\sgn_h}$. 

\item If $h$ is orientation-reversing, the sign of the isomorphism~\eref{wtDisom_e} 
composed with the isomorphism~\eref{conjisom_e} with $D\!=\!D_0$
is $(-1)^{\ti\eps^{\ti\phi}_{\gm,h}}$, where 
\BE{flipsign_e}\begin{split}
\ti\eps^{\ti\phi}_{\gm,h}=
\frac{1}{2}\bigg(\lr{c_1(V),\fd(\be)}
+\sum_{i=1}^m\ti{w}_1\big(V^{\ti\phi},b_i\big)\bigg)
+\blr{\vp,\fd(\be)}+\sum_{i=1}^m\flr{w,b_i}&\\
+\big(1\!-\!g+m+\sgn_h\big)\rk_{\C}V +\sgn_h&+2\Z\in\Z_2,
\end{split}\EE
\end{enumerate}
where $\flr{w,b_i}$ is as in~\eref{Torspindfn_e} if $b_i$ is a two-torsion class
and~0 otherwise.
\end{crl}

\begin{proof}
By our assumptions on $\Si$, $|c|_0\!=\!m$, $|c|_1\!=\!0$, $\sgn_h^0\!=\!\sgn_h$,
and $[\gm]_{h;1}^c\!=\!0$ in~\eref{epsD_e} and~\eref{epsConjTrar_e}.
Let 
\begin{alignat*}{2}
&Z_h=\bI\!\times\!\Si/\!\!\sim,&\qquad 
&(1,z)\sim\big(0,h(z)\big)~~\forall\,z\!\in\!\prt\Si, \\
&F_{\gm}\!: Z_h\lra X, &\qquad 
&F_{\gm}\big([s,z]\big)=u_s(z)~~\forall\,(s,z)\!\in\!\bI\!\times\!\Si\,.
\end{alignat*}
Thus, $F_{\gm}$ determines a three-chain in $X$ with coefficients in~$\Z_2$ and
boundary
$$\prt F_{\gm}=u_0\!\sqcup_h\!u_1+\gm_{h;0}.$$
In particular, for any $\vp\!\in\!H^2(X;\Z_2)$,
$$\blr{\vp,[\gm]_{h;0}}=
\blr{\vp,[u_0\!\sqcup_h\!u_1]_{\Z_2}}\,.$$
Since $w_2(V^{\ti\phi})=w+\vp|_{X^{\phi}}$, \eref{epsConjTrar_e} becomes 
\BE{epsConjTrar_e2}
\eps^{\ti\phi}_{\gm,h}=\lr{w,[\gm]_{h;0}}+\blr{\vp,[u_0\!\sqcup_h\!u_1]_{\Z_2}}
+(\rk_{\C}V)\sgn_h+\sgn_h\,.\EE
Below we analyze the first two terms on the right-hand side of~\eref{epsConjTrar_e2}
in the two cases separately.\\

\noindent
(1) Suppose $h$ is orientation-preserving.
Since $[\gm]_{h;0}$ is a sum of classes represented by tori and 
$w$ is a spin class, $\lr{w,[\gm]_{h;0}}\!=\!0$.
Since $u_1\!=\!u_0\!\circ\!h$ in this case,  $[u_0\!\sqcup_h\!u_1]_{\Z_2}\!=\!0$.
Thus, the first claim follows from~\eref{epsConjTrar_e2}.\\

\noindent
(2) Suppose $h$ is orientation-reversing.
By the same argument as in~(2) of the proof of Corollary~\ref{transfsign_crl0},
$$\lr{w,[\gm]_{h;0}}=\sum_{i=1}^m\flr{w,b_i}\,.$$
Since  $[u_0\!\sqcup_h\!u_1]_{\Z}\!=\!\fd(\be)$ in this case,
$$\blr{\vp,[u_0\!\sqcup_h\!u_1]_{\Z_2}}=\blr{\vp,\fd(\be)}.$$
Combining the last two equations with~\eref{epsConjTrar_e2},
we find that 
$$\ti\eps^{\ti\phi}_{\gm,h}
=\eps^{\ti\phi}_{\gm,h}+\eps_{D_0}
=\sum_{i=1}^m\flr{w,b_i}+\blr{\vp,\fd(\be)}
+(\rk_{\C}V)\sgn_h+\sgn_h+\eps_{D_0}\,,$$
where $\eps_{D_0}$ is the relative sign of the conjugation on~$D_0$ computed
by~\eref{epsD_e}. 
Along with~\eref{Maslovdfn_e}, this establishes the second claim.
\end{proof}

\begin{rmk}\label{Sol_rmk}
If $h$ is orientation-reversing and $w\!=\!\ka^2$ for some 
$\ka\!\in\!H^1(X^{\phi};\Z_2)$, then \eref{flipsign_e} becomes
\BE{flipsign_e2}\begin{split}
\ti\eps^{\ti\phi}_{\gm,h}=
\frac{1}{2}\bigg(\lr{c_1(V),\fd(\be)}
+\sum_{i=1}^m\ti{w}_1\big(V^{\ti{c}},b_i\big)\bigg)
+\blr{\vp,\fd(\be)}+\sum_{i=1}^m\flr{\ka,b_i}&\\
+\big(1\!-\!g+m+\sgn_hV\big)\rk_{\C}V +\sgn_h&+2\Z\in\Z_2;
\end{split}\EE
see~\eref{flrka_e}.
The two sign formulas of \cite[Proposition~5.1]{Sol} are special cases of
this formula:
\hbox{$(V,\ti\phi)\!=\!(TX,\tnd\phi)$}, $\fJ_1\!=\!\fJ_0$, and either
$\ka\!=\!0$ or $\ka\!=\!w_1(TX^{\phi})$,
i.e.~the Lagrangian $X^{\phi}$ is either relatively pin$^+$
or relatively~pin$^-$, respectively.
In turn, \cite[Proposition~5.1]{Sol} includes the orientability results
of \cite{FOOO9, Melissa}. 
\end{rmk}

\section{Applications}
\label{appl_sec}

\noindent
Let $(X,\om)$ be a symplectic manifold, $L\!\subset\!X$ be a Lagrangian submanifold,
$\be\!\in\!H_2(X,L;\Z)$,  and $J$ be an $\om$-compatible almost complex structure.
As shown in \cite[Section~8.1.1]{FOOO}, the moduli space 
$$\fM_{D^2}(X,L,\be;J)\equiv 
\big\{u\!\in\!C^{\i}(D^2,X)\!:\,\bp_{J,\fJ_0}u\!=\!0,
\,u(S^1)\!\subset\!L,\, u_*[D^2,S^1]=\be\big\}/\sim$$
of holomorphic maps from the disk $D^2$ with the standard complex structure~$\fJ_0$
is orientable
if ($L$ is orientable and) the pair $(X,L)$ admits a relative spin structure;
furthermore, every such structure canonically determines 
an orientation on~$\fM_{D^2}(X,L,\be;J)$.
This fundamental observation is extended in \cite[Theorem~1.1]{Sol} to  
arbitrary bordered Riemann surfaces~$(\Si,\fJ_0)$, with a fixed complex structure on
the domain,
and a relative pin$^{\pm}$ structure on the pair~$(X,L)$, if there is one.
In this case, the moduli space
$$\fM_{\Si}'(X,L,\be;J)\equiv 
\big\{u\!\in\!C^{\i}(\Si,X)\!:\,\bp_{J,\fJ_0}u\!=\!0,\,u(\prt\Si)\!\subset\!L,\,
u_*[\Si,\prt\Si]=\be\big\}$$
need not be orientable, but its determinant line bundle is isomorphic to 
a tensor product of the pull-backs of $\La_{\R}^{\top}(TL)$ by evaluation maps at boundary
marked points.
The pair~$(X,L)$ admits a relative pin$^{\pm}$ structure if and only~if
\BE{pincod_e}w_2(L)=\vp|_L \qquad\hbox{or}\qquad w_2(TL)=w_1(L)^2+\vp|_L\EE
for some $\vp\!\in\!H^2(X;\Z_2)$.
As observed in \cite[Remark~1.9]{Ge}, there is nothing special about~$w_1(L)$
in~\eref{pincod_e} for the purposes of the above isomorphism statement.\\

\noindent
In Sections~\ref{SpinSubStr_subs} and~\ref{MSopen_subs}, 
we introduce generalizations of the notions of 
(relative) spin/pin structures  which better capture the orientability 
in open Gromov-Witten theory.
For a symplectic manifold~$(X,\om)$ with an anti-symplectic involution~$\phi$,
we introduce a notion of compatibility of such structures with the involution
that better captures the relevant features than a similar notion introduced 
in~\cite{FOOO9}; see Sections~\ref{MSreal_subs} and~\ref{FOOO_subs} for details.

\subsection{Spin sub-structures}
\label{SpinSubStr_subs}

\noindent
We recall that a \sf{spin structure} on a real oriented vector bundle $W\!\lra\!M$
of $\rk_{\R}W\!\ge\!3$ is a collection of homotopy classes of trivializations of 
$\al^*W\!\lra\!S^1$, one class for each loop $\al\!:S^1\!\lra\!M$, such that for 
every bordered surface~$\Si$ and continuous map $f\!:\Si\!\lra\!M$ 
the vector bundle $f^*W\!\lra\!\Si$ admits a trivialization
restricting to the chosen homotopy classes of trivializations on the boundary 
components of~$\Si$ under their identification with~$S^1$.
A \sf{spin structure} on a real oriented vector bundle $W\!\lra\!M$
of $\rk_{\R}W\!=\!2$ is a spin structure on $W\!\oplus\!M\!\times\!\R\!\lra\!M$.
An oriented vector bundle $W\!\lra\!M$ admits a spin structure if and only if 
the oriented vector bundle
\BE{Wstab_e}f^*\!W\oplus\Si\!\times\!\R\lra\Si\EE
is trivial for every closed surface~$\Si$ and continuous map $f\!:\Si\!\lra\!M$,
i.e.~$w_2(W)\!=\!0$.
However, \cite[Theorem~1.1]{Ge} implies that the triviality of~\eref{Wstab_e} 
matters for the orientability problem only when $\Si\!=\!\bT$ is the two-torus,
and so the notion of spin structure
is too restrictive from this point of view whenever $M$ is not simply connected.
As it is generally difficult to describe the atorical classes,
we focus on the larger collection of spin classes in $H^2(M;\Z_2)$;
see Definition~\ref{Klein_dfn}.
The next lemma describes a geometric realization of such classes.

\begin{lmm}\label{spinclass_lmm}
Let $M$ be a paracompact topological space.
For every spin class $w\!\in\!H^2(M;\Z_2)$, there exists a complex line bundle
$W\!\lra\!M$ such that $w_2(W)\!=\!w$.
\end{lmm}

\begin{proof}
By the proof of the Universal Coefficient Theorem for Cohomology \cite[Theorem~53.1]{Mu2},
there is a commutative diagram of short exact sequences
$$\xymatrix{0\ar[r]& \Ext\big(H_1(M;\Z),\Z\big)\ar[r]\ar[d]& 
H^2(M;\Z) \ar[r]\ar[d]& \Hom\big(H_2(M;\Z),\Z\big) \ar[r]\ar[d]&0\\
0\ar[r]& \Ext\big(H_1(M;\Z),\Z_2\big)\ar[r]& 
H^2(M;\Z_2) \ar[r]& \Hom\big(H_2(M;\Z),\Z_2\big) \ar[r]&0\,,}$$
with the vertical arrows induced by the nonzero homomorphism $\Z\!\lra\!\Z_2$.
By \hbox{\cite[Exercise~52.4]{Mu2}},
the first vertical arrow in the above diagram is surjective.
It follows that every element of $\Ext(H_1(M;\Z),\Z_2)$ comes from $H^2(M;\Z)$
and thus equals $w_2(W)$ for some complex line bundle $W\!\lra\!M$.
\end{proof}

\noindent
For example, if $w\!=\!\ka^2$ for some $\ka\!\in\!H^1(M;\Z_2)$, then
$w\!=\!w_2(W'\!\oplus\!W')$ for some real line bundle $W'\!\lra\!M$,
since these line bundles correspond to elements of $H^1(M;\Z_2)$.
By Lemma~\ref{sqvan_lmm}, quite often every spin class is a square,
in which case every spin class is of the form
$w_2(W'\!\oplus\!W')\!=\!w_1(W')^2$ for some real line bundle $W'\!\lra\!M$.\\

\noindent
Every oriented real vector bundle $W\!\lra\!S^1$ admits a trivialization.
If $\rk\,W\!\ge\!3$, there are two homotopy classes of trivializations,
since $\pi_1(\SO(k))\!=\!\Z_2$ for $k\!\ge\!3$.
If $W\!\lra\!S^1$ is the restriction of a vector bundle over~$D^2$,
there is a unique homotopy class of trivializations of  $W\!\lra\!S^1$
that extend to a trivialization over~$D^2$;
this is used in the proof of Lemma~\ref{spinsubstr_lmm} below.

\begin{dfn}\label{subspinbnd_dfn}
Let  $W\!\lra\!M$ be an oriented real  vector bundle.
\begin{enumerate}[label=(\arabic*),leftmargin=*]
\item\label{spindfn_it} If $\rk_{\R}W\!\ge\!3$, a \sf{spin sub-structure} on $W\!\lra\!M$ 
is a collection of homotopy classes of trivializations of $\al^*W\!\lra\!S^1$,
one class for each loop $\al\!:S^1\!\lra\!M$, such that 
for~every 
\begin{enumerate}[label=$\bullet$,leftmargin=*]
\item connected oriented surface~$\Si$ with boundary components $(\prt\Si)_1$
and $(\prt\Si)_2$, 
\item orientation-reversing diffeomorphism $\vph_1\!:S^1\!\lra\!(\prt\Si)_1$,
\item orientation-preserving diffeomorphism $\vph_2\!:S^1\!\lra\!(\prt\Si)_2$, and
\item continuous map $F\!:\Si\!\lra\!M$, 
\end{enumerate}
the vector bundle $F^*W\!\lra\!\Si$ admits a trivialization such that its pull-backs
by $\vph_1$ and $\vph_2$ agree with the chosen homotopy classes of trivializations 
for the loops $F\!\circ\!\vph_1,F\!\circ\!\vph_2\!:S^1\!\lra\!M$.
\item If $\rk_{\R}W\!=\!2$, a \sf{spin sub-structure} on $W\!\lra\!M$ is 
a spin sub-structure on $W\!\oplus\!M\!\times\!\R\!\lra\!M$.
\end{enumerate}
\end{dfn}

\begin{lmm}\label{spinsubstr_lmm}
An oriented real vector bundle $W\!\lra\!M$ admits a spin sub-structure
if and only if  $w_2(W)$ is a spin class.
\end{lmm}

\begin{proof} We can assume that $\rk\,W\!\ge\!3$ and $M$ is connected.\\

\noindent
(1) Suppose we have chosen a spin sub-structure $W\!\lra\!M$.
Let $F_0\!:D^2\!\lra\!M$ be a fixed continuous map (e.g.~a constant map).
By reversing the chosen homotopy classes of trivializations
for all null-homologous loops if necessary (reversing for either all or for none of them),
we can assume that the restriction of a trivialization of $F_0^*W\!\lra\!D^2$ to~$S^1$ 
agrees with the chosen homotopy class of trivializations 
for the loop $F_0|_{S^1}\!:S^1\!\lra\!M$.
If $F_1\!:D^2\!\lra\!M$ is another continuous map, there exists a continuous map
$F\!:\bI\!\times\!S^1\!\lra\!M$ such that the~map
\begin{gather*}
f=F_0\!\sqcup\!F\!\sqcup\!F_1\!: 
S^2\approx \big(0\!\times\!D^2\sqcup \bI\!\times\!S^1\sqcup 1\!\times\!D^2\big)\big/\!\!\sim
\lra M, \\
0\!\times\!D^2\ni(0,z)\sim(0,z)\in\bI\!\times\!S^1,~
0\!\times\!D^2\ni(1,z)\sim(1,z)\in\bI\!\times\!S^1~~\forall\,z\!\in\!S^1\subset D^2,
\end{gather*}
is continuous and homotopically trivial.\footnote{Project $[1/4,3/4]\!\times\!S^1$ 
onto a path $[1/4,3/4]\!\lra\!M$ between $F_0(0)$ and $F_1(0)$;
then take $F$ on $[0,1/4]\!\times\!S^1$ and $[3/4,1]\!\times\!S^1$
to be homotopies from $F_0|_{S^1}$ to the constant loop at~$F_0(0)$
and from the constant loop at~$F_1(0)$ to $F_1|_{S^1}$, respectively.}
Thus, $f^*W\!\lra\!S^2$ admits a trivialization that restricts to 
the chosen homotopy class of trivializations for the loop~$F_0|_{S^1}$.
So the chosen homotopy class of trivializations over a contractible loop consists of 
trivializations that extend over a bounding disk.\\

\noindent
If $f\!:\Si\!\lra\!M$ is any continuous map from a closed oriented surface,
the restriction of $f^*W$ to the complement of two disjoint disks~$D_1^2$ and~$D_2^2$
admits a trivialization that restricts to the chosen homotopy classes of
trivializations on~$\prt D_1^2$ and~$\prt D_2^2$.
Since both of the latter extend over the corresponding disks, 
it follows that $f^*W\!\lra\!\Si$ is a trivial vector bundle and~so
$$\blr{w_2(W),f_*[\Si]_{\Z_2}}=\blr{w_2(f^*W),[\Si]_{\Z_2}}=0.$$
Thus, $w_2(W)$ is a spin class.\\

\noindent
(2) A spin sub-structure on $W\!\lra\!M$ can be obtained as follows.
Pick a representative $\al\!:S^1\!\lra\!M$ for each element~$[\al]$ of $H_1(M;\Z)$
and a homotopy class of trivializations of $\al^*W$.
Given a loop $\be\!:S^1\!\lra\!M$ with $[\be]\!=\![\al]$ in~$H_1(M;\Z)$, choose 
\begin{enumerate}[label=$\bullet$,leftmargin=*]
\item a connected oriented surface~$\Si$ with boundary components $(\prt\Si)_1$
and $(\prt\Si)_2$, 
\item an orientation-reversing diffeomorphism $\vph_1\!:S^1\!\lra\!(\prt\Si)_1$,
\item an orientation-preserving diffeomorphism $\vph_2\!:S^1\!\lra\!(\prt\Si)_2$, and
\item a continuous map $F\!:\Si\!\lra\!M$, 
\end{enumerate}
so that $F\!\circ\!\vph_1\!=\!\al$ and $F\!\circ\!\vph_2\!=\!\be$.
Since an oriented vector bundle over a connected surface with at least one boundary
component is trivial, there exists a trivialization of $f^*W\!\lra\!\Si$
extending the chosen trivialization over~$(\prt\Si)_1$. 
We take the homotopy class of the trivializations for $\be^*W\!\lra\!S^1$ to be
the homotopy class of the restriction of this trivialization to~$(\prt\Si)_2$.\\

\noindent
Given another collection $(\Si',\vph_1',\vph_2',F')$ as above, let 
$$\br\Si=(\Si\sqcup\bar\Si')/\!\!\sim,\quad z\sim\vph_1'(\vph_1^{-1}(z))~\forall\,
z\!\in\!(\prt\Si)_1,\qquad
\hat\Si=\br\Si/\!\!\sim,\quad z\sim\vph_2'(\vph_2^{-1}(z))~\forall\,
z\!\in\!(\prt\Si)_2\,,$$
where $\bar\Si'$ denotes $\Si'$ with the opposite orientation.
The maps~$F$ and~$F'$ induce continuous maps $\br{F}\!:\br\Si\!\lra\!M$
and $\hat{F}\!:\hat\Si\!\lra\!M$.
A trivialization of $\al^*W$ in the chosen homotopy class
extends to a trivialization of $\br{F}^*W\!\lra\!\br\Si$ and
$$\hat{F}^*W\approx \br\Si\!\times\!\R^k\big/\!\sim, \qquad
(z,v)\sim\big(\vph_2'(\vph_2^{-1}(z)),g(z)v\big)~\forall\,(z,v)\in(\prt\Si)_2\!\times\!\R^k\,,$$
for some $g\!:(\prt\Si)_2\!\lra\!\SO(k)$.
Since oriented vector bundles over $\hat\Si$ with rank at least~3 are classified
by their~$w_2$ and $\pi_1(\SO(k))\!\approx\!\Z_2$ for $k\!\ge\!3$,
$g$ is homotopically trivial if and only if $w_2(\hat{F}^*W)\!=\!0$.
By our assumption on~$w_2(W)$,
$$\blr{w_2(\hat{F}^*W),[\hat\Si]_{\Z_2}}=\blr{w_2(W),\hat{F}_*[\hat\Si]_{\Z_2}}=0.$$
Thus, $g$ is homotopically trivial and so the trivializations of $\be^*W$ induced
via~$F$ and~$F'$ are the same.\\

\noindent
Suppose $(\Si,\vph_1,\vph_2,F)$ is as in Definition~\ref{subspinbnd_dfn}\ref{spindfn_it},
$\al$ is the chosen representative for the homology class
$$[F\circ\vph_1]_{\Z}=[F\circ\vph_2]_{\Z}\in H_1(M;\Z),$$
and $(\Si',\vph_1',\vph_2',F')$ is as in the construction of the induced trivialization
for $\be\!\equiv\!F\!\circ\!\vph_1$, replacing $(\Si,\vph_1,\vph_2,F)$.
Let 
$$\br{F}\!=\!F'\!\sqcup\!F\!:\, \br\Si=(\Si'\!\sqcup\!\Si)/\!\!\sim\lra M, 
\qquad z\sim\vph_2'(\vph_1^{-1}(z))~\forall\,z\!\in\!(\prt\Si)_1.$$
The chosen trivialization of $\al^*W$ extends to a trivialization of~$\br{F}^*W$.
By construction, the restrictions of the latter to~$(\prt\Si)_1$ and~$(\prt\Si)_2$
lie in the chosen homotopy classes of trivializations for the loops 
$F\!\circ\!\vph_1$ and $F\!\circ\!\vph_2$.
Thus, trivializations of $\{F\!\circ\!\vph_1\}^*W$ and $\{F\!\circ\!\vph_2\}^*W$
in the chosen homotopy classes extend to  a trivialization of~$F^*W$,
as required.
\end{proof}

\noindent
A spin sub-structure can thus be viewed as a weak spin structure for oriented vector bundles
with $w_2$ lying in the subgroup $\Ext(H_1(M;\Z),\Z_2)$ of $H^2(M;\Z_2)$, 
instead of being~0;
this extension group is non-trivial if $H_1(M;\Z)$ has even torsion.
The group of~{\it{maps}}  
$$\vt\!: H_1(M;\Z)\lra \Z_2$$
acts freely and transitively on the set of spin sub-structures on $W\!\lra\!M$,
if this set is non-empty; $\vt$ changes the chosen homotopy class of trivializations 
along a loop $\al\!:S^1\!\lra\!M$ if and only if $\vt(\al)\!\neq\!0$.
In the case of spin structures, the same role is played~by
$$H^1(M;\Z_2)=\Hom(H_1(M;\Z),\Z_2),$$
the group of~{\it{homomorphisms}} from $H_1(M;\Z)$ to~$\Z_2$.\\

\noindent
Let $\fc_{S^1}\!:S^1\!\lra\!S^1$ denote the restriction of the standard conjugation on
$S^1\!\subset\!\C$.
If $\al\!:S^1\!\lra\!M$, let 
$$\al_{\fc}=\al\!\circ\!\fc_{S^1}\!:S^1 \lra M\,.$$
For any vector bundle $W\!\lra\!M$, the diffeomorphism~$\fc_{S^1}$ induces 
the commutative diagram
\BE{conjtriv_e}\begin{split}
\xymatrix{ \al^*W\ar[r]^{\ti\fc_{S^1}}\ar[d]& \al_{\fc}^*W\ar[d]&&
(z,w)\ar[r]\ar[d]& (\fc_{S^1}(z),w)\ar[d]& \forall~w\!\in\!W_{\al(z)}\\
S^1\ar[r]^{\fc_{S^1}}&S^1&& z\ar[r]& \fc_{S^1}(z)& \forall\,z\!\in\!S^1\,.}
\end{split}\EE
If $[\al]\!\in\!H_1(M;\Z)$ is a two-torsion class, $[\al_c]\!=\![\al]$.
The next lemma describes the action of the above commutative diagram on 
homotopy classes of trivializations.

\begin{lmm}\label{subspinch_lmm}
Let $W\!\lra\!M$ be an oriented real vector bundle with a chosen spin sub-structure
and $\al\!:S^1\!\lra\!M$ be a representative of a two-torsion element of~$H_1(M;\Z)$. 
The commutative diagram~\eref{conjtriv_e} takes the chosen homotopy class of trivializations
of~$\al^*W$ to the chosen homotopy class of trivializations of~$\al_{\fc}^*W$
if and only if $\flr{w_2(W),[\al]_{\Z_2}}\!=\!0$,
with $\flr{w_2(W),[\al]_{\Z_2}}$ defined as in~\eref{Torspindfn_e}.
In particular, if $w_2(W)\!=\!\ka^2$ for some $\ka\!\in\!H^1(M;\Z_2)$,
the commutative diagram~\eref{conjtriv_e} respects a spin sub-structure
on~$W$ if and only if $\lr{\ka,[\al]_{\Z_2}}\!=\!0$.
\end{lmm}

\begin{proof}
Since $\al$ represents a two-torsion element of $H_1(M;\Z)$, there exist
\begin{enumerate}[label=$\bullet$,leftmargin=*]
\item a connected oriented surface~$\Si$ with boundary components $(\prt\Si)_1$
and $(\prt\Si)_2$, 
\item an orientation-reversing diffeomorphism $\vph_1\!:S^1\!\lra\!(\prt\Si)_1$,
\item an orientation-preserving diffeomorphism $\vph_2\!:S^1\!\lra\!(\prt\Si)_2$, and
\item a continuous map $F\!:\Si\!\lra\!M$, 
\end{enumerate}
such that $F\!\circ\!\vph_1\!=\!\al$ and $F\!\circ\!\vph_2\!=\!\al_{\fc}$.
The map~$F$ descends to a continuous~map
$$\hat{F}\!:\hat\Si=\Si/\!\!\sim\lra M, \qquad 
z\sim\vph_2\big(\fc_{S^1}(\vph_1^{-1}(z))\big)~\forall\,z\!\in\!(\prt\Si)_1\,.$$
Since the diffeomorphism $\vph_2\!\circ\!\fc_{S^1}\!\circ\!\vph_1^{-1}\!:(\prt\Si)_1\!\lra\!(\prt\Si)_2$
is orientation-preserving, the closed compact surface~$\hat\Si$ is not orientable.
We note~that 
$$\hat{F}^*W=F^*W/\!\!\sim, \qquad v\sim
\ti\fc_{S^1}(v)~\forall\,v\!\in\!\al^*W.$$
Since $\pi_1(\SO(k))\!=\!\Z_2$ for $k\!\ge\!3$ and there is an oriented vector bundle
over~$\hat\Si$ with a nonzero~$w_2$, 
the oriented vector bundles over $\hat\Si$ of rank at least~3 are classified by 
their~$w_2$.
Since the commutative diagram~\eref{conjtriv_e} respects the homotopy classes
of the restrictions of a trivialization of $F^*W\!\lra\!\Si$ 
to~$(\prt\Si)_1$ and~$(\prt\Si)_2$ if and only~if 
$\hat{F}^*W\!\lra\!\hat\Si$ is a trivial vector bundle, 
it follows that this is the case if and only~if
$$0=\blr{w_2(\hat{F}^*W),[\Si]_{\Z_2}}=\blr{w_2(W),\hat{F}_*[\Si]_{\Z_2}}
\equiv\flr{w_2(W),[\al]_{\Z_2}}.$$
The last claim of the lemma now follows from~\eref{flrka_e}.
\end{proof}

\subsection{Moduli spaces of open maps}
\label{MSopen_subs}

\noindent
We now define a relative version of the spin sub-structure of Definition~\ref{subspinbnd_dfn}
and show that it canonically induces orientations of the determinant lines of 
real Cauchy-Riemann operators in open Gromov-Witten theory.

\begin{dfn}\label{quasispin_dfn}
Let $L$ be a submanifold of a manifold~$X$.
\begin{enumerate}[label=(\arabic*),ref=\arabic*,leftmargin=*]
\item\label{str_it} A \sf{relative spin sub-structure} on a real oriented vector bundle 
$F\!\lra\!L$ consists~of
\begin{enumerate}[label=(\ref{str_it}\alph*),leftmargin=*]
\item a spin sub-structure on a real oriented bundle $W\!\lra\!L$, 
\item a real oriented vector bundle $E\!\lra\!X_{[3]}$, where $X_{[3]}$ is the 3-skeleton
of~$X$ in a triangulation extending a triangulation of~$L$, and
\item a spin structure on $F\!\oplus\!W\!\oplus\!E \!\lra\!L\!\cap\!X_{[3]}$.
\end{enumerate}
\item A \sf{relative pin sub-structure} on a real vector bundle 
$F\!\lra\!L$ is a relative spin sub-structure on the oriented vector bundle 
$F\!\oplus\!3\La_{\R}^{\top}F$.
\item  A \sf{relative spin/pin sub-structure} on $(X,L)$ 
is a relative spin/pin sub-structure on the real vector bundle $TL\!\lra\!L$.
\item The pair $(X,L)$ is \sf{relatively subspin/subpin}
if $(X,L)$ admits a relative spin/pin structure.
\end{enumerate}
\end{dfn}

\noindent
A relative spin structure on $F\!\lra\!L$ in the sense of \cite[Definition~3.1.1]{FOOO}
is a relative spin sub-structure on $F\!\lra\!L$ with $W\!=\!L\!\times\!\{0\}$.
The next corollary gives a purely cohomological criterion for the existence of 
a spin/pin sub-structures.

\begin{crl}\label{spinsubstr_crl}
Let $L$ be a submanifold of a manifold~$X$ and $F\!\lra\!L$ be a real vector bundle.
\begin{enumerate}[label=(\arabic*),leftmargin=*]
\item If $F$ admits a  relative spin or pin sub-structure, then
$w_2(F)=w+\vp|_L$ for some spin class $w\!\in\!H^2(L;\Z_2)$ and for some
$\vp\!\in\!H^2(X;\Z_2)$.
\item If $w_2(F)=w+\vp|_L$ for some spin class $w\!\in\!H^2(L;\Z_2)$ and for some $\vp\!\in\!H^2(X;\Z_2)$, 
then $F$ admits a pin sub-structure.
If in addition $F$ is orientable, then $F$ admits a relative spin sub-structure.
\end{enumerate}
\end{crl}

\begin{proof}
(1) If $F$ admits a  relative spin sub-structure as in Definition~\ref{quasispin_dfn},
the vector bundle $F\!\oplus\!W\!\oplus\!E$ over $L\!\cap\!X_{[3]}$ 
admits a spin structure and~so
\begin{gather*}
0=w_2\big((F\!\oplus\!W\!\oplus\!E)|_{L\cap X_{[3]}}\big)
=w_2(F)|_{L\cap X_{[3]}}+w_2(W)|_{L\cap X_{[3]}}+w_2(E)|_{L\cap X_{[3]}}\\
\Lra\qquad w_2(F)=w_2(W)+\vp|_L
\end{gather*}
for some $\vp\!\in\!H^2(X;\Z_2)$.
Since $W$ admits a spin sub-structure, $w_2(W)$ is a spin class by 
Lemma~\ref{spinsubstr_lmm}.
The same reasoning, with $F$ replaced by $F\!\oplus\!3\La_{\R}^{\top}F$,
applies in the pin case.\\

\noindent
(2) It is sufficient to consider the orientable case.
By Lemma~\ref{spinclass_lmm}, $w\!=\!w_2(W)$ for some oriented vector bundle $W\!\lra\!L$.
By Lemma~\ref{spinsubstr_lmm}, $W$ admits a spin sub-structure.
By the usual obstruction theory reasoning, there exists an oriented rank~3
vector bundle $E\!\lra\!X_{[3]}$ such that 
$w_2(E)\!=\!\vp|_{X_{[3]}}$.\footnote{There is
a continuous map $f\!:X\!\lra\!K(\Z_2,2)$ such that $\vp\!=\!f^*\Om$,
where $K(\Z_2,2)$ is the Eilenberg-MacLane space with $\pi_2\!=\!\Z_2$ and
$\Om$ is the generator of $H^2(K(\Z_2,2);\Z_2)$.
It can be assumed that $f$ takes $X_{[3]}$ to the 3-skeleton $K(\Z_2,2)_{[3]}$
of~$K(\Z_2,2)$.
Since $K(\Z_2,2)$ and $\bB\SO(3)$ are simply-connected with $\pi_2\!=\!\Z_2$,
there is a continuous map $F\!: K(\Z_2,2)_{[3]}\!\lra\!\bB\SO(3)$ inducing
an isomorphism on~$\pi_2$ and thus on the second $\Z$-homology.
This implies that $\Om\!=\!F^*w_2(\gm_3)$, where $\gm_3\!\lra\!\bB\SO(3)$ is 
the tautological oriented rank~3 vector bundle and so 
$\vp|_{X_{[3]}}\!=\!w_2(f^*F^*\gm_3)$.\label{ObsRea_ftnt}}
Since 
\begin{equation*}\begin{split}
w_2\big((F\!\oplus\!W\!\oplus\!E)|_{L\cap X_{[3]}}\big)
&=w_2(F)|_{L\cap X_{[3]}}+w_2(W)|_{L\cap X_{[3]}}+w_2(E)|_{L\cap X_{[3]}}\\
&=w_2(F)|_{L\cap X_{[3]}}+w|_{L\cap X_{[3]}}+\vp|_{L\cap X_{[3]}} =0,
\end{split}\end{equation*}
the vector bundle $F\!\oplus\!W\!\oplus\!E$ over $L\!\cap\!X_{[3]}$ 
admits a spin structure.
\end{proof}

\noindent
Let $X$ be a smooth manifold, $L\!\subset\!X$ be a smooth submanifold,
and $\Si$ be a compact bordered surface,
with ordered boundary components $(\prt\Si)_1,\ldots,(\prt\Si)_m$.
Given
\BE{btuple_eq2}
\b=(\be,b_1,\ldots,b_m)
\in H_2(X,L;\Z)\oplus H_1(L;\Z)^m\,,\EE
we define
$$\fB_{\Si}(X,L,\b)=\big\{u\!\in\!C^{\i}(\Si,X)\!:\,u(\prt\Si)\!\subset\!L,\,
 u_*[\Si,\prt\Si]=\be,~u_*[(\prt\Si)_i]=b_i~\forall\,i\!=\!1,\ldots,m\big\}.$$
If in addition $\k\!=\!(k_1,\ldots,k_m)$ is a tuple of nonnegative integers,  let
$$\fB_{\Si;\k}(X,L,\b)=\fB_{\Si}(X,L,\b)\times
\prod_{i=1}^m\!\!\big((\prt\Si)_i^{k_i}-\De_{i,k_i}\big),$$
where 
$$\De_{i,k_i}=\big\{(x_{i,1},\ldots,x_{i,k_i})\!\in\!(\prt\Si)_i^{k_i}\!:~
x_{i,j'}\!=\!x_{i,j}~\tn{for some}~j,j'\!=\!1\ldots,k_i,~j\!\neq\!j'\big\}$$
is the big diagonal.
Let
\begin{equation*}\begin{split}
\cH_{\Si,\k}(X,L,\b)&=
\big(\fB_{\Si,\k}(X,L,\b)\!\times\!\cJ_{\Si}\big)\big/\cD_{\Si},\\
\mf_{\Si,\k}(X,L,\b;J)&=
\big\{[u,\x_1,\ldots,\x_m,\fJ]\!\in\!\cH_{\Si,\k}(X,L,\b)\!:~
\dbar_{J,\fJ}u\!=\!0\big\},
\end{split}\end{equation*}
if $J$ is an almost complex structure on~$X$.
For each $i\!=\!1,\ldots,m$ and $j\!=\!1,\ldots,k_i$, we define 
$$\ev_{i,j}\!:\,\cH_{\Si,\k}(X,L,\b)\lra L$$
to be the evaluation map at the $j$-th marked point of the $i$-th boundary component.

\begin{prp}\label{SpinOrient_prp}
Suppose $X$ is a smooth manifold, $L\!\subset\!X$ is a smooth submanifold,
$V\!\lra\!X$ is a complex vector bundle, $V_{\R}\!\subset\!V|_L$
is a totally real subbundle, $\Si$ is a compact oriented surface
 with ordered boundary components $(\prt\Si)_1,\ldots,(\prt\Si)_m$,
and $x_i\!\in\!(\prt\Si)_i$ for each $i\!=\!1,\ldots,m$.

\begin{enumerate}[label=(\arabic*),leftmargin=*]
\item For every complex structure $\fJ$ on~$\Si$ and 
every map $u\!:(\Si,\prt\Si)\!\lra\!(X,L)$,
a relatively pin sub-structure on $V_{\R}\!\lra\!L$
canonically induces an orientation on the twisted determinant line 
$$\wt\det(D_u)= \det(D_u)\otimes\hspace{-.3in}
\bigotimes_{\lr{u^*w_1(V_{\R}),(\prt\Si)_i}=0}
\hspace{-.3in}\big(\La_{\R}^{\top}V_{\R}|_{u(x_i)}\big)^*$$
of a real Cauchy-Riemann operator~$D_u$ on the bundle pair
$(u^*V,u|_{\prt\Si}^*V_{\R})\lra(\Si,\prt\Si)$. 

\item Let $\b$ be as in~\eref{btuple_eq2} and
$\k\!\equiv\!(k_1,\ldots,k_m)$ be a tuple of non-negative integers
such that $k_i\!>\!0$ for each~$i$ with $\lr{w_1(V_{\R}),b_i}\!=\!0$.
If $V_{\R}\!\lra\!L$ is relatively subpin, the twisted determinant line bundle 
$$\wt\det(D_{V,V_{\R}})\lra \cH_{\Si,\k}(X,L,\b)$$
induced by $(V,V_{\R})$ as in \cite[Remark~1.3]{Ge}, 
with $x_i\!=\!x_{i;1}$, is orientable.

\item  Let $\b$ be as in~\eref{btuple_eq2} and
$\k\!\equiv\!(k_1,\ldots,k_m)$ be a tuple of non-negative integers.
If $V_{\R}\!\lra\!L$ is relatively  subspin, the determinant line bundle 
$$\det(D_{V,V_{\R}})\lra \cH_{\Si,\k}(X,L,\b)$$
is orientable. 
\end{enumerate}
\end{prp}
 
\begin{proof}
The proof of \cite[Theorem~1.1]{Ge} reduces the non-orientable case to 
the orientable case, with the difference accounted for by twisting the
determinant.
Thus, we can assume that $V_{\R}\!\lra\!L$ is orientable.
By the proof of \cite[Proposition~8.1.4]{FOOO}, 
an orientation on $\det(D_u)$ is canonically determined by 
a choice of homotopy classes of trivializations of $u_0|_{(\prt\Si)_i}^*V_{\R}$
for each boundary component $(\prt\Si)_i$ of~$\Si$ for a map 
$$u_0\!:(\Si,\prt\Si)\lra (X_{[3]},X_{[3]}\cap L).$$
Given a relatively spin sub-structure on $F\!=\!V_{\R}$ as in Definition~\ref{quasispin_dfn},
a homotopy class of trivializations 
of $u_0|_{(\prt\Si)_i}^*V_{\R}$ is canonically determined~by
\begin{enumerate}[label=$\bullet$,leftmargin=*]
\item the homotopy class of trivializations of $u_0|_{(\prt\Si)_i}^*W$ provided by the spin sub-structure on~$W$ and
\item a homotopy class of trivializations of $u_0|_{(\prt\Si)_i}^*E$.
\end{enumerate}
We choose a collection of homotopy classes of trivializations of $u_0|_{(\prt\Si)_i}^*E$,
one for each boundary component, so that it is the restriction of a trivialization
of $u^*E\!\lra\!\Si$.
In the case $\Si\!=\!D^2$ considered in \cite[Proposition~8.1.4]{FOOO},
there is only one such collection.
In general, any two such collections differ by changing the chosen homotopy class
of trivializations on an even number of boundary components.
Since changing the homotopy class of trivializations along a single boundary component
changes the orientation of~$\det(D_u)$, 
any two collections of trivializations of $u_0|_{(\prt\Si)_i}^*E$ that come from
a trivialization of $u_0^*E\!\lra\!\Si$ induce the same orientation on~$\det(D_u)$;
see the proof of \cite[Proposition~3.1]{Sol} for a different presentation of
this point.
This establishes~(1), which immediately implies~(2) and~(3).
\end{proof}

\begin{crl}\label{SpinOrient_crl}
Suppose $(X,\om)$ is a symplectic manifold, $L\!\subset\!X$ is a Lagrangian submanifold,
$J\!\in\!\cJ_{\om}$, $\Si$ is a compact oriented surface
with ordered boundary components $(\prt\Si)_1,\ldots,(\prt\Si)_m$,
$\b$ is as in~\eref{btuple_eq2}, and $\k\!\equiv\!(k_1,\ldots,k_m)$ 
is a tuple of non-negative integers.

\begin{enumerate}[label=(\arabic*),leftmargin=*]
\item If $L$ is orientable and admits a relative spin sub-structure,
the moduli space $\mf_{\Si,\k}(X,\b;J)$ is orientable.
Furthermore, a choice of such a structure canonically determines an orientation
on $\mf_{\Si,\k}(X,\b;J)$.

\item If $k_i\!>\!0$ for each boundary component $(\prt\Si)_i$ with
$\lr{u^*w_1(TL),(\prt\Si)_i}\!=\!0$ and 
$L$ admits a relative pin sub-structure,
the orientation line bundle of $\mf_{\Si,\k}(X,\b;J)$ is isomorphic~to 
$$\bigotimes_{\lr{u^*w_1(TL),(\prt\Si)_i}=0}
\ev_{i;1}^*(\La_{\R}^{\top}TL).$$
Furthermore, a choice of such a structure determines such an isomorphism
up to homotopy.
\end{enumerate}
\end{crl}
 
\begin{proof}
Both statements follow from Proposition~\ref{SpinOrient_prp};
see the proof of \cite[Corollary~1.8]{Ge}. 
\end{proof}

\subsection{Moduli spaces of real maps}
\label{MSreal_subs}

\noindent
Let $(X,\om,\phi)$ be a symplectic manifold with an anti-symplectic involution
and $J$ be an $\om$-compatible almost complex structure on~$X$ such that 
$\phi^*J\!=\!-J$.
Every real map from a symmetric surface~$(\hat\Si,\si)$ to~$(X,\phi)$
can be represented by a map from an oriented sh-surface $(\Si,c)$,
but not uniquely in general.
This gives rise to coverings of moduli spaces of the former by 
moduli spaces of the latter.
These coverings are regular if the genus of~$\hat\Si$ is~0 or~1
or if $\hat\Si\!-\!\hat\Si^{\si}$ is disconnected;
there are 5 topological types of symmetric surfaces of genus~0 or~1.
Orientability of moduli spaces of maps from oriented sh-surfaces is addressed
in~\cite{GZ}.
Propositions~\ref{relsign_prp} and~\ref{transfsign_prp} determine
when the deck transformations of these coverings are orientation-preserving.
We apply them in this section to establish more general versions
of Theorems~\ref{g0orient_thm}, \ref{g1orient_thm}, and~\ref{seorient_thm}.\\

\noindent
If $(\Si,c)$ is an oriented sh-surface and~$\b$ is as in~\eref{btuple_eq}, 
denote by $\cP(\b)$ the set of~tuples obtained from the tuples
$$\b=\big(B,b_1,\ldots,b_{|c|_0},b_{|c|_0+1},\ldots,b_{|c|_0+|c|_1}\big)
\quad\hbox{and}\quad
\bar\b=\big(B,-b_1,\ldots,-b_{|c|_0},-b_{|c|_0+1},\ldots,-b_{|c|_0+|c|_1}\big)$$
by permuting the $b_1,\ldots,b_{|c|_0}$-entries and 
the $b_{|c|_0+1},\ldots,b_{|c|_0+|c|_1}$-entries (within each of the two sets).
If $|c|_1\!=\!0$, i.e.~$(\Si,c)$ is a bordered surface without crosscaps,
and~$\b$ is as in~\eref{btupldfn_e}, 
denote by $\cP(\b)$ the set of~tuples obtained from the tuples
$$\b=\big(\be,b_1,\ldots,b_{|c|_0}\big)
\qquad\hbox{and}\qquad
\bar\b=\big(-\phi_*\be,-b_1,\ldots,-b_{|c|_0}\big)$$
by permuting the $b_1,\ldots,b_{|c|_0}$-entries.
We define
$$\fM^{\cup}(X,\b;J)^{\phi,c}=\bigcup_{\b'\in\cP(\b)}\!\!\!\!\!\fM(X,\b';J)^{\phi,c},
\qquad
\fM_{\Si}^{\cup}(X,X^{\phi},\b;J)=\bigcup_{\b'\in\cP(\b)}\!\!\!\!\!
\fM_{\Si}(X,X^{\phi},\b';J)$$
in the two cases, respectively.
If $h\!:\Si\!\lra\!\Si$ is a diffeomorphism commuting with~$c$ on~$\prt\Si$, 
similarly to \eref{DMhdfn_e} we define
\BE{fMhdfn_e}\begin{aligned}
\fM_h\!: \fM^{\cup}(X,\b;J)^{\phi,c}&\lra \fM^{\cup}(X,\b;J)^{\phi,c},&\qquad 
[u,\fJ]&\lra \big[\phi^{|h|}\!\circ\!u\!\circ\!h,(-1)^{|h|}h^*\fJ\big],\\
\fM_h\!: \fM_{\Si}^{\cup}(X,X^{\phi},\b;J)&\lra \fM_{\Si}^{\cup}(X,X^{\phi},\b;J),&\qquad 
[u,\fJ]&\lra \big[\phi^{|h|}\!\circ\!u\!\circ\!h,(-1)^{|h|}h^*\fJ\big],
\end{aligned}\EE
with the notation as in~\eref{cHcMdfn_e}.
We will call these automorphisms \sf{the natural automorphisms~of} 
$\fM^{\cup}(X,\b;J)^{\phi,c}$ and $\fM_{\Si}^{\cup}(X,X^{\phi},\b;J)$, respectively.

\begin{crl}\label{seorient_crl}
Let $(X,\om)$ be a symplectic $2n$-manifold with an anti-symplectic involution~$\phi$,
\hbox{$J\!\in\!\cJ_{\phi}$}, $\Si$ be a genus~$g$ oriented bordered surface
with $m$ boundary components,
and $\b$ be as in~\eref{btupldfn_e}.
If either $2b_i\!=\!0$ for all~$i$ and $m\!-\!g\!\in\!2\Z$
or $b_i\!=\!\pm b_j$ for some $i\!\neq\!j$,
assume also that $n$ is odd. 
If $X^{\phi}\!\subset\!X$ is orientable and there exist 
a spin class $w\!\in\!H^2(X^{\phi};\Z_2)$ 
and a class $\vp\!\in\!H^2(X;\Z_2)$ such~that 
\BE{realorient_e0}
w_2(TX^{\phi})=w+\vp|_{X^{\phi}} 
\qquad\hbox{and}\qquad 
\frac12\lr{c_1(TX),\fd(\be)}+\lr{\vp,\fd(\be)}
+\sum_{i=1}^m\flr{w,b_i}\in 2\Z\,,\EE
then the natural automorphisms of $\fM_{\Si}^{\cup}(X,X^{\phi},\b;J)$
are orientation-preserving with respect to the orientation induced 
by some relative spin sub-structure associated with~$w$ and~$\vp$ as in Definition~\ref{quasispin_dfn}.
If $w\!=\!0$, this is the case for the orientation induced by every relative spin structure
associated with~$\vp$.
\end{crl}

\begin{proof}
By Corollary~\ref{SpinOrient_crl}, the moduli space $\fM_{\Si}(X,X^{\phi},\b;J)$ is orientable
under our assumptions, and a relative spin sub-structure determines an 
orientation.
We consider three cases separately.\\

\noindent
(1) Suppose first that $\Si\!=\!D^2$.
Let $\fJ_0$ be the standard complex structure on the disk~and
\begin{equation*}\begin{split}
\wt\fM_{D^2}^{\cup}(X,X^{\phi},\b;J)&=
\big\{u\!\in\!\fB(X)^{\phi,\id_{S^1}}\!:\,\dbar_{J,\fJ_0}u\!=\!0,~
u_*[D^2,S^1]\!=\!\be,~u_*[S^1]\!=\!b_1\big\}\\
&\qquad \cup\big\{u\!\in\!\fB(X)^{\phi,\id_{S^1}}\!:\,\dbar_{J,\fJ_0}u\!=\!0,~
u_*[D^2,S^1]\!=\!-\phi_*\be,~u_*[S^1]\!=\!-b_1\big\}.
\end{split}\end{equation*}
Since $\fM_{D^2}^{\cup}(X,X^{\phi},\b;J)=\wt\fM_{D^2}^{\cup}(X,X^{\phi},\b;J)/\PGL_2^0\R$,
there is a canonical isomorphism
$$\det(D_{TX,\tnd\phi})=
\La^{\top}_{\R}\big(T\wt\fM_{D^2}^{\cup}(X,X^{\phi},\b;J)\big)\approx
\La^{\top}_{\R}(T_{\id}\PGL_2^0\R)\otimes
\La^{\top}_{\R}\big(T\fM_{D^2}^{\cup}(X,X^{\phi},\b;J)\big)\,,$$
with $\det(D_{TX,\tnd\phi})$ as at the end of Section~2.1.
In this case, \eref{conjaut_e2} is the only automorphism of $\fM_{D^2}^{\cup}(X,X^{\phi},\b;J)$
to consider, since all (orientation-preserving) automorphisms of~$D^2$ are isotopic to the identity.
By Corollary~\ref{transfsign_crl}(2) and Remark~\ref{Vorient_rmk}, 
the action of this automorphism on $\det(D_{TX,\tnd\phi})$ at 
$u\!\in\!\wt\fM_{D^2}^{\cup}(X,X^{\phi},\b;J)$ is orientation-preserving under 
our assumptions~\eref{realorient_e0} 
if the $J$-holomorphic maps $u$ and $\phi\!\circ\!u\circ\!\fc_{D^2}$
are homotopic (as continuous maps).
Its action on $T\PGL_2^0\R$ is orientation-preserving as well,
since there is a canonical isomorphism
$$\La^{\top}_{\R}(T_{\id}\PGL_2^0\R)\approx
T_1S^1\otimes \La^{\top}_{\R}(T_0D^2)$$
and the automorphism~\eref{conjaut_e2} reverses the orientations of both factors.
This establishes both claims of Corollary~\ref{seorient_crl}  
at the elements~$[u]$ of $\fM_{D^2}^{\cup}(X,X^{\phi},\b;J)$ such that 
$u$ and $\phi\!\circ\!u\circ\!\fc_{D^2}$ are homotopic.\\

\noindent
Suppose the restrictions of $u$ and $v\!\equiv\!\phi\!\circ\!u\circ\!\fc_{D^2}$
to $S^1\!\subset\!D^2$ are homologous, 
but the maps $u$ and~$v$ are not necessarily homotopic.
Suppose also that the lines $\det(D_{TX,\tnd\phi})$ at~$u$ and~$v$ are
oriented by a spin sub-structure on~$TX^{\phi}$ as in the proof of 
Proposition~\ref{SpinOrient_prp}.
In the terminology of the proof of Proposition~\ref{transfsign_prp},
the trivializations of $TX^{\phi}\!\oplus\!W\!\oplus\!E$ at~$v|_{S^1}$
transferred by a cobordism and pulled back from the trivialization of this bundle at~$u|_{S^1}$
are the trivializations given by the spin structure on this bundle.
By Lemma~\ref{subspinch_lmm}, the difference between the two trivializations
of~$v|_{S^1}^*W$ is given~by
$$\flr{w_2(W),[u|_{S^1}]_{\Z_2}}=\flr{w,b_1}\,.$$
The difference between the trivialization of $v^*E$ in the second bullet
in the proof of Propositions~\ref{SpinOrient_prp} and the trivialization
pulled back from the corresponding trivialization of~$u^*E$ is given~by
$$\blr{w_2(E),[u\!\sqcup_{\fc_{D^2}}\!v]_{\Z_2}}=\blr{\vp,\fd(\be)}.$$
Along with Proposition~\ref{relsign_prp}, this implies that 
the sign of the action of the automorphism~\eref{conjaut_e2} on 
$\det(D_{TX,\tnd\phi})$ at~$u$ is
still given by the left-hand side of the second expression in~\eref{realorient_e0}.
Thus, the first claim of Corollary~\ref{seorient_crl} holds
in this case as~well.\\

\noindent
If the restrictions of $u$ and $v\!\equiv\!\phi\!\circ\!u\circ\!\fc_{D^2}$ to $S^1\!\subset\!D^2$ 
are not homologous, we can simply choose a spin sub-structure on~$TX^{\phi}$,
starting with a trivialization of $TX^{\phi}\!\oplus\!W\!\oplus\!E$ at~$u|_{S^1}$,
so~that the action of the automorphism~\eref{conjaut_e2} on $\det(D_{TX,\tnd\phi})$ at~$u$ 
is orientation-preserving.
So, the first claim of Corollary~\ref{seorient_crl} holds again.
If $w\!=\!0$, the $W\!=\!0$ case of the discussion in the previous paragraph
(without transfers) applied to any spin structure on $TX^{\phi}\!\oplus\!E$
establishes the last claim of Corollary~\ref{seorient_crl}.\\

\noindent
(2) Suppose next that $\Si$ is a cylinder with ordered boundary components. Let 
\begin{equation*}\begin{split}
\wt\fM_{\Si}(X,X^{\phi},\b;J)&=
\big\{(u,r)\!\in\!\fB(X)^{\phi,\id_{\prt\Si}}\!\times\!\oI:\,
\dbar_{J,\fJ_r}u\!=\!0,~u_*[\Si,\prt\Si]\!=\!\be,\\
&\hspace{1in}
~u_*[(\prt\Si)_1]\!=\!b_1,~u_*[(\prt\Si)_2]\!=\!b_2\big\},\\
\wt\fM_{\Si}^{\cup}(X,X^{\phi},\b;J)&=
\bigcup_{\b'\in\cP(\b)}\!\!\!\!\!\wt\fM_{\Si}(X,X^{\phi},\b';J),
\end{split}\end{equation*}
with $\fJ_r\!\in\!\cJ_{\Si}$ as in Section~\ref{analys_subs}.
Since $\fM_{\Si}^{\cup}(X,X^{\phi},\b;J)=\wt\fM_{\Si}^{\cup}(X,X^{\phi},\b;J)/S^1$,
there is a canonical isomorphism
$$\det(D_{TX,\tnd\phi})=
\La^{\top}_{\R}\big(T\wt\fM_{\Si}^{\cup}(X,X^{\phi},\b;J)\big)\approx
\La^{\top}_{\R}(T_1S^1)\otimes
\La^{\top}_{\R}\big(T\fM_{\Si}^{\cup}(X,X^{\phi},\b;J)\big)\,.$$
Since every diffeomorphism of $\Si$ preserving the orientation and the boundary components
is isotopic to the identity,
it is sufficient to consider the automorphisms induced by
the diffeomorphisms~$h_{\Si}$ and~$\fc_{\Si}$ defined in Section~\ref{analys_subs}
and their composite.
By Corollary~\ref{transfsign_crl} and Remark~\ref{Vorient_rmk}, 
the actions of the first two automorphisms
on $\det(D_{TX,\tnd\phi})$ at  $u\!\in\!\wt\fM_{D^2}^{\cup}(X,X^{\phi},\b;J)$ 
are orientation-reversing under our assumptions~\eref{realorient_e0}
if~$u$ and its images under these automorphisms are homotopic.
Their actions on~$S^1$ are also orientation-reversing,
since they are induced by the~maps
$$S^1\lra S^1, \qquad z\lra 1/z=\bar{z}\,.$$
Thus, these automorphisms preserve an orientation on~$\fM_{\Si}^{\cup}(X,X^{\phi},\b;J)$.
The same is the case of the composite automorphism.
This establishes both claims of Corollary~\ref{seorient_crl} at the elements~$[u]$ 
of $\wt\fM_{\Si}^{\cup}(X,X^{\phi},\b;J)$ homotopic to their images
under each automorphism.
The remaining cases are handled as in the disk case above.\\

\noindent
(3) Suppose $\Si$ is not a disk or a cylinder.
The forgetful morphism
$${\mathfrak f}\!: \fM_{\Si}^{\cup}(X,X^{\phi},\b;J)\lra\cM_{\Si}$$
canonically induces an isomorphism
$$\La^{\top}_{\R}\big(T\fM_{\Si}^{\cup}(X,X^{\phi},\b;J)\big)\approx
\det(D_{TX,\tnd\phi})\otimes {\mathfrak f}^*\La^{\top}_{\R}\big(T\cM_{\Si}\big).$$
The sign of the action on $\La^{\top}_{\R}(T\cM_{\Si})$
induced by a diffeomorphism $h\!:\Si\!\lra\!\Si$ is described~by Proposition~\ref{DMsgn_prp}.
The corresponding sign for $\det(D_{TX,\tnd\phi})$ is given by 
Corollary~\ref{transfsign_crl} without the extra~$\sgn_h$ term
in the homotopic cases;
see Remark~\ref{Vorient_rmk}.
Under our assumptions~\eref{realorient_e0}, the two signs are again the same.
The non-homotopic cases are treated as in~(1) above.
\end{proof}

\noindent
By \cite[Lemma~2.3]{Teh}, every rank~$n$ real bundle pair $(V,\ti{c})\!\lra\!(S^1,\fa)$
is trivial, i.e.~there is a vector bundle isomorphism
$\Psi\!:V\!\lra\!S^1\!\times\!\C^n$ covering the identity on~$S^1$ such that 
$$\Psi\circ c =\{\id\!\times\fc_{\C^n}\}\circ\Psi,$$
where $\fc_{\C^n}$ is the standard conjugation on~$\C^n$.
Furthermore, there are two homotopy classes of such real trivializations and 
they correspond to the two homotopy classes of real trivializations of 
$\La_{\C}^{\top}(V,\ti{c})$.
In similarity with Definition~\ref{subspinbnd_dfn}, 
we define a \sf{spin sub-structure on a real bundle pair} 
$(W,\ti\phi)\!\lra\!(X,\phi)$ to be a collection of homotopy classes of 
trivializations of real bundle pairs $\al^*(W,\ti\phi)\!\lra\!(S^1,\fa)$,
one class for each real loop $\al\!:(S^1,\fa)\!\lra\!(X,\phi)$, such that 
for~every real map 
$$F\!:(\bI\!\times\!S^1,\id_{\bI}\!\times\!\fa)\!\lra\!(X,\phi),
\qquad (s,z)\lra F_s(z),$$
a trivialization of the real bundle pair 
$F^*(W,\ti\phi)\!\lra\!(\bI\!\times\!S^1,\id_{\bI}\!\times\!\fa)$
restricts to a trivialization  of the real bundle pair 
$F_s^*(W,\ti\phi)\!\lra\!(S^1,\fa)$ in the chosen homotopy class 
for each $s\!\in\!\bI$.

\begin{crl}\label{etaorient_crl}
Let $(X,\om)$ be a symplectic $2n$-manifold with an anti-symplectic involution~$\phi$,
$(\Si,c)$ be a genus~$g$ oriented sh-surface, 
$J\!\in\!\cJ_{\phi}$, and $\b$ be as in~\eref{btuple_eq}.
We also assume~that 
\begin{enumerate}[label=$\bullet$,leftmargin=*]
\item $n$ is odd if either 
$2b_i\!=\!0$ for all~$i$ and $|c|_0\!+\!|c|_1\!-\!g\!\in\!2\Z$
or $b_i\!=\!\pm b_j$ for some $i\!\neq\!j$;
\item $X^{\phi}$ is orientable and $w_2(TX^{\phi})\!\in\!H^2(X^{\phi};\Z_2)$
is a spin class if $|c|_0\!\neq\!0$;
\item $w_2^{\La^{\top}_{\C}\tnd\phi}(\La^{\top}_{\C}TX)\!\in\!H^2_{\phi}(X)$
is a spin class if $|c|_1\!\neq\!0$.
\end{enumerate}
If
\BE{realorient_e1}
\frac12\lr{c_1(TX),B}
+\sum_{i=1}^{|c|_0}\flr{w_2(TX),b_i}
+\sum_{i=|c|_0+1}^{|c|_0+|c|_1}\!\!\!\!\!\flr{w_2^{\La^{\top}_{\C}\tnd\phi}(\La^{\top}_{\C}TX),b_i}
\in 2\Z\,,\EE
then the natural automorphisms of $\fM^{\cup}(X,\b;J)^{\phi,c}$
are orientation-preserving with respect to the orientation induced 
by some spin sub-structure on~$TX^{\phi}$  
and some spin sub-structure on~$(TX,\tnd\phi)$;
the former is needed only if $|c|_0\!\neq\!0$, while the latter is needed only if  
$|c|_1\!\neq\!0$.
If $\Si\!=\!D^2$ and $c\!\neq\!\id_{S^1}$, 
the condition~\eref{realorient_e1} can be dropped.
\end{crl}

\begin{proof}
By the same argument as in the proof of Lemma~\ref{spinsubstr_lmm}, 
$(TX,\tnd\phi)$ admits a spin sub-structure if 
$w_2^{\La^{\top}_{\C}\tnd\phi}(\La^{\top}_{\C}TX)$ is a spin class.
By \cite[Corollary~6.2]{GZ}, $\fM^{\cup}(X,\b;J)^{\phi,c}$ is orientable under 
our assumptions.
By the proofs of Propositions~\ref{SpinOrient_prp} and \cite[Corollary~6.2]{GZ}, 
an orientation on $\fM^{\cup}(X,\b;J)^{\phi,c}$ is induced by a spin sub-structure on~$TX^{\phi}$
and a spin sub-structure on~$(TX,\tnd\phi)$;
the former is needed only if $|c|_0\!\neq\!0$, while the latter is needed only if  
$|c|_1\!\neq\!0$.
By Corollary~\ref{transfsign_crl0}, Remark~\ref{Vorient_rmk}, 
Proposition~\ref{DMsgn_prp}, \cite[Lemma~6.1]{GZ},
(1) and~(2) in the proof of Corollary~\ref{seorient_crl}, 
and the treatment of non-homotopic cases in~(1) of the proof of Corollary~\ref{seorient_crl},
the sign of the action on $\det(D_{TX,\tnd\phi})$ 
induced by an orientation-reversing diffeomorphism $h\!:\Si\!\lra\!\Si$ is 
the left-hand side of~\eref{realorient_e1} plus $n\,\sgn(\DM_h)$,
where $\sgn(\DM_h)$ is the corresponding sign on $\cD_c^*$ if 
$\Si$ is a disk or a cylinder and on~$\cM_{\Si}$ otherwise;
if $h$ is orientation-preserving, the sign is just~$n\,\sgn(\DM_h)$.
Thus, the first claim of Corollary~\ref{etaorient_crl} is obtained 
by considering three cases as in the proof of Corollary~\ref{seorient_crl}.
The last claim follows from Corollary~\ref{equivSQrt_crl}.
\end{proof}

\begin{rmk}
The conclusions of Corollaries~\ref{seorient_crl} and~\ref{etaorient_crl} hold under
more general circumstances. 
In particular,
the cases with $X^{\phi}$ non-orientable can be handled, but
with additional care.
\end{rmk}

\noindent
We conclude with a short proof of an observation obtained
by a rather delicate argument in~\cite{FOOO9}.

\begin{crl}[{\cite[Proposition~3.14]{FOOO9}}]\label{FOOO9prp314_crl}
If $m\!\in\!\Z^+$, the standard anti-holomorphic involution~$\phi$ does not preserve
any relative spin structure on the pair $(\P^{4m+1},\R\P^{4m+1})$.
\end{crl}

\begin{proof}
Let $\vp\!\in\!H^2(\P^{4m+1};\Z_2)$ and 
$\be\!\in\!H^2(\P^{4m+1},\R\P^{4m+1};\Z)$ be the standard generators.
Since $\R\P^{4m+1}$ is orientable,
$$w_2(\R\P^{4m+1})=\vp|_{\R\P^{4m+1}}   \qquad\hbox{and}\qquad 
\frac12\lr{c_1(\P^{4m+1}),\fd(\be)}+\lr{\vp,\fd(\be)}=2m+2\in2\Z,$$
the automorphism~\eref{conjaut_e2} on $\fM_{D^2}(\P^{4m+1},\R\P^{4m+1},\be;J)$
is orientation-preserving by Corollary~\ref{seorient_crl}.
Since the minimal Maslov index of the pair $(\P^{4m+1},\R\P^{4m+1})$ evaluated on~$\be$,
i.e.~$4n\!+\!2$, is not divisible by~4, \cite[Theorem~1.1]{FOOO9} implies that 
no relative spin structure on the pair $(\P^{4m+1},\R\P^{4m+1})$ is preserved 
by~$\phi$. 
\end{proof}

\subsection{Floer theory}
\label{FOOO_subs}

\noindent
A number of striking implications of anti-symplectic involutions
to Floer homology are described in~\cite{FOOO9}.
In this section, we streamline some aspects of the approach in~\cite{FOOO9},
modifying one of the key notions introduced in~\cite{FOOO9}.
This allows us to extend some statements in~\cite{FOOO9} and 
significantly simplify some of the proofs, 
without altering the fundamental principles behind them.\\

\noindent
If  $(X,\om)$ is a symplectic manifold with an anti-symplectic involution~$\phi$
and $\be\!\in\!H_2(X,X^{\phi};\Z)$, let $\fd(\be)\!\in\!H_2(X;\Z)$ denote
the natural $\phi$-double of~$\be$ as before; see \cite[Section~3]{Ge2}.

\begin{dfn}\label{phirelspin_dfn}
Let $(X,\om)$ be a symplectic manifold with an anti-symplectic involution~$\phi$
such that $X^{\phi}$ is orientable.
A \sf{$\phi$-relative spin structure} on~$(X,X^{\phi})$ consists~of
\begin{enumerate}[label=(\arabic*),leftmargin=*]
\item a real oriented vector bundle \hbox{$E\!\lra\!X_{[3]}$}, 
where $X_{[3]}$ is the 3-skeleton of~$X$ in a $\phi$-invariant triangulation 
extending a triangulation of~$X^{\phi}$, such~that 
\BE{phirelspin_e}w_2(TX^{\phi})=\vp|_{X^{\phi}} 
\quad\hbox{and}\quad 
\frac12\lr{c_1(TX),\fd(\be)}+\lr{\vp,\fd(\be)}\in 2\Z
\quad\forall\,\be\in H_2(X,X^{\phi};\Z),\EE
with $\vp\!\in\!H^2(X;\Z_2)$ defined by $\vp|_{X_{[3]}}\!=\!w_2(E)$, and
\item a spin structure on $TX^{\phi}\!\oplus\!E \!\lra\!X^{\phi}\!\cap\!X_{[3]}$.
\end{enumerate} 
\end{dfn}

\noindent
By Footnote~\ref{ObsRea_ftnt}, $(X,X^{\phi})$ admits 
a $\phi$-relative spin structure if and only if there exists 
$\vp\!\in\!H^2(X;\Z_2)$ satisfying~\eref{phirelspin_e}.
We note that our notion of $\phi$-relatively spin structure is different
from that of \cite[Definition~3.11]{FOOO9}; see more below.
A natural equivalence on the set of relative spin structures is described by 
\cite[Definition~3.4]{FOOO9}; we will view two relative spin structures 
as identical if they are equivalent in this sense.\\

\noindent
As explained in \cite[Section~3.2]{FOOO9},
there are two relative spin structures on $(\P^{2n+1},\R\P^{2n+1})$.
They correspond to the same class~$\vp$, which is~0 if $n$ is odd and
the generator of $H^2(\P^{2n+1};\Z_2)$  if $n$ is even, 
and are $\phi$-relative spin in the sense of Definition~\ref{phirelspin_dfn}.
If $(Y,\om_Y)$ is a symplectic manifold, the interchange of factors 
is an anti-symplectic involution on $(Y\!\times\!Y,\pi_1^*\om_Y\!-\!\pi_2^*\om_Y)$;
the diagonal is the fixed locus.
In this case, the class $\vp\!=\!\pi_1^*c_1(TY)$ provides  
a $\phi$-relative spin structure.\\

\noindent
The significance of Definition~\ref{phirelspin_dfn} from the point of view
of the Floer-theoretic applications in~\cite{FOOO9} is described by the next corollary,
which is essentially a special case of Corollary~\ref{seorient_crl}.

\begin{crl}\label{seorient_crl2}
Let $(X,\om)$ be a symplectic manifold with an anti-symplectic involution~$\phi$,
\hbox{$J\!\in\!\cJ_{\phi}$}, and $\be\!\in\!H_2(X,X^{\phi};\Z)$.
If $X^{\phi}\!\subset\!X$ is orientable and there exists $\vp\!\in\!H^2(X;\Z_2)$ such 
that~\eref{phirelspin_e} holds, then the isomorphism 
$$\fM_{D^2}(X,X^{\phi},\be;J)\lra \fM_{D^2}(X,X^{\phi},-\phi_*\be;J)$$
given by~\eref{conjaut_e2} is orientation-preserving
with respect to the orientations induced by 
any relative spin structure associated with~$\vp$.
\end{crl}

\noindent
In \cite[Definition~3.11]{FOOO9}, a relative spin structure on $X^{\phi}\!\subset\!X$
is called \sf{$\phi$-relative} if it is preserved by the involution~$\phi$;
no simple test, like~\eref{phirelspin_e}, is provided for this property.
The motivation for \cite[Definition~3.11]{FOOO9} appears to be the mistake 
in \cite[Proposition~11.5]{FOOOold},
which misses the possibility that a relatively spin structure need not to be
preserved when pulled back by the involution~$\phi$.
The statement of \cite[Proposition~11.5]{FOOOold} is indeed valid for 
the $\phi$-relative spin structures of \cite[Definition~3.11]{FOOO9},
which follows immediately from \cite[Proposition~5.1]{Sol} 
and \cite[Theorem~1.1]{FOOO9}; the latter corrects the statement of
\cite[Proposition~11.5]{FOOOold} by taking into account the change 
of the relative spin structure under the pull-back by~$\phi$.
However, \cite[Definition~3.11]{FOOO9} does not seem ideally suited
for the remarkable applications considered in \cite{FOOO9}.
The primary idea behind these applications is to study whether 
the involution~\eref{conjaut_e2}
is orientation-preserving with respect to the orientation defined by a fixed
relative spin structure on the two sides;
the applications depend on this involution being orientation-preserving.
This involution is orientation-preserving if the relative spin structure 
is $\phi$-relative spin in the sense of \cite[Definition~3.11]{FOOO9}
{\it and} the Maslov index of $(TX,TX^{\phi})$ is divisible by~4,
but not otherwise; see \cite[Theorem~1.1]{FOOO9}.
For example, this is the case for $(\P^{2n+1},\R\P^{2n+1})$ with $n$ odd, but
not even, and for $(Y\!\times\!Y,\pi_1^*\om_Y\!-\!\pi_2^*\om_Y)$ with $c_1(TY)$
even, but not~odd.
The nature of \cite[Definition~3.11]{FOOO9} forces 
the authors to split the consideration of $\P^{2n-1}$ 
and $(Y\!\times\!Y,\pi_1^*\om_Y\!-\!\pi_2^*\om_Y)$ in \cite[Section~6.4]{FOOO9} 
and in \cite[Section~6.3.1]{FOOO9}, respectively, based on the parity of~$n$
and~$c_1(TY)$, with a careful consideration of the cases which are
not $\phi$-relatively spin.
There is no distinction between the two cases from the point of view
of Definition~\ref{phirelspin_dfn} and Corollary~\ref{seorient_crl2}.
In the same spirit, we obtain the following extensions of results in~\cite{FOOO9}.

\begin{prp}\label{FOOO9thm1.5_prp}
Let $(X,\om)$ be a compact symplectic manifold with an anti-symplectic involution~$\phi$.
\begin{enumerate}[label=(\arabic*),leftmargin=*]
\item The justification of the conclusion of \cite[Theorem~1.5]{FOOO9} applies 
also if $X^{\phi}$ is relative spin in~$X$ 
and $\eps_1$ is increased~by  $\lr{\vp,\fd(\be)}$,
with $\vp$ as in Definition~\ref{phirelspin_dfn}.
\item The justification of the conclusion of \cite[Corollary~1.6]{FOOO9} applies
also if $(X,X^{\phi})$ is $\phi$-relative spin in the sense of 
Definition~\ref{phirelspin_dfn}.  
\item The justification of the conclusion of \cite[Corollary~1.8]{FOOO9} applies
also if $c_1(X)|_{\pi_2(X)}\!=\!0$, $X^{\phi}$ is orientable, and 
$w_2(X^{\phi})\!=\!\vp|_{X^{\phi}}$ for some $\vp\!\in\!H^2(X;\Z_2)$  such~that 
$\lr{\vp,\fd(\be)}\!=\!0$ for all $\be\!\in\!\pi_2(X,X^{\phi})$.
\end{enumerate}
\end{prp}

\begin{proof}
(1) The sign $(-1)^{\eps_1}$ in \cite[Theorem~1.5]{FOOO9} is determined by 
\cite[Theorem~4.12]{FOOO9} applied with a 
$\phi$-relative spin structure in the sense of \cite[Definition~3.11]{FOOO9}; 
see the beginning of \cite[Section~6.2]{FOOO9}.
The restriction on the relative spin structure in  \cite[Theorem~1.5]{FOOO9}
ensures that the pull-back relative spin structure used to orient the moduli space
on the left-hand side of \cite[(4.10)]{FOOO9} is the same as the 
relative spin structure used to orient the moduli space
on the right-hand side of \cite[(4.10)]{FOOO9}.
For an arbitrary relative spin structure, the difference between the orientations
induced by the two relative spin structures is given by \cite[Proposition~3.10]{FOOO9}
and equals $(-1)^{\lr{\vp,\fd(\be)}}$, which needs to be combined with 
the sign of $(-1)^{\eps_1}$ in~\cite{FOOO9}.
Alternatively and more directly, the sign of the action on the unmarked moduli space of 
disks in the proof of \cite[Theorem~1.5]{FOOO9} is $(-1)^{\eps}$, where
$$\eps=\frac12\lr{c_1(TX),\fd(\be)}+\lr{\vp,\fd(\be)}\,;$$ 
see the first part of the proof of Corollary~\ref{seorient_crl},
restricting to the $w\!=\!0$ case.
The proof of \cite[Theorem~4.12]{FOOO9} then shows that the statement of 
\cite[Theorem~4.12]{FOOO9} applies to the map in \cite[(4.10)]{FOOO9}
with the moduli space on the left-hand side oriented by the same relative spin structure
as on the right-hand side, provided $\eps$ in \cite[(4.10)]{FOOO9} is increased by 
$\lr{\vp,\fd(\be)}$.
The proof of  \cite[Theorem~1.5]{FOOO9} in \cite[Section~6.2]{FOOO9}
then applies without any changes under our weaker assumptions.\\

\noindent
(2) By the first part of this proposition, the crucial identity 
$${\mathfrak m}_{0,\tau_*\be}(1)=-{\mathfrak m}_{0,\be}(1)$$
in the proof of \cite[Corollary~1.6]{FOOO9} in \cite[Section~6.2]{FOOO9}
remains valid under our weaker assumptions.
The rest of the proof in~\cite{FOOO9} applies without any changes.\\

\noindent
(3) By the second part of this proposition, 
the proof of \cite[Corollary~1.8]{FOOO9} in \cite[Section~6.2]{FOOO9}
applies without any changes under our weaker assumptions.
\end{proof}

\noindent
For example, \cite[Theorem~1.5]{FOOO9} does not apply to complete intersections
$X_{n;\a}\!\subset\!\P^{n-1}$ of dimension at least~2 
such~that 
$$n-|\a|\equiv 2\mod 4 \qquad\hbox{and}\qquad
a_1^2+\ldots+a_l^2\equiv|\a| \mod4,$$
since  
these complete intersections do not admit a $\phi$-relative spin structure
in the sense of \cite[Definition~3.11]{FOOO9}.
However, Proposition~\ref{FOOO9thm1.5_prp}(1) applies to these symplectic manifolds,
since they admit a $\phi$-relative spin structure 
in the sense of Definition~\ref{phirelspin_dfn}.

\vspace{.2in}

\noindent
{\it Department of Mathematics, Princeton University, Princeton, NJ 08544\\
pgeorgie@math.princeton.edu}\\

\noindent
{\it Department of Mathematics, SUNY Stony Brook, Stony Brook, NY 11790\\
azinger@math.sunysb.edu}\\

\end{document}